\documentclass[11pt]{article}

\usepackage{etex}
\usepackage[margin={0.8 in}]{geometry}
\usepackage[titletoc]{appendix}

\usepackage{tikz,eso-pic}
\usepackage{tikz-3dplot}
\usetikzlibrary{cd,matrix, calc, arrows,
decorations.pathreplacing, decorations.markings, decorations.pathmorphing,
backgrounds,fit,positioning,shapes.symbols,chains,shadings,fadings,intersections}
\usetikzlibrary{trees}
\newcommand{\nn}{node{$\bullet$}}

\tikzset{->-/.style={decoration={  markings,  mark=at position #1 with
    {\arrow{>}}},postaction={decorate}}}
\tikzset{-<-/.style={decoration={  markings,  mark=at position #1 with
    {\arrow{<}}},postaction={decorate}}}

\usepackage{eucal}
\usepackage{mathrsfs}
\usepackage{stmaryrd}

\usepackage{enumitem}
\makeatletter
\newcommand{\mylabel}[2]{#2\def\@currentlabel{#2}\label{#1}}
\makeatother

\usepackage{hyperref} 

\usepackage{rotating}

\usepackage{longtable}

\usepackage[sc]{mathpazo}
\usepackage{courier}

\usepackage[utf8]{inputenc}
\usepackage[T1]{fontenc}

\usepackage[english]{babel}
\usepackage{blindtext}

\usepackage{amsmath}

\usepackage{microtype}

\usepackage{float}

\usepackage{amsmath}
\usepackage{amssymb}
\usepackage{amsthm}
\usepackage{color}
\usepackage{latexsym,extarrows}

\usepackage{esint}

\usepackage[all]{xy}

\usepackage[stable,multiple]{footmisc}

\RequirePackage[font=small,format=plain,labelfont=bf,textfont=it]{caption}
\makeatletter
\newsavebox{\@brx}
\newcommand{\llangle}[1][]{\savebox{\@brx}{\(\m@th{#1\langle}\)}%
  \mathopen{\copy\@brx\kern-0.5\wd\@brx\usebox{\@brx}}}
\newcommand{\rrangle}[1][]{\savebox{\@brx}{\(\m@th{#1\rangle}\)}%
  \mathclose{\copy\@brx\kern-0.5\wd\@brx\usebox{\@brx}}}
\makeatother

\def\e#1\e{\begin{equation}#1\end{equation}}
\def\ea#1\ea{\begin{align}#1\end{align}}

\theoremstyle{plain}
\newtheorem{thm}{Theorem}[section]
\newtheorem{lem}[thm]{Lemma}
\newtheorem{lem-dfn}[thm]{Lemma/Definition}
\newtheorem{prop}[thm]{Proposition}
\newtheorem{cor}[thm]{Corollary}
\theoremstyle{definition}
\newtheorem{dfn}[thm]{Definition}
\newtheorem{ex}[thm]{Example}
\newtheorem{rem}[thm]{Remark}

\newcommand{\A}{\mathcal{A}}

\newcommand{\op}{\operatorname}
\newcommand{\C}{\mathbb{C}}

\newcommand{\R}{\mathbb{R}}

\newcommand{\Z}{\mathbb{Z}}
\renewcommand{\H}{\mathbf{H}}

\newcommand{\im}{\op{im}}

\newcommand{\abracket}[1]{\left\langle#1\right\rangle}
\newcommand{\bbracket}[1]{\left[#1\right]}
\newcommand{\fbracket}[1]{\left\{#1\right\}}
\newcommand{\bracket}[1]{\left(#1\right)}

\newcommand{\pa}{\partial}

\newcommand{\dbar}{\bar\pa}
\newcommand{\OO}{{\mathcal O}}

\newcommand{\RE}{{\mathcal{R}^{E}}}

\DeclareMathOperator{\Res}{Res}

\DeclareMathOperator{\Lie}{Lie}

\def\Xint#1{\mathchoice
{\XXint\displaystyle\textstyle{#1}}%
{\XXint\textstyle\scriptstyle{#1}}%
{\XXint\scriptstyle\scriptscriptstyle{#1}}%
{\XXint\scriptscriptstyle\scriptscriptstyle{#1}}%
\!\int}
\def\XXint#1#2#3{{\setbox0=\hbox{$#1{#2#3}{\int}$}
\vcenter{\hbox{$#2#3$}}\kern-.5\wd0}}

\def\dashint{\Xint-}

\numberwithin{equation}{section}

\makeatletter
\newcommand{\subjclass}[2][2010]{%
  \let\@oldtitle\@title%
  \gdef\@title{\@oldtitle\footnotetext{#1 \emph{Mathematics Subject Classification.} #2}}%
}
\newcommand{\keywords}[1]{%
  \let\@@oldtitle\@title%
  \gdef\@title{\@@oldtitle\footnotetext{\emph{Key words and phrases.} #1.}}%
}
\makeatother

\makeatletter
\let\orig@afterheading\@afterheading
\def\@afterheading{%
   \@afterindenttrue
  \orig@afterheading}
\makeatother

\begin{document}
\title{\mbox{\bf Regularized Integrals on Riemann Surfaces and Modular Forms}}
\author{Si Li and Jie Zhou}
\date{}
\maketitle

\begin{abstract}
We introduce a simple procedure to integrate differential forms with arbitrary holomorphic poles on Riemann surfaces. It gives rise to an intrinsic regularization of such singular integrals in terms of the underlying conformal geometry. Applied to products of Riemann surfaces, this regularization scheme establishes an analytic theory for integrals over configuration spaces, including Feynman graph integrals arising from two dimensional chiral quantum field theories. Specializing to elliptic curves, we show such regularized graph integrals  are almost-holomorphic modular forms that geometrically provide modular completions of the corresponding ordered $A$-cycle integrals. 

\end{abstract}

\setcounter{tocdepth}{2} \tableofcontents

\section{Introduction}

The present work aims to develop analytic tools for integrals on configuration spaces of Riemann surfaces arising from two dimensional chiral quantum field theories. 

\subsection*{Regularized integrals on Riemann surfaces}

Let $\Sigma$ be a closed Riemann surface without boundary.  Let $\omega$ be a $2$-form on $\Sigma$ which is smooth away from a finite subset $D\subset \Sigma$ but may admit holomorphic poles of arbitrary order\footnote{To be a little more precise,
$\omega$ is an element of 
$ \mathcal{A}^{1,1}({\Sigma}, \star D) $, see Section \ref{sec-curve-integral} for details.} 
along $D$. In this paper, we introduce the notion of \textbf{regularized integral} (Definition \ref{defn-RI})
$$
\dashint_{\Sigma}\omega
$$
as a recipe to integrate the singular form $\omega$ on $\Sigma$. This is defined as follows. 

We first decompose $\omega$ into  (Lemma \ref{lem-decomposition}) 
$$
\omega=\alpha+\pa \beta
$$
where $\alpha$ is a $2$-form with at most logarithmic pole along $D$, $\beta$ is a $(0,1)$-form with arbitrary order of poles along $D$, and $\pa$ is the holomorphic de Rham differential. Then we define
$$
\dashint_{\Sigma}\omega:= \int_{\Sigma}\alpha
$$
where the right hand side is absolutely integrable. 

Here are a few remarks. 
\begin{enumerate}
\item[(1)]  The integral $\int_{\Sigma}\alpha$ does not depend on the choice of the decomposition $\omega=\alpha+\pa \beta$ (Proposition \ref{prop-integration-noboundary}). Therefore it is reasonable to denote it by $\dashint_{\Sigma}\omega$.  
\item[(2)] The regularized integral is extended to the case when $\Sigma$ has boundary. Then
$$
\dashint_{\Sigma}\omega:=\int_{\Sigma}\alpha+ \int_{\pa \Sigma} \beta 
$$
which again does not depend on the choice of the decomposition (Theorem \ref{thm-integral}). 
\item[(3)] $\dashint_{\Sigma}$ is invariant under conformal transformations (Proposition \ref{prop-pullback}). 

\item[(4)] $\dashint_{\Sigma}$ gives an intrinsic meaning of the Cauchy principal value (Theorem \ref{thm-PV}). Such regularized integral can be generalized to the $n$-th Cartesian product $\Sigma^n$ (see discussions below), even though the meaning of Cauchy principal value is not clear there. 
\end{enumerate}

The above properties can be summarized as 
$$
\boxed{\text{the conformal geometry of $\Sigma$ gives an intrinsic regularization of the integral  $\int_\Sigma\omega$}}. 
$$
The regularized integral enjoys many other nice properties. For example,
\begin{itemize}
\item A version of Stokes Theorem holds (Theorem \ref{thm-de-Rham})
$$
\dashint_\Sigma d \alpha=-2\pi i \Res_\Sigma(\alpha)+\int_{\pa\Sigma}\alpha\,. 
$$ 
\item Riemann-Hodge  bilinear type relation holds (Proposition \ref{prop-bilinear}).  
\item A version of push-forward map exists which intertwines the holomorphic de Rham differential (Theorem \ref{thm-family}). 
\end{itemize}

\subsection*{Regularized integrals on configuration spaces}
Let $\Sigma$ be a compact Riemann surface without boundary and $\Sigma^n$ be the $n$-th Cartesian product of $\Sigma$. Let
$
\Delta_{ij}:= \{(z_1,\cdots, z_n)\in \Sigma^n| z_i=z_j\}
$
and $\Delta$ be the collection of all such diagonal divisors called the big diagonal
$$
\Delta= \bigcup_{1\leq i\neq j\leq n} \Delta_{ij}\,.
$$

Let $\omega$ now be a $2n$-form on $\Sigma^n$ which is smooth away from $\Delta$ but may admit holomorphic poles of arbitrary order along $\Delta$. Such $\omega$ defines a smooth $2n$-form on $\Sigma^n-\Delta$, which is the configuration space of $n$ points on $\Sigma$. We can decompose\footnote{
The technique of finding logarithmic representatives for top-degree cohomology classes is classical, see e.g., \cite{Leray:1959}.
} (Lemma \ref{lem-global-decomp})
$
\omega=\alpha+\pa \beta
$
where $\alpha$ has at most logarithmic pole along $\Delta$ (in the sense of Definition \ref{defn-log}) and thus is absolutely integrable on $\Sigma^n$. This allows us to  define the regularized integral
$$
\dashint_{\Sigma^n}\omega:=\int_{\Sigma^n} \alpha
$$
in the same fashion as above (Definition \ref{defn-integral-product}).

Unlike the $n=1$ case, it is not clear how to identify the above $\dashint_{\Sigma^n}\omega$ as an intrinsic Cauchy principal value in general.
Nevertheless 
it can be shown that such regularized integral is equal to the $n$-times iterated regularized integral on $\Sigma$ (Theorem \ref{thm-integral-product})
$$
\dashint_{\Sigma^n}\omega= \dashint_\Sigma \dashint_{\Sigma}\cdots \dashint_{\Sigma}\omega\,. 
$$
This can be viewed as a canonical regularization of $\int_{\Sigma^n}\omega$ via the conformal structure of $\Sigma$. 

Such $\omega$ arises in chiral quantum field theories on $\Sigma$, such as chiral bosons, chiral $\beta\gamma$-systems, chiral $bc$-systems and their  deformations. The diagonal singularities of $\omega$ come from the local behavior of propagators. The integrations on $\Sigma^n$ correspond to Feynman graph integrals. Due to the existence of diagonal singularities, the naive Feynman graph integrals on $\Sigma^n$ are problematic and require regularizations. In the particular case when $\Sigma$ is an elliptic curve,  integrals over product of disjoint $A$-cycles instead of over $\Sigma^n$ are often studied. Such $A$-cycle integrals are mathematically well-defined and are expected to capture essential aspects of the original chiral theories via the mechanism of contact terms \cite{Douglas:1995conformal, Dijkgraaf:1997chiral}. Our construction of  $\dashint_{\Sigma^n}\omega$ fills the gap by providing a geometric regularization of Feynman graph integrals in two dimensional chiral theories.  Furthermore, we build a precise link between our regularized integrals and the $A$-cycle integrals  in Theorem \ref{main-theorem} below. 


 \begin{rem}
The situation here is very different from the finiteness phenomenon of Feynman graph integrals in topological field theories considered by Kontsevich \cite{kontsevich1994feynman,kontsevich2003deformation}, Axelrod-Singer \cite{axelrod1993chern} and Getzler-Jones \cite{getzler1994operads}. There we usually have an integral
$$
\int_{X^n} \Omega
$$
which is convergent by the reason that $\Omega$ extends to a smooth form on a real version of the Fulton-MacPherson compactification of the configuration space of $n$ points on $X$. 
In our case, such extension is impossible since the naive integral $\int_{\Sigma^n}\omega$ is not absolutely convergent in general. Instead, the regularized integral $\dashint_{\Sigma^n}\omega$ can be viewed as the holomorphic counterpart in dimension two of the above construction in topological field theories. 
\end{rem}

\begin{rem}

A more algebraic approach in dealing with a singular integral of the form
$\int_{\sigma}\omega$, 
where $\sigma$ is a cycle intersecting with poles of an algebraic form $\omega$,
is to suitably interpret the integral as
the pairing between certain relative homology 
and relative cohomology. This involves studies of mixed Hodge structures on relative (co)homologies, as well as 
algebraic structures on graph complexes
and graph (co)homologies,
see for example \cite{Connes:2000renormalization, Bloch:2006motives,
Bloch:2007motives, Bloch:2008mixed, Bloch:2008motives,
Marcolli:2009feynman, Bloch:2015}.
To a large extent, this approach avoids having to integrate singular forms, but can bring in additional complicated combinatorics
and thus make exact computations difficult. This algebraic setting does not seem to apply directly to our case where $\omega$ is a smooth top-form with holomorphic singularities, and we hope to connect our analytic approach to this algebraic one in a future investigation.
\end{rem}

\subsection*{Quasi-modularity and geometric modular completion}

In this paper, we mainly apply the notion of regularized integrals to the case when 
$$
\Sigma= E_\tau=\C/\Lambda_{\tau}\,,\quad \Lambda_{\tau}:=\Z\oplus \Z \tau
$$
is an elliptic curve. Here $\tau$ is a point on  the upper-half plane $\H$. We shall show that regularized integrals lead to tremendous geometric constructions of modular objects. 

As a prototypical example, let $\Phi(z;\tau)$ be a meromorphic elliptic function on $\mathbb{C}\times \H$ which is modular of weight $k\in \mathbb{Z}$ (Definition \ref{dfnellipticityandmodularity}). Hence
$\Phi(-;\tau)$ defines a meromorphic function on $E_\tau$. Assume $\Phi(-;\tau)$ does not have residue at any pole, so 
$
 \varphi=\Phi(z;\tau)dz
$
defines a 2nd kind Abelian differential on $E_\tau$. Then the following holds (Proposition \ref{prop-modular-completion})
$$
\dashint_{E_{\tau}}{d^2z\over \im \tau}\,\Phi= \int_{A} dz\,\Phi(z;\tau) -{1\over 2 i \im \tau}\cdot  2\pi i\,\langle \varphi, dz\rangle_{\mathrm{P}}\,, \quad \quad d^2z :={i \over 2}dz\wedge d\bar z\,.
$$
Here $A$ is a representative of the $A$-cycle class that does not intersect the poles of $\Phi(-;\tau)$, and $\langle  \varphi, dz \rangle_{\mathrm{P}}$ is the Poincar\'e residue pairing. In this expression, 
$$
\dashint_{E_{\tau}}{d^2z\over \im \tau}\Phi\quad \text{is modular of weight $k$}
$$
while 
$$
 \int_{A} dz\,\Phi(z;\tau) \quad \text{is quasi-modular of weight $k$}\,.
$$
We see that the regularized integral gives the modular completion of the $A$-cycle integral. 

Such phenomenon actually occurs in great generality for integrals on configuration spaces. Let us first recall the following notion of holomorphic limit (Definition \ref{dfn-holomorphiclimit})
\begin{align*}
\lim_{\bar\tau\to \infty}: \OO_{\H}[{1\over \im \tau}]\to \OO_\H\,, 
  \qquad     f(\tau, \bar \tau)=\sum_{i=0}^N {f_i(\tau)\over (\im \tau)^i} &\to  f_0(\tau)\,. 
\end{align*}
Here the $f_i(\tau)$'s are holomorphic in $\tau$. It sends modular quantities to quasi-modular ones. See Appendix \ref{secmodularformsellipticfunctions} for the basics on modularity, quasi-modularity, and their relations. 

The main application for us in this paper is summarized in the following theorem. 

\begin{figure}[H]\centering
	\begin{tikzpicture}[scale=1]

\draw[cyan!23,fill=cyan!23](0,0)to(4,0)to(5,3)to(1,3)to(0,0);
\draw (0,0) node [below left] {$0$} to (4,0) node [below] {$1$} to (5,3) node [above right] {$1+\tau$} to (1,3) node [above left] {$\tau$} to (0,0);

\draw(1,3)coordinate(a1)(5,3)coordinate(a2) (0,0)coordinate(a3)(4,0)coordinate(a4);

\draw[ultra thick,blue,->-=.5,>=stealth]($(a3)!0.8!(a1)$)to ($(a4)!0.8!(a2)$);
\node [above] at ($($(a3)!0.8!(a1)$)!0.5!($(a4)!0.8!(a2)$)$) {$A_n$};

\draw[ultra thick,blue,->-=.5,>=stealth]($(a3)!0.6!(a1)$)to ($(a4)!0.6!(a2)$);
\node [above] at ($($(a3)!0.6!(a1)$)!0.5!($(a4)!0.6!(a2)$)$) {$A_{n-1}$};
\node [below] at ($($(a3)!0.6!(a1)$)!0.5!($(a4)!0.6!(a2)$)$) {$\vdots$};

\draw[ultra thick,blue,->-=.5,>=stealth]($(a3)!0.1!(a1)$)to ($(a4)!0.1!(a2)$);
\node [above] at ($($(a3)!0.1!(a1)$)!0.5!($(a4)!0.1!(a2)$)$) {$A_1$};

	\end{tikzpicture}
\begin{tikzpicture}[xscale=.5,yscale=.7]\clip(0,-3) rectangle (11,2);
\draw[black,very thick,fill=cyan!23](10,0).. controls +(90:2.5) and +(90:2.5) .. (1,0)
.. controls +(-90:2.5) and +(270:2.5) ..(10,0);
\draw[black,thick, fill=white](4,0)to[bend right](7,0);
\draw[black,thick, fill=white](4.2,0)to[bend left=15](6.8,0);
\draw[blue,thick,->-=.5,](6.5,-1.8)to[bend left=-75](5.8,-0.4);
\draw[blue,thick,dashed,-<-=.5,](6.5,-1.8)to[bend right=-75](5.8,-0.4);
\draw[blue,thick,->-=.5,](9,-1.2)to[bend left=-75](7,0);
\draw[blue,thick,dashed,-<-=.5,](9,-1.2)to[bend right=-75](7,0);
\draw[blue,thick,->-=.5,](4.0,-1.8)to[bend left=-75](4.2,-0.1);
\draw[blue,thick,dashed,-<-=.5,](4.0,-1.8)to[bend right=-75](4.2,-0.1);
\node [below] at (6.5,-1.8) {$A_{i}$};
\node [below] at (9,-1.4) {$A_{i+1}$};
\end{tikzpicture}
\caption{Ordered $A$-cycle representatives.}\label{fig:$A$-cycles}
\end{figure}
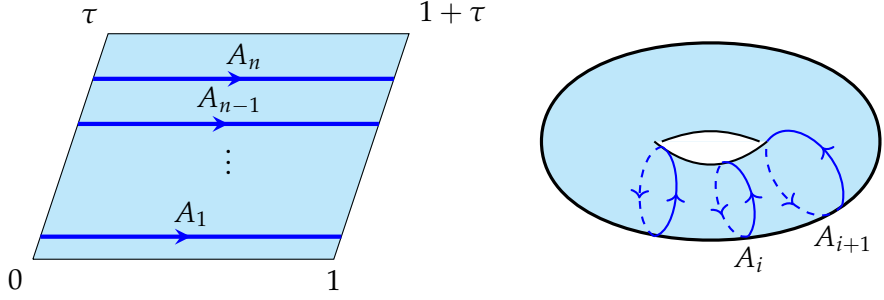

\begin{thm}[Theorem \ref{thm-2d}, Theorem \ref{thm-1d}]\label{main-theorem} Let $\Phi(z_1,\cdots, z_n;\tau)$ be a meromorphic  elliptic function on 
$\mathbb{C}^{n}\times \H$ (Definition  \ref{dfnellipticityandmodularity}) which is holomorphic away from diagonals (Definition  \ref{dfnholomorphicityawayfromdiagonals}).   Let $A_1, \cdots, A_n$ be  $n$ disjoint representatives of the $A$-cycle class on $E_\tau$ ordered as in Fig. \ref{fig:$A$-cycles}. Then
\begin{enumerate}
\item The regularized integral 
$$
 \dashint_{E_\tau^n} \bracket{\prod_{i=1}^n {d^2z_i\over \im \tau}}
  \Phi(z_1,\cdots, z_n;\tau)  \quad \text{lies in}\quad \OO_{\H}[{1\over \im \tau}]. 
$$
Its holomorphic limit is given by the average of ordered $A$-cycle integrals (Definition \ref{dfnorderedAcycleintegral})
$$
\boxed{\lim_{\bar\tau \to \infty} \dashint_{E_\tau^n} \bracket{\prod_{i=1}^n {d^2z_i\over \im \tau}} \Phi(z_1,\cdots, z_n;\tau)={1\over n!}\sum_{\sigma\in S_n} \int_{A_{1}} dz_{\sigma(1)}\cdots \int_{A_{n}} dz_{\sigma(n)} \Phi(z_1,\cdots, z_n;\tau) }. 
$$
\item If $\Phi$ is modular of weight $m$ on $\mathbb{C}^{n}\times \H$, then 
$$
\boxed{ \dashint_{E_\tau^n} \bracket{\prod\limits_{i=1}^n {d^2z_i\over \im \tau}}
 \Phi \quad \text{is modular of weight $m$ on $\H$}}
 $$
 and thus
 $$
{ {1\over n!}\sum_{\sigma\in S_n} \int_{A_{1}} dz_{\sigma(1)}\cdots \int_{A_{n}} dz_{\sigma(n)} \Phi(z_1,\cdots, z_n;\tau) \quad \text{is quasi-modular of weight $m$ on $\H$}}.
 $$
 
 \item If $\Phi$ is modular of weight $m$ on $\mathbb{C}^{n}\times \H$, then for any $\sigma\in S_{n}$ the ordered $A$-cycle integral
$$
{ \int_{A_{1}} dz_{\sigma(1)}\cdots \int_{A_{n}} dz_{\sigma(n)} \,\Phi(z_1,\cdots, z_n;\tau) \quad \text{is quasi-modular of mixed weight on $\H$}}
$$
with leading weight $m$. Moreover, there exists an explicit formula for each weight component.

 
 \end{enumerate}

\end{thm}

The 1st and 2nd statements of Theorem \ref{main-theorem} say that the average of ordered $A$-cycle integrals 
$$
 {1\over n!}\sum_{\sigma\in S_n} \int_{A_{1}} dz_{\sigma(1)}\cdots \int_{A_{n}} dz_{\sigma(n)} \Phi(z_1,\cdots, z_n;\tau) 
$$
for a modular $\Phi$
is quasi-modular of the same weight as $\Phi$. Its modular completion is precisely given by 
$$
 \dashint_{E_\tau^n} \bracket{\prod\limits_{i=1}^n {d^2z_i\over \im \tau}}
 \Phi\,. 
$$
This generalizes the previous result on the modular completion of a single $A$-cycle integral. 

The 3rd statement in Theorem \ref{main-theorem} says that in general the ordered $A$-cycle integral
$$
 \int_{A_{1}} dz_{\sigma(1)}\cdots \int_{A_{n}} dz_{\sigma(n)} \,\Phi(z_1,\cdots, z_n;\tau) 
$$
is quasi-modular of mixed weight. 
While the quasi-modularity property
has been discussed intensively in the literature \cite{Dijkgraaf:1995, Kaneko:1995, Roth:2009mirror, Li:2011mi, Boehm:2014tropical, Goujard:2016counting,Oberdieck:2018} (see also \cite{Bloch:2000, Eskin:2001asymptotics, Roth:2010mirror,  Zagier:2016partitions, Chen:2018quasimodularity, Chen:2020masur}),
the mixed-weight phenomenon was only recently discovered in \cite{Goujard:2016counting, Oberdieck:2018}. It is also proved in  \cite{Oberdieck:2018} by a different method that the average of ordered $A$-cycle integrals
gives a quasi-modular form of pure weight.

Our result offers a geometric origin of the mixed-weight quasi-modularity. In fact, unlike the ordered $A$-cycle integrals, where specifying the $A$-cycles breaks the symmetry under the action of the modular group, our regularized integral is over the \emph{whole elliptic curve} and therefore \emph{respects modularity}. Theorem \ref{main-theorem} offers a precise connection between regularized integrals and ordered  $A$-cycle integrals. It not only shows the mixed-weight quasi-modularity of each ordered $A$-cycle integral, but also provides  explicit combinatorial formulae for all the components of different weights.  Such formulae arise (see Section \ref{subsec:A-cycle integral}) from the Poincar\'{e}-Birkhoff-Witt Theorem applied to the standard algebraic fact 
$$
\text{tensor algebra\,=\,universal enveloping algebra of free Lie algebra}. 
$$
This fact tells one can express any ordered tensor into symmetric tensors of multi-commutators.  Explicit formulae can be obtained from the result in \cite{Solomon:1968}. The multi-commutators of $A$-cycle integrations lead to multi-residues, while symmetric tensors are organized into quasi-modular objects of pure weight by Theorem \ref{main-theorem}. Each residue operation decreases the modular weight by one, leading to the mixed-weight  phenomenon.

Here is an example (Example \ref{exn=3mixedweight}) with $n=3$. 
\begin{align*}
&\int_{A_1} dz_1 \int_{A_2} dz_2 \int_{A_3} dz_3 \,\Phi=
\lim_{\bar\tau\to \infty}\left\{\dashint_{E_\tau}{d^2z_1\over \im \tau} \dashint_{E_\tau}{d^2z_2\over \im \tau} \dashint_{E_\tau}{d^2z_3\over \im \tau} \,\Phi\right.\\
+&{1\over 2}\dashint_{E_\tau}{d^2z_1\over \im \tau} \dashint_{E_\tau}{d^2z_3\over \im \tau} \oint_{z_3} dz_2\, \Phi+{1\over 2}\dashint_{E_\tau}{d^2z_2\over \im \tau} \dashint_{E_\tau}{d^2z_3\over \im \tau} \oint_{z_3} dz_1 \,\Phi+{1\over 2}\dashint_{E_\tau}{d^2z_2\over \im \tau} \dashint_{E_\tau}{d^2z_3\over \im \tau} \oint_{z_2} dz_1\,\Phi\\
&\left.+ {1\over 3}\dashint_{E_\tau}{d^2z_3\over \im \tau} \oint_{z_3}dz_1 \oint_{z_3}dz_2\, \Phi-{1\over 6}\dashint_{E_\tau}{d^2z_3\over \im \tau} \oint_{z_3}dz_2 \oint_{z_3}dz_1 \,\Phi \right\}. 
\end{align*}
Terms in different lines of the above formula have different modular weights. By the 1st statement in Theorem \ref{main-theorem}, each term is given by 
an average of ordered $A$-cycle integrals.  

We also present an example  (Example \ref{exmixedweight})  with $n=4$ and 
$$\Phi(z_1,z_2,z_3,z_4;\tau)=\wp(z_{1}-z_{2};\tau)\wp(z_{2}-z_{3};\tau)\wp(z_{3}-z_{4};\tau)\wp(z_{4}-z_{1};\tau)\,.$$
Here $\wp(z;\tau)$ is the Weierstrass $\wp$-function.
All of the inequivalent ordered $A$-cycle integrals are explicitly computed
to be quasi-modular forms of mixed weight (here  $'={1\over 2\pi i}\partial_{\tau}$)
\begin{eqnarray*}
\int_{A_{1}}dz_{4}\int_{A_{2}}dz_{3}\int_{A_{3}}dz_{2}\int_{A_{4}}dz_{1}\,\Phi&=&
\left({\pi^{8}\over 3^{4}}
E_{2}^{4} 
-{2^{5}\pi^{8}\over 3^2}E_{2}'''\right)
+\left({2^{5}\pi^{8}\over 3^2\cdot  5 }E_{4}'\right)
\,,\\
\int_{A_{1}}dz_{3}\int_{A_{2}}dz_{4}\int_{A_{3}}dz_{2}\int_{A_{4}}dz_{1}\,\Phi &=&\left(
{\pi^{8}\over 3^{4}}
E_{2}^{4}-{2^{5}\pi^{8}\over 3^2}E_{2}'''\right)
+\left(-{2^{4}\pi^{8}\over 3^2 \cdot 5}E_{4}'\right)\,,\\
\int_{A_{1}}dz_{4}\int_{A_{2}}dz_{2}\int_{A_{3}}dz_{3}\int_{A_{4}}dz_{1}\,\Phi&=&\left(
{\pi^{8}\over 3^{4}}
E_{2}^{4}-{2^{5}\pi^{8}\over 3^2}E_{2}'''\right)
+\left(-{2^{4}\pi^{8}\over 3^2 \cdot 5}E_{4}'\right)
\,.
\end{eqnarray*}
 It is illuminating to  see directly here that averaging the ordered $A$-cycle integrals leads to cancellation of lower-weight terms and we find a quasi-modular form
 of pure weight
$$
{1\over 4!}\sum_{\sigma\in S_4} \int_{A_1}dz_{\sigma(1)}
\int_{A_2}dz_{\sigma(2)}\int_{A_3}dz_{\sigma(3)} \int_{A_4} dz_{\sigma(4)} \Phi
={\pi^{8}\over 3^{4}}E_{2}^{4}+{2\pi^{8}\over 3^{4}} (3 E_{2}^{2}E_{4} -4 E_{2}E_{6} +  E_{4}^{2})\,.
$$

\begin{rem}
Theorem \ref{main-theorem} clarifies mathematically several aspects of chiral deformations of two dimensional conformal field theories in the sense of \cite{Dijkgraaf:1997chiral}. The integral $ \dashint_{E_\tau^n} $
  can be viewed as a direct computation of correlation functions on the elliptic curve $E_\tau$ using Feynman rules, while the ordered $A$-cycle integrals can be viewed as computations from the operator formalism point of view. These two computations are not exactly the same in general but are related to each other by contact terms and the holomorphic limit $\bar \tau\to \infty$ \cite{rudd1994string, Douglas:1995conformal, Dijkgraaf:1995,Dijkgraaf:1997chiral}. This explains why the operator formalism usually leads to quasi-modularity and how it is related to the modularity inherited from the geometry of the elliptic curve. In the theory of chiral deformations of free boson \cite{Dijkgraaf:1995,Dijkgraaf:1997chiral,li2016vertex}, the appearance of $\bar\tau$-dependence is a two-dimensional example of the holomorphic anomaly in the context of Kodaira-Spencer gravity on Calabi-Yau manifolds  \cite{Bershadsky:1993cx,Costello:2012cy}. 
\end{rem}

\subsection*{Regularized v.s. $A$-cycle Feynman graph integrals: chiral boson example}
We consider Feynman graph integrals in chiral boson theory on $E_\tau$ (see Section \ref{subsecgraphintegrals}).  Let 
$$
\widehat{P}(z_1,z_2;\tau,\bar\tau):=\wp(z_1-z_2;\tau)+{\pi^2\over 3}\widehat{E}_{2}(\tau,\bar{\tau})\,, \quad 
\widehat{E}_{2}(\tau,\bar{\tau})=E_{2}(\tau)-{3\over \pi }{1\over \im \tau}\,.
$$
Here $\wp(z;\tau)$ is the Weierstrass $\wp$-function and $E_2$ is the 2nd Eisenstein series. Let
$$
{P}(z_1,z_2;\tau):=\wp(z_1-z_2;\tau)+{\pi^2\over 3}{E}_{2}(\tau)\,.
$$

Let $\Gamma$ be an oriented graph with no self-loops. Let $E(\Gamma)$ be its set of edges, and $V(\Gamma)$ be
its set of vertices with cardinality $n=|V(\Gamma)|$. We label the vertices by fixing an identification 
$$
V(\Gamma)\to \{1,2,\cdots, n\}\,. 
$$
Consider the following quantity associated to the labeled graph 
$$\Phi_{\Gamma}(z_1,\cdots,z_n;\tau, \bar\tau):=\prod_{e\in E(\Gamma)}\widehat{P}(z_{t(e)},z_{h(e)};\tau,\bar\tau)\,. 
$$
Here $h(e)$ is the head of the edge $e$ and $t(e)$ is the tail. Denote
$$
\lim\limits_{\bar\tau\to \infty}\Phi_{\Gamma}(z_1,\cdots,z_n;\tau, \bar\tau):=\prod_{e\in E(\Gamma)}{P}(z_{t(e)},z_{h(e)};\tau)\,. 
$$

The regularized Feynman graph integral on $E_\tau$ for $\Gamma$  in our context is  (Definition \ref{dfn-graph})
$$\widehat{I}_{\Gamma}:=\dashint_{E_{\tau}^{n}}  \bracket{\prod_{i=1}^n {d^2z_i\over \im \tau}}
  \Phi_{\Gamma}(z_1,\cdots, z_n;\tau,\bar{\tau}) \,. 
  $$

\begin{thm}[Theorem \ref{thm-Feynmangraphintegralsmodular}] \label{main-thm2}
$\widehat{I}_{\Gamma}$
  is an almost-holomorphic modular form
 of weight $2|E(\Gamma)|$.
  Its holomorphic limit is given by 
$$
\lim_{\bar\tau\to \infty}\widehat{I}_{\Gamma}={1\over n!}\sum_{\sigma\in S_n} \int_{A_{1}} dz_{\sigma(1)}\cdots \int_{A_{n}} dz_{\sigma(n)} \,\lim_{\bar{\tau}\rightarrow \infty}
\Phi_{\Gamma}(z_1,\cdots, z_n;\tau,\bar{\tau}) $$
which is a holomorphic quasi-modular form of the same weight. Here $A_1, \cdots, A_n$ are $n$ disjoint representatives of the $A$-cycle class on $E_\tau$. 
\end{thm}

The corresponding ordered $A$-cycle integrals, especially those associated to trivalent graphs, have attracted a lot of attention in recent years
\cite{Roth:2009mirror, Li:2011mi, Boehm:2014tropical, Goujard:2016counting,Oberdieck:2018}
since their introduction \cite{Dijkgraaf:1995} to the studies of mirror symmetry for elliptic curves.

Theorem \ref{main-thm2} above
connects  regularized Feynman graph integrals to the corresponding ordered $A$-cycle integrals via the holomorphic limit $\lim\limits_{\bar\tau\to \infty}$.
In particular, it provides a very practical way to compute regularized Feynman graph integrals from ordered $A$-cycle integrals. This is demonstrated through several examples (Examples \ref{exFeymangraphintegral1}, Example \ref{exFeymangraphintegral2}, Example \ref{exFeymangraphintegral3}).
 
\begin{rem} There is another approach \cite{Li:2011mi} to study such graph integrals on $E_\tau$ using the heat kernel regularization, following the effective renormalization method developed in \cite{Costello:2011book}. 
We expect that the heat kernel regularization there and the regularization discussed in this paper produce
the same regularized graph integrals.
\end{rem}

\subsection*{Organization of the paper}

In Section \ref{secregularizedintegrals}
we introduce the notion of regularized integrals on Riemann surfaces and on 
configuration spaces of Riemann surfaces, and establish their main properties.

In Section \ref{sec-3}
we apply our theory to elliptic curves. We prove the modularity of regularized integrals on configuration spaces and relate it to the quasi-modularity of ordered $A$-cycle integrals. As a byproduct, we offer a geometric proof of the mixed-weight phenomenon of
ordered $A$-cycle integrals and provide concrete formulae for each weight component. We then illustrate our results through examples of Feynman graph integrals. 

In Appendix \ref{secmodularformsellipticfunctions}
we review the basics of modular forms and elliptic functions.
Some combinatorial arguments in proving our main results and details in evaluating certain regularized/ordered $A$-cycle integrals are relegated to 
Appendix \ref{appendixalgebraicidentity}
and Appendix  \ref{secstraightforwardevaluation}, respectively.

\subsection*{Acknowledgment}

The authors would like to thank Kevin Costello, Robbert Dijkgraaf, Yi-Zhi Huang, Noriko Yui, and Dingxin Zhang
 for helpful communications. 
 Thanks also go to the anonymous referees for carefully reading the manuscript and for providing very useful suggestions that have 
 improved the paper.
 S.~L. would like to thank Yunchen Li and Xinyi Li, whose growth have reformulated many aspects of the presentation of the current work.
 
 This work of S.~L. and J.~Z. is supported by the National Key R\&D Program of China  (NO. 2020YFA0713000). S.~L. is also partially supported by grant 11801300 of National Natural Science Foundation of China  and grant Z180003 of Beijing Municipal Natural Science Foundation.  J.~Z. is also partially supported by the start-up grant at Tsinghua University and the Young overseas high-level talents introduction plan of China. 
Part of this work was initiated while J.~Z. was a postdoctoral fellow and S.~L. was a visiting fellow  at Perimeter Institute for Theoretical Physics in 2017. 
We thank the institute for its hospitality and provision of excellent working environment.

\section{Regularized integrals on Riemann surfaces}
\label{secregularizedintegrals}

In this section we introduce the notion of regularized integrals on Riemann surfaces (Definition \ref{defn-RI}),
and on the configuration spaces of Riemann surfaces (Definition \ref{defn-integral-product})). We explore aspects of their geometric properties which will play important roles in Section \ref{sec-3} for our application.

The main goal is to give a meaning (see Definition \ref{defn-integral-product} and Theorem \ref{thm-integral-product}) to the divergent integral 
$$
\int_{\Sigma^n} \omega
$$
where $\omega$ is a differential form on the product $\Sigma^n$ with arbitrary meromorphic poles along the diagonals, and thus defines a smooth form on the $n$-point configuration space of $\Sigma$. These integrals arise naturally as correlation functions of non-local operators in a chiral CFT on $\Sigma$  through the form
$$
 \abracket{\int_{\Sigma}\OO_1(z_1)\cdots  \int_{\Sigma}\OO_n(z_n)}=\int_{\Sigma^n}  \abracket{\OO_1(z_1)\cdots \OO_n(z_n)}.
$$
Here the $\OO_i$'s are $2$-form valued operators on $\Sigma$. Then our notion of regularized integrals can be viewed as providing a mathematical tool to study such correlation functions of non-local operators.

\subsection{Regularized integrals}\label{sec-curve-integral}

\subsubsection*{Compact surface}

Let $\Sigma$ be a compact Riemann surface, possibly with  boundary $\pa\Sigma$. We will concentrate on regularized integrals on compact surfaces and briefly discuss modifications to non-compact surfaces at the end of this subsection. 

Let $\OO_\Sigma$ be the sheaf of holomorphic functions and $\Omega^\bullet_\Sigma$ be  the sheaf of holomorphic forms on $\Sigma$. We sometimes write $\OO, \Omega^\bullet$ for simplicity when $\Sigma$ is clear from the context. Let
$$
\A^{p,q}(\Sigma):= \A^{0,q}(\Sigma, \Omega^p)\,, \quad p,q=0,1
$$
be the space of smooth $(p,q)$-forms on $\Sigma$ and 
$$
  \A^k(\Sigma):= \bigoplus_{p+q=k}\A^{p,q}(\Sigma)
$$
be the space of smooth $k$-forms. 

Let $D\subset \Sigma$ be a finite subset of points which does not meet $\pa \Sigma$. Let
$$
\Omega_{\Sigma}^\bullet(\star D):=\bigcup_{n\geq 0} \Omega^\bullet_\Sigma(nD)
$$
be the sheaf of meromorphic forms which are holomorphic on $\Sigma-D$ but possibly with arbitrary order of poles along $D$. Let $\Omega^1_\Sigma(\log D)$ be the subsheaf of $\Omega^1_\Sigma(\star D)$ consisting of 1-forms that are  logarithmic along $D$. Since $\Sigma$ is complex one-dimensional, we have
$$
\Omega^1_\Sigma(\log D)=\Omega_\Sigma^1\otimes_{\OO_\Sigma} \OO_\Sigma(D)\,. 
$$
Let 
$$
\Omega_\Sigma^\bullet(\log D):= \Omega^0_\Sigma\oplus \Omega^1_\Sigma(\log D)\,.
$$

We use
$$
\A^{p,q}(\Sigma,\star D):= \A^{0,q}(\Sigma, \Omega^p(\star D))\,, \quad \A^{p,q}(\Sigma,\log D):= \A^{0,q}(\Sigma, \Omega^p(\log D)), \quad p,q=0,1
$$
for the corresponding $(p,q)$-forms with specified poles along $D$ and 
$$
  \A^k(\Sigma,\star D):= \bigoplus_{p+q=k}\A^{p,q}(\Sigma,\star D)\,,\quad   \A^k(\Sigma,\log D)=\bigoplus_{p+q=k}\A^{p,q}(\Sigma,\log D)\,. 
$$
By definition, elements of $\A^k(\Sigma,\star D)$ are $k$-forms $\omega$ which are smooth on $\Sigma-D$ and are of the form 
$$
\omega= {\alpha \over z^n} \quad \text{in a small neighborhood of $p\in D$}\,.
$$
Here $z$ is a local holomorphic coordinate with $z(p)=0$, $n$ is a non-negative integer and $\alpha$ is a smooth $k$-form around $p$. The form  $\omega$ lies in $\A^k(\Sigma,\log D)$ if it is locally of the form
$$
\omega = \alpha + \beta {dz\over z}
$$
where $\alpha, \beta$ are smooth $(0,\bullet)$-forms around $p$. 

The complex $\A^{\bullet,\bullet}(\Sigma,\star D)$ is a bi-graded complex with two natural differentials 
$$
\dbar: \A^{\bullet,\bullet}(\Sigma,\star D)\to \A^{\bullet,\bullet+1}(\Sigma,\star D)\,, \quad \pa: \A^{\bullet,\bullet}(\Sigma,\star D)\to \A^{\bullet+1,\bullet}(\Sigma,\star D)\,.
$$
Moreover, $\A^{\bullet, \bullet}(\Sigma,\log D)\subset \A^{\bullet,\bullet}(\Sigma,\star D)$ is a bi-graded subcomplex. The total differential
$$
d=\dbar+\pa
$$
is the de Rham differential. 

The goal of this subsection is to explain that the following integral
$$
\int_{\Sigma} \omega\,, \quad \omega \in \A^2(\Sigma,\star D)
$$
has a canonical meaning although $\omega$ may have poles along $D$. This will be called \textbf{regularized integral}. To avoid possible confusion, the regularized integral will be denoted by 
$$
\dashint_\Sigma \omega\,. 
$$
It will extend the usual integral for smooth forms, i.e., the following diagram is commutative 
$$
\xymatrix{
\A^2(\Sigma)\ar@{^{(}->}[rr]  \ar[dr]_{\int_\Sigma} && \A^2(\Sigma,\star D)\ar[dl]^{\dashint_\Sigma}\\
&\C&
}
$$

We start with two useful lemmas. 

\begin{lem}\label{lem-decomposition} Any $\omega \in \A^{1,\bullet}(\Sigma,\star D)$ can be written as
$$
\omega=\alpha+\pa \beta\,, \quad \text{where}\quad \alpha\in \A^{1,\bullet}(\Sigma, \log D)\,, \quad \beta\in \A^{0,\bullet}(\Sigma, \star D)\,. 
$$
The supports of $\alpha$ and $\beta$ can be chosen to be contained in the support of $\omega$. 
\end{lem}
\begin{proof} Let $D=\{p_1,\cdots, p_m\}$. Let $\{U_i\}$ be disjoint open subsets such that $p_i\in U_i$ for each $i$. Let $V_i\subset U_i$ be smaller open subsets with $\overline{V}_i\subset U_i$ such that
$$
U_0:= \Sigma-\cup_i \overline{V_i}\,, \quad U_1, \cdots, U_m
$$
define an open cover of $\Sigma$. Let 
$$
1=\rho_0+\rho_1+\cdots+\rho_m\,, \quad \overline{\text{Supp}(\rho_k)}\subset U_k \,,\quad k=0,\cdots,m
$$
be a partition of unity subordinate to this open cover. Let $\omega_j= \rho_j \omega$. We have
$$
\omega=\omega_0+\omega_1+\cdots+\omega_m\,. 
$$
Since $\omega_0$ is smooth, we only need to show that each $\omega_i$ can be written as
$$
\omega_i=\alpha_i+\pa \beta_i\,, \quad i=1,\cdots, m\,.
$$
Then we can choose
$$
   \alpha=\omega_0+ \sum_{i=1}^m \alpha_i\,, \quad \beta=\sum_{i=1}^m \beta_i\,. 
$$

The problem is local and we focus on the small neighborhood $U_i$ with local holomorphic coordinate $z$ such that $z(p_i)=0$. Assume
$$
\omega_i={ dz  \over z^n} \wedge g
$$
where $g$ is smooth  with compact support in $U_i$. If $n=1$, then we can choose
$$
\alpha_i= \omega_i\,, \quad \beta_i=0\,. 
$$
If $n>1$, then 
$$
\omega_i=-{1\over n-1} \pa\bracket{g \over z^{n-1}}+{\pa g \over (n-1)z^{n-1}}\,. 
$$
We repeat this process to reduce the order of pole and eventually find $\alpha_i, \beta_i$ as required. 
\end{proof}

\begin{lem}\label{lem-limit} Let $f$ be a smooth function around the origin $0\in \C$. Let $n$ be a positive integer. Then
$$
\lim_{\epsilon\to 0} \int_{|z|=\epsilon} {f d\bar z\over z^n}=0\,, \qquad \lim_{\epsilon\to 0} \int_{|z|=\epsilon} {f d z\over z^n}= {2\pi i  \over (n-1)!}\pa_z^{n-1} f(0)\,.
$$ 
Here the integration contour is counter-clockwise oriented. 
\end{lem}
\begin{rem}When $f$ is holomorphic, this reduces to the usual residue formula. 
\end{rem}
\begin{proof} Let
$$
  h=\sum_{\substack{k,m\geq 0\\ k+m\leq n}}a_{km}z^k \bar z^m\,, \quad a_{km}={1\over k!m!}\pa_z^k\pa_{\bar z}^m f |_{z=0}
$$
be the Taylor approximation of $f$ at $z=0$ up to order $n$. From the remainder estimate
$$
f= h+ o(|z|^n)\,,
$$
we have
$$
\lim_{\epsilon\to 0} \int_{|z|=\epsilon} {f d\bar z \over z^n}=\lim_{\epsilon\to 0} \int_{|z|=\epsilon}  {h d\bar z\over z^n}\,, \quad \lim_{\epsilon\to 0} \int_{|z|=\epsilon} {f dz \over z^n}=\lim_{\epsilon\to 0} \int_{|z|=\epsilon}  {h dz\over z^n}\,.
$$
Therefore we only need to analyze the case when $f=z^k\bar z^m$ is a single monomial. Then
\begin{eqnarray*}
\lim_{\epsilon\to 0} \int_{|z|=\epsilon} {z^k \bar z^m d\bar z \over z^n}&=&\lim_{\epsilon\to 0} \epsilon^{k+m+1-n} (-i)\int_{0}^{2\pi} e^{(k-m-1-n)i\theta}d\theta\\
&=&\lim_{\epsilon\to 0} \epsilon^{k+m+1-n} (-2\pi i)\delta_{k,m+1+n}=0\,.
\end{eqnarray*}
Similarly 
\begin{eqnarray*}
\lim_{\epsilon\to 0} \int_{|z|=\epsilon} {z^k \bar z^m d z \over z^n}&=&\lim_{\epsilon\to 0} \epsilon^{k+m+1-n} i\int_{0}^{2\pi} e^{(k-m+1-n)i\theta}d\theta\\
&=&2\pi i \lim_{\epsilon\to 0}  \epsilon^{k+m+1-n}\delta_{k,m-1+n}=2\pi i \delta_{m,0}\delta_{k,n-1}\,. 
\end{eqnarray*}

\end{proof}

\begin{thm}\label{thm-integral} Let $\omega \in \A^2(\Sigma, \star D)$. Then there exist $ \alpha\in \A^2(\Sigma, \log D),  \beta\in \A^{0,1}(\Sigma, \star D)$ such that 
$
\omega=\alpha+\pa \beta
$. The integral $
\int_\Sigma \alpha 
$
is absolutely convergent and the sum
$$
\int_{\Sigma}\alpha + \int_{\pa \Sigma} \beta
$$
does not depend on the choice of $\alpha, \beta$. 

\end{thm}
\begin{proof} Such $\alpha, \beta$ exist by Lemma \ref{lem-decomposition}. $\int_\Sigma \alpha$ is absolutely convergent since $\alpha$ is logarithmic. $ \int_{\pa \Sigma} \beta$ is also well-defined since $D$ does not meet $\pa \Sigma$. Assume we have two expressions
$$
\omega=\alpha+\pa \beta=\alpha^\prime+\pa \beta^\prime\,. 
$$
Let $z_i$ be a local coordinate around $p_i\in D$ such that $z_i(p_i)=0$. Let $B_\epsilon^i=\{|z_i|\leq \epsilon\}$ be a small $\epsilon$-ball centered at $p_i$. Then 
\begin{align*}
\int_\Sigma(\alpha-\alpha^\prime)&=\lim_{\epsilon\to 0} \int_{\Sigma- \cup_i B_\epsilon^i} (\alpha-\alpha^\prime)=-\lim_{\epsilon\to 0}\int_{\Sigma- \cup_i B_\epsilon^i} \pa (\beta-\beta^\prime)\\
&=-\lim_{\epsilon\to 0}\int_{\Sigma- \cup_i B_\epsilon^i} d (\beta-\beta^\prime)=-\int_{\pa \Sigma}(\beta-\beta^\prime)+\sum_i \lim_{\epsilon\to 0}\int_{\pa B_\epsilon^i}(\beta-\beta^\prime)\\
&=-\int_{\pa \Sigma}(\beta-\beta^\prime)\,. 
\end{align*}
Here we have used Lemma \ref{lem-limit} in the last step. This proves the theorem. 
\end{proof} 

\begin{dfn}\label{defn-RI} We define the \textbf{regularized integral} 
$$
\dashint_\Sigma: \A^\bullet(\Sigma, \star D)\to \C
$$
by
$$
\dashint_{\Sigma}\omega:= \begin{cases}
0 & \text{if}\quad \omega \in \A^{\leq 1}(\Sigma, \star D)\,,\\
\int_{\Sigma}\alpha+ \int_{\pa \Sigma} \beta & \text{if}\quad \omega=\alpha+\pa \beta\in \A^2(\Sigma, \star D)\,. 
\end{cases}
$$
Here  $\alpha\in \A^2(\Sigma, \log D),  \beta\in \A^{0,1}(\Sigma, \star D)$.  
\end{dfn}

The regularized integral is well-defined by Theorem \ref{thm-integral}. It clearly extends  the usual integration of smooth forms
$$
\dashint_\Sigma\omega =\int_\Sigma \omega, \quad \text{for}\quad \omega \in \A^2(\Sigma)\,. 
$$

\begin{ex}\label{exregularizedintegralex}
Let $E_{\tau}=\mathbb{C}/(\mathbb{Z}\oplus \mathbb{Z}\tau),\mathrm{im}\,\tau>0$ be an elliptic curve and
$\wp(z;\tau)$ be the Weierstrass $\wp$-function. We now evaluate the regularized integral
$$\dashint_{E_\tau} \wp(z;\tau)  {i\over 2\,\mathrm{im}\tau}dz\wedge d\bar{z}\,.$$
It is a standard fact from the theory of elliptic functions (see Appendix \ref{secmodularformsellipticfunctions}) that
$$\widehat{P}(z;\tau):=\wp(z;\tau)+\eta_1+{-\pi\over \mathrm{im}\,\tau}=-\partial_{z}
\widehat{Z}\,,\quad
\widehat{Z}:=\zeta(z;\tau)-z \eta_1+{-\pi\over \mathrm{im}\,\tau} (\bar{z}-z)\,,$$
where $\eta_1={\pi^2\over 3}E_{2}$ is a semi-period, $\zeta(z;\tau)$ the Weierstrass $\zeta$-function.
Note that both $\widehat{P}$ and $\widehat{Z}$ are elliptic functions on $\mathbb{C}$ and thus descend to functions on $E_{\tau}$, while 
this is not so for $\zeta$.
It follows that
$$\widehat{P}(z;\tau){i\over 2\,\mathrm{im}\tau}dz\wedge d\bar{z}
=0-\partial \left( 
\widehat{Z}\cdot {i\over 2\,\mathrm{im}\,\tau}d\bar{z}
\right)\,.
 $$
Hence by Definition \ref{defn-RI} one has
$$\dashint_{E_\tau} \widehat{P}(z;\tau)  {i\over 2\, \mathrm{im}\tau}dz\wedge d\bar{z}
=0\,,$$
and thus
$$\dashint_{E_\tau} \wp(z;\tau)  {i\over 2\, \mathrm{im}\tau}dz\wedge d\bar{z}
=-\dashint_{E_\tau}( \eta_1+{-\pi\over \mathrm{im}\,\tau}) {i\over  2\, \mathrm{im}\tau}dz\wedge d\bar{z}
=-(\eta_1+{-\pi\over \mathrm{im}\,\tau})\,.$$

\end{ex}

\begin{prop}\label{prop-integration-noboundary}
Assume $\pa\Sigma=\emptyset$.  Then the regularized integral factors through the quotient
$$
\dashint_{\Sigma}: \quad  {\A^2(\Sigma, \star D)\over \pa \A^{0,1}(\Sigma, \star D)}\to \C\,. 
$$
\end{prop}
\begin{proof} This follows by construction. 
\end{proof}

\begin{prop}\label{prop-pullback} Let $f: (\Sigma_1, \pa \Sigma_1)\to (\Sigma_2, \pa \Sigma_2)$ be a diffeomorphism which is bi-holomorphic. Let $D_1\subset \Sigma_1$ be a finite subset which does not meet $\pa \Sigma_1$ and $D_2=f(D_1)$. Let $\omega \in \A^2(\Sigma_2, \star D_2)$. Then the pull-back $f^*\omega \in \A^2(\Sigma_1, \star D_1)$ and
$$
\dashint_{\Sigma_1} f^*\omega= \dashint_{\Sigma_2} \omega\,. 
$$
\end{prop}
\begin{proof} It is clear that $f^*\omega \in \A^2(\Sigma_1, \star D_1)$. Let $ \alpha\in \A^2(\Sigma_2, \log D_2),  \beta\in \A^{0,1}(\Sigma_2, \star D_2)$ such that 
$
\omega=\alpha+\pa \beta
$. Since $f$ is holomorphic, 
$$
f^*\omega=f^*\alpha +f^*(\pa \beta)  =f^*\alpha +\pa (f^*\beta)\,.
$$
Therefore
$$
\dashint_{\Sigma_1} f^*\omega=\int_{\Sigma_1} f^*\alpha+\int_{\pa \Sigma_1}f^*\beta=\int_{\Sigma_2}\alpha+\int_{\pa\Sigma_2}\beta=\dashint_{\Sigma_2} \omega\,. 
$$

\end{proof}

\subsubsection*{Cauchy principal value}

The regularized integral can be viewed as a version of Cauchy principal value (see e.g., \cite{Demailly:2012complex}). In other words, the conformal structure of the Riemann surface leads to a regularization scheme that defines intrinsically the principal value of integrals of forms with holomorphic poles. This is demonstrated by the following theorem.

\begin{thm}\label{thm-PV} Suppose
$\omega \in \A^2(\Sigma, \star D)$.  Let $z_i$ be a local holomorphic coordinate around $p_i\in D$ such that $z_i(p_i)=0$. Let $B_\epsilon^i=\{|z_i|\leq \epsilon\}$ be a small $\epsilon$-ball centered at $p_i$. Then
$$
\dashint_\Sigma \omega= \lim_{\epsilon\to 0} \int_{\Sigma- \cup_i B_\epsilon^i}  \omega\,. 
$$
In particular, the limit $\lim\limits_{\epsilon\to 0} \int_{\Sigma- \cup_i B_\epsilon^i}  \omega$ exists and does not depend on the choice of the local coordinate $z_i$'s. 
\end{thm}
\begin{proof} Let $\omega=\alpha+\pa \beta$ where $\alpha\in \A^2(\Sigma, \log D),  \beta\in \A^{0,1}(\Sigma, \star D)$. Then $\omega=\alpha+d \beta$ also holds and
$$
\lim_{\epsilon\to 0} \int_{\Sigma- \cup_i B_\epsilon^i}  \omega=\lim_{\epsilon\to 0} \int_{\Sigma- \cup_i B_\epsilon^i} (\alpha+ d\beta)=\lim_{\epsilon\to 0} \int_{\Sigma- \cup_i B_\epsilon^i} \alpha + \int_{\pa\Sigma}\beta- \sum_i \lim_{\epsilon\to 0}\int_{\pa B_\epsilon^i}\beta= \int_\Sigma \alpha+\int_{\pa\Sigma}\beta\,. 
$$
Here $\lim\limits_{\epsilon\to 0}\int_{\pa B_\epsilon^i}\beta=0$ by Lemma \ref{lem-limit}. 
\end{proof}

\subsubsection*{Non-compact surface}

Let $\Sigma$ be a  non-compact Riemann surface without boundary.\footnote{For non-compact Riemann surfaces with boundaries, 
 the same formula in Definition \ref{defn-RI} defines the regularized integral for compactly support forms. } Let 
$$
\A^k_c(\Sigma, \star D)\subset \A^k(\Sigma, \star D)
$$
denote the space of forms with compact support, i.e., forms that vanish outside a compact subset in $\Sigma$. 

\begin{dfn} Given $\omega \in \A^2_c(\Sigma, \star D)$, we define the regularized integral by 
$$
\dashint_\Sigma \omega:=\int_\Sigma \alpha\,. 
$$
Here $\omega=\alpha+\pa \beta$ where $\alpha\in \A_c^2(\Sigma, \log D),  \beta\in \A_c^{0,1}(\Sigma, \star D)$. 
\end{dfn}

The existence of $\alpha, \beta$ is guaranteed by Lemma \ref{lem-decomposition}. The regularized integral  for compactly supported forms is similar to that on compact surfaces without boundary. It factors through 
$$
\dashint_\Sigma: {\A^2_c(\Sigma, \star D)\over \pa \A^{0,1}_c(\Sigma, \star D)}\to \C\,. 
$$

\subsection{Residues and Stokes theorem}

In this subsection, we establish a version of Stokes formula for regularized integrals. 

Let us first describe an extension of residue in our context. 

\begin{lem-dfn}\label{lem-residue} Let $\alpha \in \A^1(\Sigma, \star D)$. Let $p\in D$ and $z$ be a local holomorphic coordinate around $p$ such that $z(p)=0$. Then the following limit 
$$
\lim_{\epsilon\to 0}\int_{|z|=\epsilon} \alpha
$$
exists and does not depend on the choice of the local coordinate $z$. Here the integration contour is counter-clockwise oriented. We will denote this limit by 
$$
\oint_p \alpha:=\lim_{\epsilon\to 0}\int_{|z|=\epsilon} \alpha\,.
$$
Moreover, if $\alpha=\pa \beta$ for some $\beta\in \A^0(\Sigma, \star D)$ or $\alpha\in \A^{0,1}(\Sigma, \star D)$, then 
$
\oint_p \alpha=0. 
$

\end{lem-dfn}
\begin{proof} Let $z$ be a local holomorphic coordinate with $z(p)=0$. By Lemma \ref{lem-limit}, $\lim\limits_{\epsilon\to 0}\int_{|z|=\epsilon} \alpha$ exists and will vanish if $\alpha$ is a $(0,1)$-form. If  $\alpha=\pa \beta$ for some $\beta\in \A^0(\Sigma, \star D)$, then
$$
\lim_{\epsilon\to 0}\int_{|z|=\epsilon} \alpha=\lim_{\epsilon\to 0}\int_{|z|=\epsilon} \pa \beta=-\lim_{\epsilon\to 0}\int_{|z|=\epsilon} \dbar \beta=0\,. 
$$

We next show such limit is independent of the choice of the local coordinate $z$. We can assume $\alpha$ is a $(1,0)$-form
and work locally near $z=0$. As in the proof of Lemma
\ref{lem-limit}, it is enough to approximate $\alpha$ by a truncated Taylor approximation.
For such a polynomial, one then has
 the expansion 
$$\alpha=\alpha_0+\bar{z}\beta\,,$$
where $\alpha_0$ is a meromorphic $(1,0)$-form and $\beta$ is smooth with possibly holomorphic pole at $z=0$.
A similar calculation as in Lemma \ref{lem-limit} shows that 
$$\lim_{\epsilon\to 0} \int_{|z|={\epsilon}} \bar{z}\beta=0\,.$$
Hence
$$\lim_{\epsilon\to 0} \int_{|z|={\epsilon}} \alpha=\lim_{\epsilon\to 0} \int_{|z|={\epsilon}} \alpha_0\,.$$
Let $w$ be another local holomorphic coordinate around $p$ with $w(p)=0$. We have a similar decomposition as above
$$\alpha=\widetilde{\alpha}_0+\bar{w}\widetilde{\beta}\,.$$
Since the coordinate transformation $z\to w$ is holomorphic, by type reasons one has
$$
\alpha_0=\widetilde{\alpha}_0
$$
and they are closed 1-forms. By Stokes theorem, 
$$ \int_{|z|={\epsilon}} \alpha_0= \int_{|w|={\epsilon}} \alpha_0\,,$$
and this integral is independent of the sufficiently small $\epsilon$.
It follows that
$$
\lim_{\epsilon\to 0}\int_{|z|=\epsilon} \alpha=
\lim_{\epsilon\to 0} \int_{|z|={\epsilon}} \alpha_0
=
 \int_{|z|={\epsilon}} \alpha_0
= \int_{|w|={\epsilon}} \alpha_0
=\int_{|w|={\epsilon}} \widetilde{\alpha}_0
=
\lim_{\epsilon\to 0}\int_{|w|=\epsilon}  \widetilde{\alpha}_0\
=
\lim_{\epsilon\to 0}\int_{|w|=\epsilon} \alpha\,.
$$
\end{proof}

\begin{rem} When $\alpha$ is a closed 1-form, the value $\int_{|z|=\epsilon} \alpha$ does not depend on $\epsilon$ by Stokes theorem. In general, $\int_{|z|=\epsilon} \alpha$ will depend on $\epsilon$, and our notation $\oint_p \alpha$ only refers to the limit value as $\epsilon\to 0$. 
 
\end{rem}

\begin{dfn} Let $\alpha \in \A^1(\Sigma, \star D)$ and $p\in D$. We define the local residue of $\alpha$ at $p$ by 
$$
\Res_{p}\alpha:={1\over 2\pi i}\oint_p \alpha\,. 
$$
We define the global residue of $\alpha$ on $\Sigma$ by
$$
\Res_{\Sigma}\alpha:=\sum_{p\in D}\Res_p \alpha\,. 
$$
\end{dfn}

Lemma \ref{lem-limit} says that $\Res_p$ factors through 
$$
\Res_p:  {\A^1(\Sigma, \star D)\over \A^{0,1}(\Sigma, \star D)+\pa \A^0(\Sigma, \star D)}\to \C\,. 
$$
It coincides with the usual residue when $\alpha$ is a meromorphic $(1,0)$-form.

 The next theorem describes a version of Stokes formula for a regularized integral. It  generalizes the global residue theorem for meromorphic 1-forms on Riemann surfaces.

\begin{thm}\label{thm-de-Rham}  Let $\Sigma$ be a compact Riemann surface possibly with boundary $\pa\Sigma$. Let $\alpha \in \A^1(\Sigma, \star D)$. Then we have the following version of Stokes formula for the regularized integral
$$
\dashint_\Sigma d \alpha=-2\pi i \Res_\Sigma(\alpha)+\int_{\pa\Sigma}\alpha\,. 
$$
\end{thm}
\begin{proof}   Let $D=\{p_1,\cdots, p_m\}$ and $z_i$ be a local coordinate around $p_i$ with $z_i(p_i)=0$. Let $B_\epsilon^i=\{|z_i|\leq \epsilon\}$ be a small $\epsilon$-ball centered at $p_i$. By Theorem \ref{thm-PV}, we have
$$
\dashint_\Sigma d \alpha=\lim_{\epsilon\to 0} \int_{\Sigma- \cup_i B_\epsilon^i}d \alpha= - \lim_{\epsilon\to 0} \sum_{i} \int_{|z_i|=\epsilon} \alpha+\int_{\pa \Sigma}\alpha=-2\pi i \Res_\Sigma(\alpha)+\int_{\pa\Sigma}\alpha\,.
$$
\end{proof}

\begin{rem} When $\Sigma$ is a compact surface without boundary and $\alpha$ is a meromorphic $(1,0)$-form so that $d\alpha=0$, Theorem \ref{thm-de-Rham} reduces to the usual global residue formula
$$
\Res_{\Sigma}\alpha=0\,. 
$$
\end{rem}

\begin{thm}\label{thm-de-Rham-noncompact}  Let $\Sigma$ be a non-compact Riemann surface. Let $\alpha \in \A^1_c(\Sigma, \star D)$. Then 
$$
\dashint_{\Sigma}d\alpha= -2\pi i \Res_\Sigma (\alpha). 
$$
\end{thm}
\begin{proof} Similar to the proof of Theorem \ref{thm-de-Rham}. 
\end{proof}

We end this subsection with a simple proposition that will be useful for computations. 

\begin{prop}\label{prop-residue} Let $\alpha \in \A^1(\Sigma, \star D), p \in D$ and $f$ be a holomorphic function on $\Sigma$.  Then
$$
\oint_p (\bar f \alpha)= \bar f(p) \oint_p \alpha\,. 
$$
Here $\bar f$ is the complex conjugate of $f$. 
\end{prop}
\begin{proof} This follows from Lemma \ref{lem-limit}. 
\end{proof}

\subsection{Riemann-Hodge bilinear formula}

The Stokes formula allows one to express certain regularized integral via Riemann-Hodge bilinear type formula. We illustrate the basic idea in this subsection by computing
$$
\dashint_\Sigma \omega
$$
where $\Sigma$ is a compact genus-$g$ Riemann surface without boundary, and $\omega$ is of the form
$$
\omega=\varphi\wedge {\alpha}\,,\quad
\varphi \in \Omega^{1}(\Sigma, \star D)\,,\quad \alpha\in \A^{1}(\Sigma)\,,\quad d\alpha=0\,.
$$
That is, $\varphi$ is a meromorphic $(1,0)$-form with poles along $D$, and $\alpha$ is a smooth closed $1$-form.

We follow the method presented in  \cite{Griffiths:2014principles} to compute the above integral.  Fix once and for all a canonical basis of $H_{1}(\Sigma,\mathbb{Z})$.
Fix a reference point $p_0\in\Sigma$.
Let $\delta_{1},\cdots ,\delta_{2g}$
be cycles representing the canonical basis  that are issued from $p_0$: $\delta_1,\cdots,\delta_g$ correspond to \emph{$A$-cycles}, and $\delta_{g+1},\cdots,\delta_{2g}$ correspond to \emph{B-cycles}.  We choose these cycles such that they do not intersect $D$. The complement of these cycles on $\Sigma$ is a simply connected region $\Delta$ on $\Sigma$.

\begin{figure}[H]\centering
	\begin{tikzpicture}[scale=1]
	\foreach \j in {1,...,8} {
		\draw[cyan!23,fill=cyan!23](-22.5+\j*45:2)to(0,0)to(22.5+\j*45:2);}
	\draw[ultra thick,red,-<-=.5,>=stealth](22.5+30*4.5:2)to(-22.5+30*4.5:2);
	\draw[ultra thick,red,-<-=.5,>=stealth](67.5+30*4.5:2)to(90+22.5+30*4.5:2); \draw[ultra thick,green,-<-=.5,>=stealth](22.5+70*4.5:2)to(-22.5+70*4.5:2);
	\draw[ultra thick,green,-<-=.5,>=stealth](67.5+70*4.5:2)to(90+22.5+70*4.5:2);
	
	\draw[ultra thick,orange,-<-=.5,>=stealth](22.5+40*4.5:2)to(-22.5+40*4.5:2);
	\draw[ultra thick,orange,-<-=.5,>=stealth](67.5+40*4.5:2)to(90+22.5+40*4.5:2);
	\draw[ultra thick,blue,-<-=.5,>=stealth](22.5+80*4.5:2)to(-22.5+80*4.5:2);
	\draw[ultra thick,blue,-<-=.5,>=stealth](67.5+80*4.5:2)to(90+22.5+80*4.5:2);
	
	\node [above] at ($(67.5+80*4.5:2)!0.5!(90+22.5+80*4.5:2)$) {$\delta_i$};
	
	\node [right] at ($(22.5+80*4.5:2)!0.5!(-22.5+80*4.5:2)$) {$\delta_i$};
	
\node [below right] at ($(22.5+70*4.5:2)!0.5!(-22.5+70*4.5:2)$) {$\delta_{g+i}$};
		
\node [above right] at ($(67.5+70*4.5:2)!0.5!(90+22.5+70*4.5:2)$) {$\delta_{g+i}$};

	\node [below] at (1,0) {$p_m$};
	
	\draw[fill] (1,0) circle  [radius=1 pt];
	
		\node [below] at (-1,0) {$p_2$};
	
	\draw[fill] (-1,0) circle  [radius=1 pt];

		\node [below] at (0,1) {$p_1$};
	
	\draw[fill] (0,1) circle  [radius=1 pt];
	
		\node [below] at (0,-1) {$\cdots$};
	
   \node at (0,-2.5) {$D=\{p_1,p_2,\cdots, p_m\}$};
	
	\foreach \j in {1,...,8} {
		\draw[black](22.5+\j*45:2)\nn;}
	\end{tikzpicture}
	\caption{Riemann surface with cut locus.}\label{fig:RSwithcut}
\end{figure}
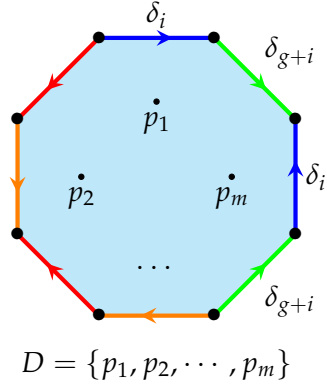

Let 
$$
\pi:\, \Delta\to \Sigma
$$
denote the inclusion. The divisor $D$ can be viewed as a finite subset of points lying in the interior of $\Delta$ via the quotient $\pi$. The pull-back under $\pi$ defines a map
$$
\pi^*: \,\A^k(\Sigma, \star D)\to \A^k(\Delta, \star D)\,. 
$$
It is straightforward to check (e.g., by using Theorem \ref{thm-PV}) that
$$
\dashint_\Delta \pi^* \omega =\dashint_\Sigma \omega\,, \quad \forall \,\omega \in \A^2(\Sigma, \star D)\,. 
$$

\begin{prop}\label{prop-bilinear} 
Let $\varphi$ be a meromorphic $(1,0)$-form with poles along $D$, and $\alpha$ be a smooth closed $1$-form. Let $u$ be the function on $\Delta$ defined by
$${u}(z)=\int_{p_0}^{z}\pi^*(\alpha)\,, \quad z\in \Delta\,. 
$$
Here the integral $\int_{p_0}^{z}$ is over an arbitrary curve inside $\Delta$ that connects $p_0$ to $z$. Then
\begin{equation*}\label{eqnbilinearformforsingularintegral}
\dashint_{\Sigma}\varphi\wedge {\alpha}
=\sum_{i=1}^{g}\left(\int_{\delta_{i}}\pi^*(\varphi) \int_{\delta_{g+i}}\pi^*({\alpha})-\int_{\delta_{i}}\pi^*({\alpha}) 
\int_{\delta_{g+i}}\pi^*(\varphi)  \right)+
2\pi i \sum_{p\in D}\Res_{p}\bracket{\pi^*(\varphi ) u}\,.
\end{equation*}

\end{prop}

\begin{rem} If $\alpha$ is holomorphic, then the regularized integral 
$$
\dashint_{\Sigma}\varphi\wedge {\alpha}
$$
is zero since $\varphi\wedge {\alpha}=0$ by type reasons. The formula above reduces  to the usual Riemann-Hodge bilinear relation. 
\end{rem}

\begin{proof}
Since $\dashint_{\Sigma}\varphi\wedge {\alpha}=\dashint_{\Delta}\pi^*(\varphi\wedge {\alpha})$, we shall work with $\Delta$. Observe that
$$\pi^*(\varphi \wedge {\alpha})=-d({u}\pi^*(\varphi ))\,.
$$
Applying Theorem \ref{thm-de-Rham}, we find 
$$
\dashint_{\Delta}\pi^*(\varphi\wedge {\alpha})=-\int_{\pa \Delta} {u}\pi^*(\varphi) +2\pi i \sum_{p\in D}\Res_p\bracket{\pi^*(\varphi )  u}. 
$$
Evaluating $\int_{\partial\Delta}
{u}\pi^*(\varphi) $ directly, one finds 
\begin{equation*}\label{eqnbilinearterm}
-\int_{\partial\Delta}
{u}\pi^*(\varphi) =\sum_{i=1}^{g}\left(\int_{\delta_{i}}\pi^*(\varphi) \int_{\delta_{g+i}}\pi^*({\alpha})-\int_{\delta_{i}}\pi^*({\alpha}) 
\int_{\delta_{g+i}}\pi^*(\varphi)  \right).
\end{equation*}
This leads to the desired formula. 
\end{proof}

As a consistent check, the bilinear formula  in Proposition \ref{prop-bilinear} is independent of 
the choices of cycles representing the canonical basis,
and the choice of canonical basis itself. In fact, although different choices can
give rise to different line integrals individually (differ by the residues of $\varphi$), 
the overall sum remains unchanged.
The formula is also independent of the choice of the reference point $p_0 $ by the global residue theorem.

In the case when $\alpha$ is an anti-holomorphic $1$-form, Proposition \ref{prop-bilinear} has an intrinsic formulation as follows. Consider the 1st homology $H_{1}(\Sigma-D)$ of the complement $\Sigma-D$ and the 1st relative homology $H_1(\Sigma,D)$. Lefschetz-Poincar\'e duality implies that we have a perfect pairing via intersection 
$$\cap:\,H_{1}(\Sigma-D)\times H_{1}(\Sigma, D)\rightarrow \mathbb{Z}\,.$$

\begin{prop} Let $\varphi$ be a meromorphic $(1,0)$-form with poles along $D$, and ${\bar\psi}$ be an anti-holomorphic $1$-form. Let $\{\gamma_i\}$ be a basis of $H_1(\Sigma-D)$ and $\{\gamma^i\}$ be the dual basis of $H_{1}(\Sigma, D)$ such that
$
  \gamma_i \cap \gamma^j=\delta_i^j. 
$ Then 
$$
\dashint_{\Sigma}\varphi\wedge {{\bar\psi}}=\sum_i \int_{\gamma_i} \varphi  \int_{\gamma^i}{\bar\psi}\,. 
$$
\end{prop}
\begin{proof}
It is enough to show this for a particular choice of basis. Then linearity implies the results for all the other choices. We present one choice inside $\Delta$ as follows. Let $D=\{p_1, \cdots, p_m\}$ which lie in the interior of $\Delta$. For each $1\leq i\leq m-1$, let $c_i$ be a small counter-clockwise oriented loop around $p_i$, and let $b_i$ be a path in $\Delta$ that start from $p_m$ and ends at $p_i$. We require all $c_i$'s do not intersect, and $b_i$ only intersects $c_i$ at one point.  Then 
$$
\{\gamma_i\}=\{\delta_1, \dots, \delta_g, \delta_{g+1}, \cdots, \delta_{2g}, c_1, \cdots, c_{m-1}\}
$$
is a basis of $H_1(\Sigma-D)$, and 
$$
\{\gamma^i\}=\{\delta_{g+1}, \cdots, \delta_{2g}, -\delta_1,\cdots, -\delta_g, b_1, \cdots, b_{m-1}\}
$$
is the dual basis of $H_1(\Sigma, D)$.  Let
$$
{u}(z)=\int_{p_0}^{z}\pi^*({{\bar\psi}}). 
$$

By Proposition \ref{prop-residue},
$$
\Res_{p}\bracket{\pi^*(\varphi ) u}=\bracket{\Res_{p}\pi^*(\varphi )} \bracket{ \int_{p_0}^{p}\pi^*({{\bar\psi}})}. 
$$
Since $
 \sum\limits_{i=1}^m\Res_{p_i}\pi^*(\varphi)=0
$, we can write 
\begin{align*}
2\pi i \sum_{p\in D}\Res_{p}\bracket{\pi^*(\varphi ) u}
=2\pi i \sum_{i=1}^{m-1}\bracket{\Res_{p_i}\pi^*(\varphi )}\cdot \int_{p_{m}}^{p_i}\pi^*({{\bar\psi}})=\sum_i \int_{c_i} \pi^*(\varphi )  \int_{b_i} \pi^*({{\bar\psi}}). 
\end{align*}
This leads to the desired formula (cf. the ``Double Copy Formula" in \cite[Corollary 3.19]{Brown:2021})
\begin{align*}
\dashint_{\Sigma}\varphi\wedge {{\bar\psi}}&=\sum_{i=1}^{g}\left(\int_{\delta_{i}}\pi^*(\varphi) \int_{\delta_{g+i}}\pi^*({{\bar\psi}})-\int_{\delta_{i}}\pi^*({{\bar\psi}}) 
\int_{\delta_{g+i}}\pi^*(\varphi)  \right)+ \sum_{i=1}^{m-1} \int_{c_i} \pi^*(\varphi )  \int_{b_i} \pi^*({{\bar\psi}}) \\
&=\sum_i \int_{\gamma_i} \varphi  \int_{\gamma^i}{\bar\psi}. 
\end{align*}
\end{proof}

 \subsubsection*{Application: prototypical example on $A$-cycle integral and quasi-modularity}
 
 As an application, we discuss the quasi-modularity of certain $A$-cycle integrals on elliptic curves. Systematic studies and generalizations will be presented in Section \ref{sec-3}. 

  Let $\tau$ be a point on the upper half-plane $\H$. Let 
\begin{equation*}\label{eqnuniformization}
 E_\tau=\C/\Lambda_{\tau}\,,\quad \Lambda_{\tau}:=\Z+ \Z \tau
 \end{equation*}
 be the corresponding elliptic curve. We will use $z$ as the linear holomorphic coordinate on the universal cover $\C$. We consider the following action of $\gamma =\begin{pmatrix}
 a& b\\
 c&d
 \end{pmatrix}  \in \mathrm{SL}_{2}(\mathbb{Z})$ on $\mathbb{C}\times \H$  by
 \begin{eqnarray*}
 \gamma :
 \mathbb{C}\times \H &\rightarrow & \mathbb{C}\times \H\,,\\
 (z;\tau)&\mapsto &
 (\gamma z;\gamma \tau):= ({z\over c\tau+d} ;{a\tau+b\over c\tau+d})\,.
   \end{eqnarray*}

Let $\Phi(z;\tau)$ be a meromorphic function on $\mathbb{C}\times \H$ which is 
 \begin{equation*}
 \text{elliptic}: \quad \Phi( z+\lambda; \tau)=\Phi(z;\tau)\,,
 \quad
 \forall \lambda\in  \Lambda_{\tau}\,,
 \end{equation*}
and modular of weight $k\in \mathbb{Z}$
 \begin{equation*}
 \Phi(\gamma z;\gamma \tau)=(c\tau+d)^{k}\Phi(z;\tau)\,,
 \quad
 \forall \gamma \in \mathrm{SL}_{2}(\Z)\,.
 \end{equation*}
Then $\Phi(-;\tau)$ defines a meromorphic function on $E_\tau$ for each $\tau\in \H$. 

\begin{prop}\label{propmodularityofregularizedintegral} The regularized integral 
$$
f(\tau)=\dashint_{E_\tau}{d^2z\over \im \tau} \Phi(z;\tau)\, , \quad d^2z :={i \over 2}dz\wedge d\bar z
$$
is modular of weight $k$ as a function of $\tau\in \H$, i.e., 
$
f(\gamma \tau)=(c\tau +d)^k f(\tau)\,, \ \forall\, \gamma\in  \mathrm{SL}_{2}(\Z)\,.
$
\end{prop}
\begin{proof} Given a fixed $\tau$, we choose a parallelogram $\square_c$ in $\C$ with vertices $\{c,c+1,c+1+\tau, c+\tau\}$ such that the poles $D_\tau$ of $\Phi(-;\tau)$  do not lie on the boundary of $\square_c$. 

Then $\square_c$ is a fundamental domain for $E_\tau$. We have
$$
\dashint_{E_\tau}{d^2z\over \im \tau} \Phi(z;\tau) =\dashint_{\square_c}{d^2z\over \im \tau} \Phi(z;\tau)\,. 
$$
Let $\gamma \square_c$ be the image of $\square_c$ under the $\gamma$-action.
Then  $\gamma \square_c$ is a fundamental domain for $E_{\gamma \tau}$ whose boundary does not intersect poles of $\Phi(z;\gamma \tau)$. Hence
\begin{align*}
f(\gamma \tau)&=\dashint_{E_{\gamma\tau}}{d^2z\over \im (\gamma\tau)} \Phi(z;\gamma\tau)=\dashint_{\gamma \square_{\tau}}{d^2z\over \im (\gamma\tau)} \Phi(z;\gamma\tau)=\dashint_{\square_c} {d^2(\gamma z)\over \im (\gamma \tau)} \Phi(\gamma z;\gamma \tau)\\
&=(c\tau+d)^{k}\dashint_{E_\tau} {d^2z\over \im \tau} \Phi(z;\tau)=(c\tau+d)^{k} f(\tau). 
\end{align*}
Here we have used Proposition \ref{prop-pullback} in the third equality.

\end{proof}

As an example, we consider the case when $\Phi(-;\tau)$ has no residue at any pole. Then
$$
\varphi= \Phi(z;\tau)dz
$$
defines a 2nd kind Abelian differential on $E_\tau$. Let $\{A, B\}\in H_1(E_\tau, \Z)$ be the basis with representatives as in the figure below. We assume such representatives do not intersect the poles $D$ of $\varphi$. Otherwise we perturb the chosen representatives slightly to avoid poles. 

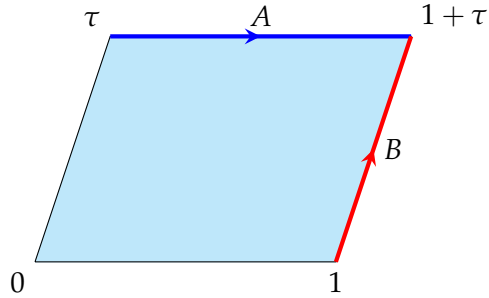
\begin{figure}[H]\centering
	\begin{tikzpicture}[scale=1]

\draw[cyan!23,fill=cyan!23](0,0)to(4,0)to(5,3)to(1,3)to(0,0);
\draw (0,0) node [below left] {$0$} to (4,0) node [below] {$1$} to (5,3) node [above right] {$1+\tau$} to (1,3) node [above left] {$\tau$} to (0,0);

\draw[ultra thick,blue,->-=.5,>=stealth](1,3)to(5,3);

\node [above] at (3,3) {$A$};

\draw[ultra thick,red,->-=.5,>=stealth](4,0)to(5,3);

\node [right] at (4.5,1.5) {$B$};

	\end{tikzpicture}
\end{figure}

Since $\varphi$ has no residues, the $A$-cycle integral 
$$
\int_A \varphi
$$
does not depend on the representative $A$ of its cycle class. By Proposition \ref{prop-bilinear}, 
\begin{align*}
\dashint_{E_{\tau}}{d^2z\over \im \tau}\Phi&=\dashint_{E_{\tau}}\varphi \wedge { d \im z \over \im \tau}=\int_A \varphi \int_B { d \im z \over \im \tau}- \int_B \varphi \int_A { d \im z \over \im \tau}
+2\pi i \sum_{p\in D}\Res_{p}\bracket{\varphi \im z \over \im \tau}.
\end{align*}
Since $\Res_p(\varphi)=0$, we have 
$$
\Res_p(\varphi \bar z)=0\,, \quad  \sum_{p\in D}\Res_p(\varphi z)=-\langle \varphi, dz \rangle_{\mathrm{P}}\,,
$$
where $\langle  \varphi, dz \rangle_{\mathrm{P}}$ is the Poincar\'e residue pairing. It follows that 
$$
\dashint_{E_{\tau}}{d^2z\over \im \tau}\Phi= \int_{A}\varphi -{1\over 2 i \im \tau}\cdot  2\pi i\,\langle \varphi, dz\rangle_{\mathrm{P}}\,. 
$$

The following proposition regarding the quasi-modularity
 (in the sense of \cite{Kaneko:1995}, see Definition \ref{dfnmodularquasimodular})
 of $\int_{A}\varphi$
  is immediate.
\begin{prop}\label{prop-modular-completion}Assume $\Phi$ is modular of weight $k$ and $\varphi=\Phi dz$ is a 2nd kind Abelian differential on $E_\tau$. Then the $A$-cycle integral 
$$
\int_A \varphi
$$
is quasi-modular of weight $k$ whose ``modular completion" (see Appendix \ref{secmodularformsellipticfunctions}) is given by the regularized integral 
$$
\dashint_{E_{\tau}}{d^2z\over \im \tau}\Phi= \int_{A}\varphi -{1\over 2 i \im \tau}\cdot  2\pi i\,\langle \varphi, dz\rangle_{\mathrm{P}}\,. 
$$

\end{prop}
\begin{proof}
By Proposition \ref{propmodularityofregularizedintegral},
the  regularized integral $\dashint_{E_{\tau}}{d^2z\over \im \tau}\Phi$
is modular.
One can prove similarly that $\langle \varphi, dz\rangle_{\mathrm{P}}$,
as a meromorphic function in $\tau$,  is also modular.
Quasi-modularity of $\int_{A}\varphi$
then follows from the modularity of $\dashint_{E_{\tau}}{d^2z\over \im \tau}\Phi$
, the modularity of $\langle \varphi, dz\rangle_{\mathrm{P}}$,
and the relation between  $\dashint_{E_{\tau}}{d^2z\over \im \tau}\Phi$ and
$\int_{A}\varphi$
 above.
\end{proof}

\begin{rem}
The integral
$\int_{A}\varphi$ being quasi-modular instead of modular is due to the fact that 
specifying the $A$-cycle breaks
the symmetry under the action of the modular group.
Our geometric perspective offers a natural way to obtain its modular completion in terms of regularized integral on the whole elliptic curve. This phenomenon will be vastly generalized in Section \ref{sec-3} (Theorem \ref{thm-2d}). 
\end{rem}

\begin{ex}\label{expmoments} 

Consider $\Phi=\wp^{m}$ with $\varphi=\wp^{m}dz$, where $m\geq 1$ and $\wp$ is the Weierstrass $\wp$-function.  
The computations of integrals $\int_{A}\wp^{m}dz$ 
are standard, see e.g., \cite{Silverman:2009},
by using the Weierstrass equation \eqref{eqn-functionalfieldofellipticcurve} 
and the relation \eqref{eqn-quasiperiods}.

For example, for $m=1$ one has from  \eqref{eqn-quasiperiods}
$$\int_{A}\wp(z;\tau) dz=-{\pi^{2}\over 3}E_{2}(\tau)\,.$$
By the Weierstrass equation \eqref{eqn-functionalfieldofellipticcurve} we have
$2\wp''=12\wp^2-{4\over 3}\pi^4 E_{4}$ and hence 
$$ \wp^2 dz={1\over 6}d\wp'+ {1\over 9}\pi^4 E_{4} dz\,.$$
This gives
\begin{align*}
\int_A \wp^{2}(z;\tau)dz =\int_A 
\left(
{1\over 6}d\wp'+ {1\over 9}\pi^4 E_{4} dz
\right)={\pi^{4}\over 9}E_{4} \,.
\end{align*}
Similarly, 
for the $m=3$ case we multiply the relation $2\wp''=12\wp^2-{4\over 3}\pi^4 E_{4}$ by $\wp$ and obtain
\begin{eqnarray*}
12\wp^3 dz-{4\over 3}\pi^4 E_{4}\wp dz
&=&2 \wp\wp''
dz=
2\, d(\wp \wp')-2(\wp')^2 dz\\
&=&
2 \, d(\wp \wp')-2 (4\wp^3 dz-{4\over 3}\pi^{4}E_{4} \wp dz-{8\over 27}\pi^{6}E_{6}dz)\,.
\end{eqnarray*}
It follows that
$$\int_{A}\wp^{3}dz={1\over 20}\int_{A} ( 4\pi^{4}E_{4}\wp+{16\over 27} \pi^6 E_{6})dz
=-{1\over 15}\pi^{6} E_{2}E_{4}+{4\over 5\cdot 27} \pi^{6}E_{6}\,.$$

Using  \eqref{eqn-g2g3}  and \eqref{eqn-wponC}, one has
$$\langle \wp dz, dz \rangle_{\mathrm{P}}=-1\,,
\quad
\langle \wp^{2} dz, dz \rangle_{\mathrm{P}}=0\,,
\quad
\langle \wp^{3} dz, dz \rangle_{\mathrm{P}}=- {3\over 20}g_{2}=-{\pi^{4}\over 5}E_{4}\,.
$$

Therefore, we obtain (cf. Example \ref{exregularizedintegralex})
\begin{eqnarray*}
\int_{A}\wp(z;\tau) dz -{1\over 2 i \im \tau}\cdot  2\pi i\,\langle \wp(z;\tau) dz, dz\rangle_{\mathrm{P}}&=&-{\pi^{2}\over 3}E_{2}(\tau)
+{\pi\over \im\tau}=-{\pi^{2}\over 3}\widehat{E}_{2}(\tau)\,,\\
\int_A \wp^{2}(z;\tau)dz -{1\over 2 i \im \tau}\cdot  2\pi i\,\langle \wp^2(z;\tau) dz, dz\rangle_{\mathrm{P}}&=&{\pi^{4}\over 9}E_{4} \,,\\
\int_A \wp^{3}(z;\tau)dz -{1\over 2 i \im \tau}\cdot  2\pi i\,\langle \wp^3(z;\tau) dz, dz\rangle_{\mathrm{P}}&=&-{1\over 15}\pi^{6} \widehat{E}_{2}E_{4}+{4\over 5\cdot 27} \pi^{6}E_{6} \,.
\end{eqnarray*}

\end{ex}

\begin{rem}\label{remcomputationontatecurve}
One can also 
evaluate these $A$-cycle integrals by first
lifting the function $\wp(z)$ along the Picard uniformization 
\begin{equation*}\label{eqn-Picarduniformization}
\mathbb{C}^{*}\rightarrow E_{\tau}= \mathbb{C}^{*}/q^{\mathbb{Z}}\,,\quad q=e^{2\pi i \tau}\,.
\end{equation*}
The resulting integrand is expressed in term of the holomorphic coordinate $u$ on the cover $\mathbb{C}^{*}$ that is related to the coordinate $z$ on the universal cover $\mathbb{C}$  by
$u=\exp(2\pi i z)$.
Then one computes
 the degree zero term
 in the $u$-expansion on $\mathbb{C}^{*}$. 
The formulae in \eqref{eqn-FourierexpansionsE2k} allow us
 to express the resulting $q$-series in terms of the Eisenstein series.
 See
\cite{Roth:2009mirror, Boehm:2014tropical, Goujard:2016counting} for discussions along these lines. Using the ``computing twice" trick, one can then produce many nontrivial identities
between $q$-series and quasi-modular forms.

For example, the $m=2$ case leads to the identity
\begin{equation*}
(2\pi i)^4 \cdot 2 \sum_{k\geq 1}{k^{2}q^{k}\over (1-q^{k})^2}={1\over 9}\pi^{4} (E_{4}-E_{2}^2)\,,
\end{equation*}
which can be proved directly by 
using \eqref{eqn-FourierexpansionsE2k} and the Ramanujan identities
\eqref{eqn-Ramanujanidentities}.

The $m=3$ case leads to 
$$(2\pi i)^{6}
\sum_{k,\ell\geq 1}  k\ell(k+\ell){q^{k+\ell} \over (1-q^{k}) (1-q^{\ell}) (1-q^{k+\ell})}
={\pi^{6}\over 2^{6}3^{4} 5} (5E_{2}^3-3E_{2} E_{4}-2E_{6})\,,$$
which seems to be less easy to prove directly.
\end{rem}

\subsection{Variation property}\label{sec-family}
In this subsection, we establish several variation properties of regularized integrals. 

We assume $\Sigma$ is a compact Riemann surface without boundary. Let $X$ be a complex manifold. Let $\{s_i: X\to \Sigma\}$ be a set of distinct holomorphic maps. Their graphs define a set of smooth divisors
$$
D_i=\{(z, x)\in \Sigma\times X~|~ z=s_i(x)\}\subset \Sigma\times X\,. 
$$
The intersection
$$
D_i\cap D_j
$$
can be viewed as a divisor $D_{ij}$ on $X$ under the projection $\pi: \Sigma\times X\to X$. Equivalently, $D_{ij}$ is given by the fiber product
$$
\xymatrix{
D_{ij}\ar[r]\ar[d] & X\ar[d]^{s_i\times s_j}\\
 \Sigma\ar[r]^{\delta} & \Sigma\times \Sigma
}
$$
where $\delta: \Sigma\to \Sigma\times \Sigma$ is the diagonal map.  Let us denote
$$
D=\bigcup\limits_i D_i\,, \quad D^{(2)}=\bigcup\limits_{i\neq j}D_{ij}\,. 
$$
Here $D$ can be viewed as a family of divisors on $\Sigma$ parametrized by $X$.

We can similarly define the spaces
$$
\A^{\bullet}(\Sigma\times X, \star D)\,, \quad \A^\bullet(X, \star D^{(2)})
$$
consisting of forms with arbitrary order of poles along the specified divisors. Locally, in a small open subset $U\times V\subset \Sigma\times X$ such that each $D_i$ is defined by $z=s_i(x)$, we have
$$
\A^{\bullet}(U\times V, \star D)=\A^\bullet(U\times V)[(z-s_i(x))^{-1}]\,, \quad \A^\bullet(V,\star D^{(2)})=\A^\bullet(V)[(s_i(x)-s_j(x))^{-1}]\,. 
$$

\begin{lem}\label{lem-family} Let 
$$
\omega= {f(z,y)dz\wedge d\bar z\over (z-z_0(y))^n }
$$
be a family of $2$-forms on $z\in \C$ parametrized by  $y\in \mathbb{R}$. Here $z_0(y)$ is smooth in $y$; $f(z,y)$ is smooth in $z,y$ and we have omitted the $\bar z$-dependence for convenience; $f(-,y)$ has compact support along $z\in \C$ for any fixed $y$. Then the regularized integral $\dashint_\C \omega$ is smooth in $y$ and 
$$
\pa_{y}\dashint_{\C} \omega=\dashint_{\C} \pa_{y}\omega+\pa_y \overline{z_0(y)} \oint_{z_0(y)}  {f(z,y)  \over (z-z_0(y))^n}dz\,.
$$
In particular, if $t$ is a complex variable, $z_0(t)$ is holomorphic in $t$, and $\omega= {f(z,t)dz\wedge d\bar z\over (z-z_0(t))^n }
$, then 
$$
\pa_{t}\dashint_{\C} \omega=\dashint_{\C} \pa_{t}\omega\,, \quad \pa_{\bar t}\dashint_{\C} \omega=\dashint_{\C} \pa_{\bar t}\omega+\overline{\pa_{t} z_0(t)} \oint_{z_0(t)} {f(z,t)  \over (z-z_0(t))^n}dz\,.
$$
\end{lem}
\begin{proof} A change of coordinate $z\to z+z_0(y)$ by shifting and Proposition \ref{prop-pullback} imply
$$
\dashint_{\C}\omega= \dashint_{\C} {f(z+z_0(y),y) dz\wedge d\bar z\over z^n}\,. 
$$
Let $\pa$ denote the holomorphic de Rham differential in $z$. We can write 
\begin{equation*}
{f(z+z_0(y),y) dz\wedge d\bar z\over z^n}= {1\over (n-1)!}{\pa_z^{n-1}f(z+z_0(y),y) dz\wedge d\bar z\over z}+ \pa \beta \tag{$\dagger$}
\end{equation*}
for some $(0,1)$-form $\beta$. It follows that 
\begin{align*}
\dashint_{\C} {f(z+z_0(y),y) dz\wedge d\bar z\over z^n}={1\over (n-1)!} \int_{\C} {\pa_z^{n-1}f(z+z_0(y),y) dz\wedge d\bar z\over z}\,.
\end{align*}
This immediately implies the smoothness of $\dashint_{\C}\omega$ in $y$. 

We can apply $\pa_y$ to equation ($\dagger$) and use the fact that $\pa_y$ commutes with $\pa$ to find
\begin{align*}
\pa_{y}\dashint_{\C}\omega
&= \dashint_{\C}\pa_y\bracket{{f(z+z_0(y),y) dz\wedge d\bar z\over z^n}}\\
&=\dashint_{\C} {(\pa_y z_0(y))\pa_zf(z+z_0(y),y) dz\wedge d\bar z\over z^n}+\dashint_{\C} {\pa_y f(z+z_0(y),y) dz\wedge d\bar z\over z^n}\\
&\quad +\dashint_{\C} {(\pa_y \overline{z_0(y)})\pa_{\bar z}f(z+z_0(y),y) dz\wedge d\bar z\over z^n}\,.
\end{align*}

The first two terms on the last expression give
\begin{align*}
&(\pa_y z_0(y))\dashint_{\C} {\pa_zf(z,y) dz\wedge d\bar z\over (z-z_0(y))^n}+\dashint_{\C} {\pa_y f(z,y) dz\wedge d\bar z\over (z-z_0(y))^n}\\
=&(\pa_y z_0(y))\dashint_{\C} \bracket{\pa\bracket{f(z,y) d\bar z\over (z-z_0(y))^n}+ {nf dz\wedge d\bar z\over (z-z_0(y))^{n+1}}}+\dashint_{\C} {\pa_y f(z,y) dz\wedge d\bar z\over (z-z_0(y))^n}\\
=& \dashint_{\C}{n (\pa_y z_0(y)) f dz\wedge d\bar z\over (z-z_0(y))^{n+1}}+\dashint_{\C} {\pa_y f(z,y) dz\wedge d\bar z\over (z-z_0(y))^n}
=\dashint_{\C}\pa_y \omega\,.
\end{align*}

The last term is 
\begin{align*}
&\dashint_{\C} {(\pa_y \overline{z_0(y)})\pa_{\bar z}f(z+z_0(y),y) dz\wedge d\bar z\over z^n}\\
=&(\pa_y \overline{z_0(y)})\dashint_{\C} {\pa_{\bar z}f(z,y) dz\wedge d\bar z\over (z-z_0(y))^n}=-(\pa_y \overline{z_0(y)})\dashint_{\C} \dbar {f(z,y) dz \over (z-z_0(y))^n}\\
=&  (\pa_y \overline{z_0(y)} )\oint_{z_0(y)}  { f(z,y)  \over (z-z_0(y))^n}dz\,. 
\end{align*}
Here we have used Theorem \ref{thm-de-Rham-noncompact} in the last step. 
\end{proof}

Given a form 
$
\omega \in \A^{\bullet}(\Sigma\times X, \star D) 
$, 
we can perform the regularized integral 
$$
\dashint_{\Sigma}\omega
$$ 
along $\Sigma$ first to end up with a form on $X$. By Lemma \ref{lem-family} and a use of partition of unity,  $\dashint_{\Sigma}\omega$ is smooth on $X-D^{(2)}$. We next analyze the singularity of $\dashint_{\Sigma}\omega$ around $D^{(2)}$. 

\begin{prop}\label{prop-push-poles} Let $\omega \in \A^{\bullet}(\Sigma\times X, \star D)$. Then
$$
\dashint_{\Sigma}\omega\in \A^{\bullet-2}(X, \star D^{(2)})\,. 
$$
In other words, $\dashint_{\Sigma}\omega$ has holomorphic poles along $D^{(2)}$. 
\end{prop}
\begin{proof} The problem is local. Let us assume we are in a neighborhood $V$ of $x$ with $s_1(x)=\cdots =s_k(x)$. Let 
$$
g= \prod_{1\leq i\neq j \leq k} (s_i(x)-s_j(x))^N
$$
where $N$ is a sufficient large integer to be determined later. Since $g$ does not depend on $\Sigma$, 
$$
\dashint_\Sigma \omega= {1\over g} \dashint_\Sigma g \omega\,. 
$$
Locally $\omega$ has the form
$$
 \omega = {dz\over \prod\limits_{i=1}^k(z-s_i(x))^{m_i}}\varphi
$$
where $\varphi$ is smooth. By a repeated use of the relation
$$
{1\over (z-s_i(x))(z-s_j(x))}={1\over (s_i(x)-s_j(x))}\bracket{{1\over z-s_i(x)}-{1\over z-s_j(x)}}
$$
we can find a large enough $N$ such that 
$$
g\omega=\sum_i {dz\over (z-s_i(x))^{m_i}} \varphi_i
$$
with $\varphi_i$ smooth. By Lemma \ref{lem-family}, $\dashint_\Sigma g\omega$ is smooth. It follows that
$$
\dashint_\Sigma \omega={1\over g}\dashint_\Sigma g\omega\in \A^{\bullet-2}(X, \star D^{(2)})\,. 
$$
\end{proof}

It follows that the regularized integral defines a push-forward map
$$
\dashint_\Sigma: \A^{\bullet}(\Sigma\times X, \star D)\to \A^{\bullet-2}(X,\star D^{(2)})\,. 
$$


\begin{rem} A version of regularized integral on $\C P^1\times \R^2$ was defined in \cite{Benini:2020}  for the 4d Chern-Simons action in a similar fashion. It  can be viewed as $\int_{\R^2} \dashint_{\C P^1}$ in our terminology. 

\end{rem}

\begin{thm}\label{thm-family} Assume $\Sigma$ is a compact Riemann surface without boundary. Then the push-forward map $
\dashint_\Sigma: \A^{\bullet}(\Sigma\times X, \star D)\to \A^{\bullet-2}(X,\star D^{(2)})
$ intertwines the holomorphic de Rham differential
$$
\dashint_\Sigma \pa_{\Sigma\times X} \omega=\pa_{X} \dashint_{\Sigma} \omega\,. 
$$
Here $\pa_{\Sigma\times X}, \pa_X$ are the holomorphic de Rham differentials on $\Sigma\times X$ and $X$, respectively.  
\end{thm}
\begin{proof}  Let us write $\pa_{\Sigma\times X}=\pa_\Sigma+ \pa_X$.  Then
$$
\dashint_{\Sigma}  \pa_{\Sigma\times X} \omega=\dashint_{\Sigma} \pa_\Sigma \omega +\dashint_{\Sigma}\pa_X \omega. 
$$
Since $\Sigma$ has no boundary, $\dashint_{\Sigma} \pa_\Sigma \omega=0$. The pole locations $s_i(x)$ vary holomorphically. Using a partition of unity and Lemma \ref{lem-family}, we find  $\dashint_{\Sigma}\pa_X \omega=\pa_{X} \dashint_{\Sigma} \omega$. The theorem follows. 
\end{proof}

\subsection{Integrals on configuration spaces} \label{sec-configuration}

Let $\Sigma$ be a compact Riemann surface without boundary.  Let $\Sigma^n$ be the $n$-th Cartesian product of $\Sigma$. 
 Fix once and for all a global enumeration $1,2,\cdots ,n$ for the factors of $\Sigma^{n}$.
 Given an index subset $I=\{i_1,\cdots,i_k\}\subset \{1,\cdots,n\}$, let
$$
\Delta_I:= \{(z_1,\cdots, z_n)\in \Sigma^n|  z_{i_1}=\cdots =z_{i_k}\}\,. 
$$
The collection of all such diagonal divisors, denoted by
$$
\Delta= \bigcup_{1\leq i\neq j\leq n} \Delta_{ij}\,,
$$
is called the big diagonal.
 
 In this subsection, we will generalize the notion of regularized integrals on $\Sigma$ to define integration of forms on $\Sigma^n$ with arbitrary holomorphic poles along $\Delta$. 
 
Let 
$$
\Omega^\bullet_{\Sigma^n}(\star \Delta) \quad \text{or simply}\quad \Omega^\bullet(\star \Delta)
$$
denote the sheaf of meromorphic forms which are holomorphic on $\Sigma^n-\Delta$ but with arbitrary order of poles along $\Delta$. Then $
\Omega^\bullet_{\Sigma^n}(\star \Delta)
$
is a sheaf of differential graded (dg) algebra where the dg structure is given by the holomorphic de Rham differential $\pa$. Let
$$
\Omega^\bullet_{\Sigma^n}(\log \Delta_{ij}) \subset \Omega^\bullet_{\Sigma^n}(\star \Delta)
$$
denote the subsheaf of logarithmic forms along the smooth divisor $\Delta_{ij}$. 

\begin{dfn}\label{defn-log} We define $\Omega^\bullet_{\Sigma^n}(\log \Delta)$ (or simply $\Omega^\bullet(\log \Delta)$) to be the sheaf of subalgebra of $\Omega^\bullet_{\Sigma^n}(\star \Delta)
$ 
generated by all $\Omega^\bullet_{\Sigma^n}( \log \Delta_{ij})$'s. 
\end{dfn}

In a local neighborhood $U$ of a point $(z_1, \cdots, z_n)$ with $z_i=z_j$ for $i,j\in I\subset \{1,\cdots, n\}$, we have
$$
\Omega^\bullet_U(\log \Delta)=\Omega^\bullet_U\bbracket{dz_{i}-dz_j\over z_i-z_j}_{i,j\in I, i\neq j}. 
$$
Here $\Omega^\bullet_U$ is the sheaf of holomorphic forms over $U$. 
\begin{rem}
The divisor
$\Delta$ is not normal crossing, so $\Omega^\bullet_{\Sigma^n}(\log \Delta)$ is not the usual log-forms associated to normal crossing divisors. We still use the notation $\log \Delta$ to illustrate its explicit meaning as above and hope it will not cause confusion. 
\end{rem}

We also consider smooth forms. Let
$$
\A^{p,q}(\Sigma^n, \star \Delta):= \A^{0,q}(\Sigma^n, \Omega^p(\star \Delta))\,,\quad \A^{p,q}(\Sigma^n, \log \Delta):= \A^{0,q}(\Sigma^n, \Omega^{p}(\log \Delta))
$$
and
$$
\A^{k}(\Sigma^n, \star \Delta)=\bigoplus_{p+q=k}\A^{p,q}(\Sigma^n, \star \Delta)\,, \quad \A^{k}(\Sigma^n, \log \Delta)=\bigoplus_{p+q=k}\A^{p,q}(\Sigma^n, \log \Delta)\,.
$$

Our first goal in this subsection is to define a regularized integral 
$$
\dashint_{\Sigma^n}: \A^{2n}(\Sigma^n, \star \Delta)\to \C
$$
in the same fashion as in Section \ref{sec-curve-integral}. As we will prove later (Theorem \ref{thm-integral-product}), such defined integral will be equal to the iterated regularized integral over the factors of $\Sigma^n$:
$$
\dashint_{\Sigma^n}=\dashint_{\Sigma_1}\cdots \dashint_{\Sigma_n}
$$
where the right hand side is well-defined by Proposition \ref{prop-push-poles}.  This would imply a Fubini type theorem for regularized integrals on $\Sigma^n$ (see Corollary \ref{cor-Fubini} below). 

\subsubsection*{Local theory}
Let us first work locally in $U^n$ where $U$ is a small open disk around the origin in $\C$. Let $(z_1,\cdots,z_n)$ be the holomorphic coordinates on $U^n$. Denote
$$
z_{ij}:= z_i-z_j\,. 
$$
Then
$$
\Omega^\bullet_{U^n}(\star \Delta)=\Omega^\bullet_{U^n}[z_{ij}^{-1}]\,,\quad \Omega^\bullet_{U^n}(\log \Delta)=\Omega^\bullet_{U^n}\bbracket{{dz_{ij}\over z_{ij}}}\,.
$$
and
$$
\A^\bullet(U^n, \star \Delta)=\A^\bullet(U^n)[z_{ij}^{-1}]\,,\quad \A^\bullet(U^n, \log \Delta)=\A^\bullet(U^n)\bbracket{{dz_{ij}\over z_{ij}}}.
$$
For simplicity, we also denote 
$$
\theta_{ij}={dz_{ij}\over z_{ij}}\,. 
$$
It is useful to observe that 
$$
{1\over z_{ij}z_{jk}}+{1\over z_{jk}z_{ki}}+{1\over z_{ki}z_{ij}}=0\,, \quad \text{for $i,j,k$ distinct}\,. 
$$
This implies that the $1$-forms  $\{\theta_{ij}\}$ satisfy the Arnold relation
$$
\theta_{ij}\theta_{jk}+\theta_{jk}\theta_{ki}+\theta_{ki}\theta_{ij}=0\,, \quad \text{for $i,j,k$ distinct}\,. 
$$

\begin{lem}\label{lem-basis}A basis  of $\C[\theta_{ij}]$ is 
$$
\fbracket{\theta_{i_1 j_1}\theta_{i_2j_2}\cdots \theta_{ i_{k}j_k}~|~ i_1<i_2<\cdots<i_k \ \text{and}\ i_1<j_1,i_2<j_2,\cdots,i_k<j_k}. 
$$
\end{lem}

Lemma \ref{lem-basis} is well-known \cite{Arnol1969The} and the list above gives a basis of the cohomology of configuration space of $n$ points on $\C$.

\begin{lem}\label{lem-basis-2} Any element in $\C[z_{ij}^{-1}]$ can be written as a linear combination of 
$$
\fbracket{\left. {1\over z^{m_1}_{i_1 j_1}z^{m_2}_{i_2j_2}\cdots z^{m_k}_{ i_{k}j_k}}\right | i_1<i_2<\cdots<i_k \ \text{and}\ i_1<j_1,i_2<j_2,\cdots,i_k<j_k, \  m_1, \cdots, m_k\geq 0}. 
$$
\end{lem}
\begin{proof} We prove by induction on the number $n$ of variables. 

Any monomial ${1\over f}$ in $\C[z_{ij}^{-1}]$ has the form
$$
{1\over f}=  {1\over z^{m_1}_{k j_1}\cdots z^{m_s}_{k j_s}  g  }\,, \quad m_1, \cdots, m_s>0
$$
where $k<j_1<\cdots<j_s$ and $g$ contains only factors of $z_{ab}$ with $a, b>k$. 

If $s=1$, we apply the induction hypothesis to ${1\over g}$. Assume $s>1$. We can use the relation 
$$
{1\over z_{k j_1}z_{kj_2}}={1\over z_{kj_1}z_{j_1j_2}}-{1\over z_{kj_2}z_{j_1 j_2}}
$$
to write ${1\over f}$ as a linear combination of terms with either $m_1$ or $m_2$ being decreased. Repeating this process, we will eventually arrive at the situation $s=1$. Then the induction applies. 

\end{proof}

\begin{lem}\label{cor-convergent} Let $\omega\in \A_c^{2n}(U^n,\log \Delta)$ be a top-form with compact support. Then the integral
$$
\int_{U^n} \omega
$$
is absolutely convergent. 
\end{lem}
\begin{proof} By Lemma \ref{lem-basis}, we can assume $\omega$ has the form
$$
\omega = \alpha \theta_{i_1 j_1}\theta_{i_2j_2}\cdots \theta_{ i_{k}j_k}
$$
where $i_1<i_2<\cdots<i_k, \ i_1<j_1,i_2<j_2,\cdots,i_k<j_k$, and $\alpha$ is a smooth $(2n-k)$-form. We can consider a linear change of coordinate such that $z_{i_1 j_1},\cdots, z_{i_k j_k}$ are part of the new coordinates. Then $\omega$ has only logarithmic pole so the integral is  absolutely convergent. 
\end{proof}

\begin{lem}\label{lem-local-structure} Any $\omega \in  \A^{2n}(U^n,\star \Delta)$ can be expressed as
$$
\omega= \alpha +\pa \beta
$$
where $\alpha \in \A^{2n}(U^n,\log \Delta)$ and $\beta \in \A^{n-1,n}(U^n, \star \Delta)$. The supports of $\alpha$ and $\beta$ can be chosen to be contained in the support of $\omega$. 
\end{lem}
\begin{proof} By Lemma \ref{lem-basis-2}, we can assume $\omega$ has the form
$$
\omega={{dz_{i_1 j_1}\over z^{m_1}_{i_1 j_1}} \cdot {dz_{i_2 j_2}\over z^{m_2}_{i_2 j_2}}\cdots {d z_{ i_{k}j_k}\over z^{m_k}_{ i_{k}j_k}}} \varphi\,, \quad \text{$\varphi$ is a smooth $(2n-k)$-form}
$$
where $i_1<i_2<\cdots<i_k$ and $i_1<j_1,i_2<j_2,\cdots,i_k<j_k, \  m_1, \cdots, m_k> 0$. Notice that $z_{i_1 j_1},\cdots, z_{i_k j_k}$ can be extended to become part of a set of linear coordinates. 

If $m_1>1$, then we can write 
\begin{align*}
\omega&=-{1\over (m_1-1)} \pa\bracket{{{1\over z^{m_1-1}_{i_1 j_1}} \cdot {dz_{i_2 j_2}\over z^{m_2}_{i_2 j_2}}\cdots {d z_{ i_{k}j_k}\over z^{m_k}_{ i_{k}j_k}}} \varphi}-(-1)^k {{1\over z^{m_1-1}_{i_1 j_1}} \cdot {dz_{i_2 j_2}\over z^{m_2}_{i_2 j_2}}\cdots {d z_{ i_{k}j_k}\over z^{m_k}_{ i_{k}j_k}}}\pa \varphi\\
&=-{1\over (m_1-1)} \pa\bracket{{{1\over z^{m_1-1}_{i_1 j_1}} \cdot {dz_{i_2 j_2}\over z^{m_2}_{i_2 j_2}}\cdots {d z_{ i_{k}j_k}\over z^{m_k}_{ i_{k}j_k}}} \varphi}+{{dz_{i_1 j_1}\over z^{m_1-1}_{i_1 j_1}} \cdot {dz_{i_2 j_2}\over z^{m_2}_{i_2 j_2}}\cdots {d z_{ i_{k}j_k}\over z^{m_k}_{ i_{k}j_k}}} \tilde\varphi
\end{align*}
where $\tilde \varphi$ is another smooth $(2n-k)$-form. Repeating this process, we can reduce $\omega$ to a form with $m_1=m_2=\cdots=m_k=1$ up a $\pa$-exact term. This proves the lemma. 
\end{proof}

\subsubsection*{Global theory}

We first have the analogue of Lemma \ref{lem-decomposition}. 
\begin{lem}\label{lem-global-decomp}Any $\omega \in  \A^{2n}(\Sigma^n,\star \Delta)$ can be expressed as
$$
\omega= \alpha +\pa \beta
$$
where $\alpha \in \A^{2n}(\Sigma^n,\log \Delta)$ and $\beta \in \A^{n-1,n}(\Sigma^n, \star \Delta)$.

\end{lem}
\begin{proof}This follows from Lemma \ref{lem-local-structure} and a use of partition of unity. 
\end{proof}

Given $\omega \in  \A^{\bullet}(\Sigma^n,\star \Delta)$, we consider its push-forward along a factor of $\Sigma$ by performing a regularized integration.  Proposition \ref{prop-push-poles} implies that
$$
\dashint_{\Sigma}: \A^{\bullet}(\Sigma^n,\star \Delta)\to \A^{\bullet-2}(\Sigma^{n-1}, \star \Delta)\,. 
$$
By Theorem \ref{thm-family}, we have 
$$
\dashint_{\Sigma} \pa \alpha=\pa \dashint_{\Sigma}\alpha\,, \quad \text{for}\quad \alpha \in \A^{\bullet}(\Sigma^n,\star \Delta)\,. 
$$
Here $\pa$ is the holomorphic de Rham differential on the corresponding space. 

\begin{thm}\label{thm-integral-product} Let $\omega \in  \A^{2n}(\Sigma^n,\star \Delta)$. Choose a decomposition (as guaranteed by Lemma \ref{lem-global-decomp})
$$
\omega= \alpha +\pa \beta\,, \quad \text{where}\quad \alpha \in \A^{2n}(\Sigma^n,\log \Delta), \beta \in \A^{n-1,n}(\Sigma^n, \star \Delta)\,. 
$$
Then the integral 
$
\int_{\Sigma^n}\alpha
$
is absolutely convergent and is equal to the iterated regularized integral
$$
\int_{\Sigma^n}\alpha= \dashint_\Sigma \dashint_{\Sigma}\cdots \dashint_{\Sigma}\omega\,. 
$$
In particular,  the value $\int_{\Sigma^n}\alpha$ does not depend on the choice of $\alpha$ and $\beta$. 
\end{thm}
\begin{proof} By Corollary \ref{cor-convergent}, $\alpha$ is absolutely integrable on $\Sigma^n$. It is logarithmic along any factor of $\Sigma$. The push-forward along a factor of $\Sigma$ gives 
$$
\dashint_\Sigma \omega= \int_\Sigma \alpha + \dashint_{\Sigma} \pa \beta= \int_\Sigma \alpha + \pa \dashint_\Sigma \beta \in \A^{2n-2}(\Sigma^{n-1},\star \Delta)\,. 
$$
The form $\int_\Sigma \alpha$ again lies in $\A^{2n-2}(\Sigma^{n-1},\log \Delta)$. This can be proved by the same method as in the proof of Proposition \ref{prop-push-poles}. Iterating this process, we eventually arrive at
$$
 \dashint_\Sigma \dashint_{\Sigma}\cdots \dashint_{\Sigma}\omega= \int_\Sigma \int_{\Sigma}\cdots \int_{\Sigma}\alpha\,.
$$
It follows that
$$
\int_{\Sigma^n}\alpha=  \dashint_\Sigma \dashint_{\Sigma}\cdots \dashint_{\Sigma}\omega\,. 
$$
In particular, this value only depends on $\omega$, but not on the choice of the decomposition. 
\end{proof}

\begin{dfn}\label{defn-integral-product} We define the \emph{regularized integral} 
$$
\dashint_{\Sigma^n}: \A^{2n}(\Sigma^n, \star \Delta)\to \C\quad \text{by}\quad \dashint_{\Sigma^n}\omega:= \int_{\Sigma^n}\alpha
$$
where $\omega=\alpha+\pa \beta$ for $\alpha \in \A^{2n}(\Sigma^n,\log \Delta)$ and $\beta \in \A^{n-1,n}(\Sigma^n, \star \Delta)$.  
\end{dfn}

Such regularized integral is well-defined by Theorem \ref{thm-integral-product}. 

\begin{cor}\label{cor-Fubini} Let $\sigma$ be any permutation of $\{1,2,\cdots, n\}$. Then the iterated regularized integral
$$
  \dashint_{\Sigma_{\sigma(1)}} \dashint_{\Sigma_{\sigma(2)}}\cdots \dashint_{\Sigma_{\sigma(n)}}\omega, \quad \omega \in \A^{2n}(\Sigma^n, \star \Delta)
$$
does not depend on the choice of $\sigma$. 

\end{cor}
\begin{proof} By Theorem \ref{thm-integral-product}, all such iterated regularized integrals are equal to $\dashint_{\Sigma^n}\omega$. 

\end{proof}

\section{Application: regularized integrals and modular forms}\label{sec-3}

In this section, we apply our theory to elliptic curves and construct a large class of modular objects, including those coming from Feynman graph integrals. We obtain a precise connection between quasi-modular forms arising from $A$-cycle integrals and regularized integrals on configuration spaces of elliptic curves (Theorem \ref{thm-2d}). This leads to a simple geometric proof of the mixed weight quasi-modularity of $A$-cycle integrals, as well as novel combinatorial formulae for all the components of different weights (Theorem \ref{thm-1d}).

We recall and fix some notations.  Let $\tau$ be a point on the upper half-plane $\H$. Let 
\begin{equation*}\label{eqnuniformization}
 E_\tau=\C/\Lambda_{\tau}\,,\quad \Lambda_{\tau}:=\Z+ \Z \tau
 \end{equation*}
 be the corresponding elliptic curve. 
 Recall that 
 $z$ is linear holomorphic coordinate on the universal cover $\C$ in terms of which
one has 
 $$
 \int_{E_\tau} {d^2z\over \im \tau}=1\,,
\quad
d^2z :={i \over 2}dz\wedge d\bar z\,.
 $$

We will fix a basis  $\{A,B\}$ for $H_{1}(E_{\tau},\mathbb{Z})$. In the universal cover $\mathbb{C}\rightarrow E_{\tau}$, $A$ is represented by the segment $[\tau,\tau+1]$ and $B$ is  represented by the segment $[1,1+\tau]$. Such $A,B$ will be called the \emph{canonical representatives}. The fundamental domains and $A, B$-cycles on both the universal cover and the Picard uniformization $\mathbb{C}^{*}\rightarrow E_{\tau}$
are displayed in Fig. \ref{fig:fundamentaldomains} below.

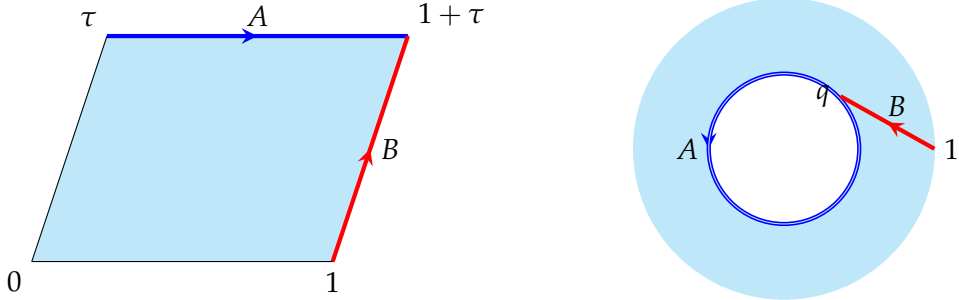
\begin{figure}[H]\centering
	\begin{tikzpicture}[scale=1.0]
\draw[cyan!23,fill=cyan!23](0,0)to(4,0)to(5,3)to(1,3)to(0,0);
\draw (0,0) node [below left] {$0$} to (4,0) node [below] {$1$} to (5,3) node [above right] {$1+\tau$} to (1,3) node [above left] {$\tau$} to (0,0);

\draw[ultra thick,blue,->-=.5,>=stealth](1,3,0)to(5,3);

\node [above] at (3,3) {$A$};

\draw[ultra thick,red,->-=.5,>=stealth](4,0)to(5,3);

\node [right] at (4.5,1.5) {$B$};

\fill[cyan!23, even odd rule](10,1.5)circle[radius=2] circle[radius=1];
\draw[cyan!23](10,1.5)circle(2); 
\node [left] at (9,1.5) {$A$};
\draw[ultra thick,blue,->-=.5,>=stealth](10,1.5)circle(1);

\draw[cyan!23](10,1.5)circle(1); 
\draw[ultra thick,red,->-=.5,>=stealth](12,1.5)to(10.75,2.2);
\node [above] at (11.5,1.8) {$B$};
\node [right] at (12,1.5) {$1$};
\node [left] at (10.75,2.2) {$q$};


	\end{tikzpicture}
	\caption{Fundamental domains and canonical representatives.
	  }
	  \label{fig:fundamentaldomains}
\end{figure}

  The following notion of holomorphic limit plays an important role in this work.

\begin{dfn}\label{dfn-holomorphiclimit}
 Let $\OO_{\H}[{1\over \im \tau}]$ denote functions $f(\tau, \bar \tau)$ on the upper half-plane $\H$ of the form
$$
f(\tau, \bar \tau)=\sum_{i=0}^N {f_i(\tau)\over (\im \tau)^i}\,, \quad N<\infty.
$$
Here the $f_i(\tau)$'s are holomorphic in $\tau$.  We define the \textbf{holomorphic limit}, denoted by
$$
\lim_{\bar\tau \to \infty}: \, \OO_{\H}[{1\over \im \tau}]\to \OO_{\H}\,,
$$
as
$$
\lim_{\bar\tau \to \infty} f(\tau, \bar \tau):= f_0(\tau)\,. 
$$
We say $f$ is \textbf{almost-holomorphic} on $\H$ if $f\in  \OO_{\H}[{1\over \im \tau}]$. 
\end{dfn}
The notion of holomorphic limit can be generalized straightforwardly to the space $\mathfrak{M}_{\H}[{1\over \im \tau}]$ consisting of 
polynomials in ${1/ \im\tau}$ with coefficients being meromorphic functions in $\tau$. We still call it holomorphic limit by abuse of language.

In Appendix \ref{secmodularformsellipticfunctions}, we collect basics of modular forms and elliptic functions that will be frequently used in this section. 

\subsection{Regularized integrals v.s. $A$-cycle integrals
 }\label{subsec:A-cycle integral}
 
\subsubsection*{Regularized integrals and modularity}

We consider the following action of $ \gamma =\begin{pmatrix}
 a& b\\
 c&d
 \end{pmatrix}  \in \mathrm{SL}_{2}(\mathbb{Z})$ on $\mathbb{C}^{n}\times \H$ ($n\geq 0$) by
 \begin{eqnarray*}
 \gamma:
 \mathbb{C}^{n}\times \H &\rightarrow & \mathbb{C}^{n}\times \H\,,\\
 (z_1,\cdots, z_{n} ;\tau)&\mapsto & (\gamma z_1,\cdots, \gamma z_{n} ; \gamma \tau):=
 ({z_1\over c\tau+d},\cdots ,{z_{n}\over c\tau+d} ;{a\tau+b\over c\tau+d})\,.
   \end{eqnarray*}

\begin{dfn}\label{dfnellipticityandmodularity}
 A function $\Phi(z_1,\cdots, z_{n};\tau)$ on $\mathbb{C}^{n}\times \H$ is \textbf{modular of weight} $k\in \mathbb{Z}$ if
 \begin{equation*}
 \Phi(\gamma z_1,\cdots, \gamma z_{n};\gamma \tau)=(c\tau+d)^{k}\Phi(z_1,\cdots, z_{n};\tau)\,,
 \quad
 \forall \,\gamma \in \mathrm{SL}_{2}(\Z)\,.
 \end{equation*}
 It is said to be \textbf{elliptic} if 
 \begin{equation*}
 \Phi( z_1+\lambda_{1},\cdots,  z_{n}+\lambda_{n}; \tau)=\Phi(z_1,\cdots, z_{n};\tau)\,,
 \quad
 \forall \,(\lambda_{1},\cdots, \lambda_{n})\in  \Lambda_{\tau}^{n}\,.
 \end{equation*}
 \end{dfn}
 
  An elliptic function $\Phi$ defines a function $\Phi(-;\tau)$ on $E_\tau^n$ for generic fixed $\tau$. We will use the same symbol $\Phi$ to denote such a function on $\mathbb{C}^{n}\times \H$ and the induced function on $E_\tau^n$ when the meaning is clear from the context. 

\begin{dfn}\label{dfnholomorphicityawayfromdiagonals} A meromorphic function $\Phi(z_1,\cdots, z_{n};\tau)$ on $\mathbb{C}^{n}\times \H$ is said to  be \textbf{holomorphic away from diagonals} if  the poles of $\Phi$ are contained in the union of all the following divisors 
$$
   \{z_i-z_j-\lambda=0\}\subset \mathbb{C}^{n}\times \H\,, \quad 1\leq i\neq j\leq n\,,\quad{\lambda\in \Lambda_\tau}\,. 
$$
\end{dfn}

 Let $\Phi(z_1,\cdots, z_n;\tau)$ be a meromorphic elliptic function on $\mathbb{C}^{n}\times \H$ which is holomorphic away from diagonals. Then the meromorphic function $\Phi(-;\tau)$ on $E_\tau^n$ has 
possible poles only along the big diagonal
 $$
\Delta= \bigcup_{1\leq i\neq j\leq n} \Delta_{ij}\subset E_\tau^n. 
 $$
So $\Phi(-;\tau)$ defines a holomorphic function on the configuration space of $n$ points on $E_\tau$. 
 
 We are interested in the following regularized integral 
 $$
 \dashint_{E_\tau^n} \bracket{\prod_{i=1}^n {d^2z_i\over \im \tau}} \Phi(z_1,\cdots, z_n;\tau)
 $$
 which is defined by Definition \ref{defn-integral-product}. By Theorem \ref{thm-integral-product}, this integral can be expressed as an iterated regularized integral on $E_\tau$
  $$
 \dashint_{E_\tau^n} \bracket{\prod_{i=1}^n {d^2z_i\over \im \tau}} \Phi(z_1,\cdots, z_n;\tau)=\dashint_{E_\tau} {d^2z_{i_1}\over \im \tau} \cdots \dashint_{E_\tau}  {d^2z_{i_n}\over \im \tau} \Phi(z_1,\cdots, z_n;\tau)\,,
 $$
where $\{i_1,\cdots, i_n\}$ is an arbitrary permutation of $\{1,\cdots, n\}$. By Corollary \ref{cor-Fubini}, its value does not depend on the choice of the ordering for integration,  i.e., the choice of $i_1,\cdots, i_n$.

The main theorem for our application in this section is the following. 

\begin{thm}\label{thm-2d}  Let $\Phi(z_1,\cdots, z_n;\tau)$ be a meromorphic  elliptic function on 
$\mathbb{C}^{n}\times \H$ which is holomorphic away from diagonals.  Then
\begin{itemize}
\item [(1)] The regularized integral 
$$
 \dashint_{E_\tau^n} \bracket{\prod_{i=1}^n {d^2z_i\over \im \tau}}
  \Phi(z_1,\cdots, z_n;\tau)  \quad \text{lies in}\quad \OO_{\H}[{1\over \im \tau}]\,. 
$$
\item[(2)] Let $A_1, \cdots, A_n$ be  $n$ disjoint representatives of the $A$-cycle class on $E_\tau$. Then
$$
\lim_{\bar\tau \to \infty} \dashint_{E_\tau^n} \bracket{\prod_{i=1}^n {d^2z_i\over \im \tau}} \Phi(z_1,\cdots, z_n;\tau)={1\over n!}\sum_{\sigma\in S_n} \int_{A_{1}} dz_{\sigma(1)}\cdots \int_{A_{n}} dz_{\sigma(n)} \Phi(z_1,\cdots, z_n;\tau) \,. 
$$
\item[(3)] If $\Phi$ is modular of weight $m$ on $\mathbb{C}^{n}\times \H$, then $ \dashint_{E_\tau^n} \bracket{\prod\limits_{i=1}^n {d^2z_i\over \im \tau}}
 \Phi$ is modular of weight $m$ on $\H$. 
 
 \end{itemize}

\end{thm}

The proof of Theorem \ref{thm-2d} will be deferred to Section \ref{sec-proof} after we have developed several techniques in the next subsections. 

\subsubsection*{$A$-cycle integrals and quasi-modularity}

\begin{dfn}\label{dfnorderedAcycleintegral}
Let $f(z_1,\cdots, z_k)$ be a meromorphic function on $E_\tau^k$ which is holomorphic when all $z_i$'s are distinct. 
Let $\{i_1,\cdots,i_k\}$ be a permutation of $\{1,2,\cdots,k\}$. We define the  \textbf{ordered $A$-cycle integral} to be the following integral
$$
\int_A dz_{i_1}\cdots \int_A dz_{i_k} \,f:=\int_{A_1} dz_{i_1}\cdots \int_{A_k} dz_{i_k}\, f\,.
$$
Here $A_1, A_2, \cdots, A_k$ are representatives of the $A$-cycle class on the fundamental domain of $E_\tau$  that are ordered and oriented by
$$
A_i=\text{interval from $\epsilon_i \tau$ to $\epsilon_i \tau+ 1$}\,,
$$
where 
\begin{equation*}\label{eqn-orderingchi0}
  0< \epsilon_1 <\epsilon_2<\cdots< \epsilon_k<1\,.
\end{equation*}

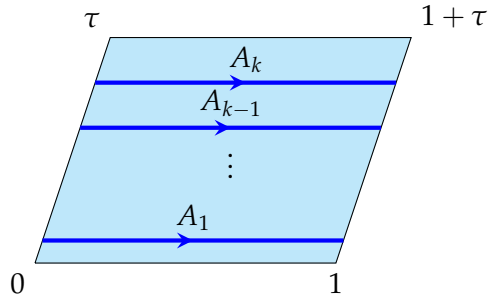
$$
	\begin{tikzpicture}[scale=1]

\draw[cyan!23,fill=cyan!23](0,0)to(4,0)to(5,3)to(1,3)to(0,0);
\draw (0,0) node [below left] {$0$} to (4,0) node [below] {$1$} to (5,3) node [above right] {$1+\tau$} to (1,3) node [above left] {$\tau$} to (0,0);

\draw(1,3)coordinate(a1)(5,3)coordinate(a2) (0,0)coordinate(a3)(4,0)coordinate(a4);

\draw[ultra thick,blue,->-=.5,>=stealth]($(a3)!0.8!(a1)$)to ($(a4)!0.8!(a2)$);
\node [above] at ($($(a3)!0.8!(a1)$)!0.5!($(a4)!0.8!(a2)$)$) {$A_k$};

\draw[ultra thick,blue,->-=.5,>=stealth]($(a3)!0.6!(a1)$)to ($(a4)!0.6!(a2)$);
\node [above] at ($($(a3)!0.6!(a1)$)!0.5!($(a4)!0.6!(a2)$)$) {$A_{k-1}$};
\node [below] at ($($(a3)!0.6!(a1)$)!0.5!($(a4)!0.6!(a2)$)$) {$\vdots$};

\draw[ultra thick,blue,->-=.5,>=stealth]($(a3)!0.1!(a1)$)to ($(a4)!0.1!(a2)$);
\node [above] at ($($(a3)!0.1!(a1)$)!0.5!($(a4)!0.1!(a2)$)$) {$A_1$};



	\end{tikzpicture}
$$
\end{dfn}
\begin{rem}
Since $f$ is meromorphic,  the integral $
\int_{A_1} dz_{i_1}\cdots \int_{A_k} dz_{i_k} \,f$
does not depend on the precise values of $\epsilon_i$'s. So we write the ordered $A$-cycle integral  as $
\int_{A} dz_{i_1}\cdots \int_{A} dz_{i_k} \,f
$ without specifying the locations. It is also invariant under the cyclic permutation
$$
\int_{A} dz_{i_1}\cdots \int_{A} dz_{i_k} \,f =\int_{A} dz_{i_2}\cdots \int_{A} dz_{i_k} \int_A dz_{i_1} \,f\,. 
$$
\end{rem}

Observe that switching the order of an iterated $A$-cycle 
integral
results in a difference related to the residue. This is illustrated by deforming the integration contour as in Fig. \ref{fig:ordering}. 
$$
\int_A dz_2 \int_A dz_1(\cdots)- \int_A dz_1 \int_A dz_2(\cdots)=  \int_A dz_1 \oint_{z_1}dz_2(\cdots)\,. 
$$

\begin{figure}[H]\centering
\begin{tikzpicture}[scale=1]

\draw[cyan!23,fill=cyan!23](0,0)to(4,0)to(5,3)to(1,3)to(0,0);
\draw (0,0)  to (4,0)  to (5,3)  to (1,3)  to (0,0);

\draw(1,3)coordinate(a1)(5,3)coordinate(a2) (0,0)coordinate(a3)(4,0)coordinate(a4);

\node at ($(a4)!0.46!(a1)$) {$\bullet z_1$};

\draw[ultra thick,blue,->-=.5,>=stealth] ($(a4)!0.6!(a2)$) to ($(a3)!0.6!(a1)$);
\node [above] at ($($(a3)!0.6!(a1)$)!0.5!($(a4)!0.6!(a2)$)$) {$\int dz_2$};

\draw[ultra thick,blue,->-=.5,>=stealth]($(a3)!0.4!(a1)$)to ($(a4)!0.4!(a2)$);
\node [above] at ($($(a3)!0.15!(a1)$)!0.5!($(a4)!0.15!(a2)$)$) {$\int dz_2$};

	\end{tikzpicture}
	\begin{tikzpicture}[scale=1]
	\node at (-0.5,1.5) {=};
		
\draw[cyan!23,fill=cyan!23](0,0)to(4,0)to(5,3)to(1,3)to(0,0);
\draw (0,0)  to (4,0)  to (5,3)  to (1,3)  to (0,0);

\draw(1,3)coordinate(a1)(5,3)coordinate(a2) (0,0)coordinate(a3)(4,0)coordinate(a4);

\draw [ultra thick,blue,->-=.5,>=stealth] ($(a4)!0.5!(a1)$)  circle [radius=0.5];

\node at ($(a4)!0.46!(a1)$) {$\bullet z_1$};

\node [above] at ($($(a3)!0.6!(a1)$)!0.5!($(a4)!0.6!(a2)$)$) {$\oint dz_2$};
	\end{tikzpicture}
	\caption{Commutator of ordered $A$-cycle integrations.}\label{fig:ordering}
\end{figure}
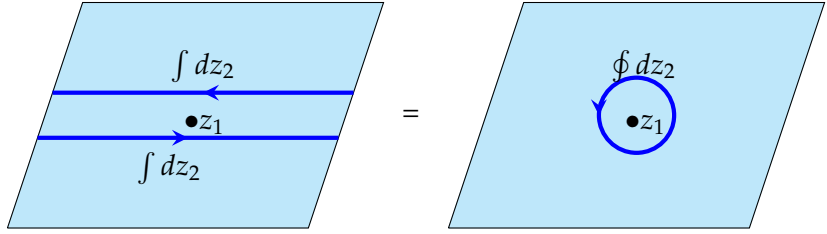

Therefore we find the following identity for two different integral operations 
$$
 \int_A dz_1 \oint_{z_1}dz_2=\bbracket{\int_A dz_2, \int_A dz_1}
$$
where the right hand side is the commutator of two ordered $A$-cycle integrals defined above. 

The same consideration proves the following lemma that expresses commutators of ordered $A$-cycle integrals in terms of residue operations. 
\begin{lem}\label{lem-A-commutator} Let $f(z_1,\cdots, z_k)$ be a meromorphic function on $E_\tau^k$ which is holomorphic when all $z_i$'s are distinct. Then
$$
 \int_A dz_1 \oint_{z_1}dz_2 \oint_{z_1} dz_3 \cdots \oint_{z_1}dz_k f(z_1,\cdots, z_k)=\bbracket{\int_A dz_2, \bbracket{\int_A dz_{3},\cdots, \bbracket{\int_A dz_k, \int_A dz_1}}} f(z_1,\cdots, z_k).
$$
\end{lem}

Let $\Phi$ be a function as in Theorem \ref{thm-2d}. Then Theorem \ref{thm-2d} can be equivalently described in terms of the ordered $A$-cycle integrals as 
$$
\lim_{\bar\tau \to \infty} \dashint_{E_\tau^n} \bracket{\prod_{i=1}^n {d^2z_i\over \im \tau}} \Phi={1\over n!} \sum_{\sigma\in S_n}\int_A dz_{\sigma(1)}\cdots \int_A dz_{\sigma(n)}\Phi. 
$$
This actually leads to a similar description for any ordered $A$-cycle integral as follows.

Let $V=\R^n$ be a real vector space of dimension $n$. Let 
$$
T(V)=\bigoplus\limits_{k\geq 0} V^{\otimes k}
$$
denote the tensor algebra over $\R$. Let $\Lie(V)$ denote the free Lie algebra over $\R$ generated by $V$. We equip $T(V)$ with the Lie algebra structure whose Lie bracket is the commutator
$$
[a,b]:= a\otimes b- b\otimes a, \quad a, b\in T(V)\,. 
$$
Then $\Lie(V)$ can be viewed as the Lie subalgebra of $T(V)$ generated by $V$ and the bracket operation $[-,-]$. It is a classical result that the tensor algebra $T(V)$
$$
T(V)=\mathcal U(\Lie(V))
$$
is the universal enveloping algebra of $\Lie(V)$.  By the Poincar\'{e}-Birkhoff-Witt Theorem, we can identify
$$
T(V)=\mathcal S(\Lie(V))
$$
as vector spaces, where $\mathcal S$ refers to symmetric tensors. This identity gives a natural way to connect any ordered $A$-cycle integrals to regularized integrals. 

Explicitly, let $x_1, \cdots, x_n$ denote a basis of $V$. Then $T(V)$ can be identified as the free $\R$-algebra generated by $x_i$'s
$$
T(V)=\R\abracket{x_1, \cdots, x_n}\,. 
$$
Let $\{Y^{(s)}_I\}_{I,s}$ be a basis of $\Lie(V)$ such that each $Y_I^{(s)}$ is of the form
$$
  [x_{i_1},[x_{i_2},\cdots,[x_{i_{s-1}},x_{i_{s}}] \cdots ]]\,,\quad
  I=(i_{1},i_{2},\cdots, i_{s})\,. 
$$
 
Using $T(V)=\mathcal S(\Lie(V))$, we can write
$$
x_1x_2\cdots x_n= \sum_{k=1}^n R^{(k)}_n(Y).  
$$ 
Here $R^{(k)}_n(Y)$ is a degree-$k$ polynomial (viewed as symmetric tensor) in $Y_I^{(\leq n)}$'s.
Explicit formula can be obtained from the result in \cite{Solomon:1968} applied to $\Lie(V)$. 

\begin{ex}\label{ex-n=3} Here is the example when $n=3$ (see \cite{Solomon:1968}). We choose 
\begin{align*}
   Y^{(1)}&=\{x_1, x_2, x_3\}\,,\\
   Y^{(2)}&=\{[x_1,x_2],[x_1,x_3], [x_2,x_3]\}\,,\\
   Y^{(3)}&=\{[x_1,[x_2,x_3]],[x_2,[x_1, x_3]]\}\,. 
\end{align*}
Then  
\begin{align*}
x_1x_2x_3=&{1\over 6}\bracket{x_1x_2x_3+x_1x_3x_2+x_2x_1x_3+x_2x_3x_1+x_3x_1x_2+x_3x_2x_1}\\
&+{1\over 4}\bracket{x_1[x_2,x_3]+[x_2,x_3]x_1}+{1\over 4}\bracket{x_2[x_1,x_3]+[x_1,x_3]x_2}+{1\over 4}\bracket{x_3[x_1,x_2]+[x_1,x_2]x_3}\\
&+{1\over 3}[x_1,[x_2, x_3]]-{1\over 6}[x_2,[x_1,x_3]]\,. 
\end{align*}
We have
\begin{align*}
R^{(3)}_3&=Y^{(1)}_1 Y^{(1)}_2  Y^{(1)}_3 \,,\\
R^{(2)}_3&={1\over 2}Y^{(1)}_1Y^{(2)}_{23}+{1\over 2} Y^{(1)}_2 Y^{(2)}_{13}+{1\over 2} Y^{(1)}_3 Y^{(2)}_{12}\,, \\
R^{(1)}_3&={1\over 3}Y^{(3)}_{123} -{1\over 6}Y^{(3)}_{213}\,. 
\end{align*}
\end{ex}

For each 
$$
Y^{(s)}_I=  [x_{i_1},[x_{i_2},\cdots,[x_{i_{s-1}},x_{i_{s}}] \cdots ]]\,,
$$
we associate an operation 
$$
\oint_{Y^{(s)}_I}:=
\begin{cases}
\oint_{z_{i_s}} dz_{i_1} \oint_{z_{i_s}} dz_{i_2}\cdots \oint_{z_{i_s}} dz_{i_{s-1}} & \quad I=(i_{1},i_{2},\cdots, i_{s})\,,\quad s\geq 2
\,,\\
\text{identity operator} & \quad I=(i)\,,\quad s=1\,.
\end{cases}
$$
Then for each $R^{(k)}_n(Y)$ that appears in the above decomposition for $x_{1}x_2\cdots x_{n}$, we write
$$
R^{(k)}_n(\oint_{Y})
$$
for the operation that replaces each $Y^{(s)}_I$ by $\oint_{Y^{(s)}_I}$ in $R^{(k)}_n$. 
 Therefore, one can write
 $$R_n^{(k)}(\oint_{Y}) \Phi=\sum\limits_{I=(i_1,\cdots, i_k)}\phi^{(k)}_I(z_{i_1},\cdots,z_{i_k};\tau)$$
 for some functions $\phi^{(k)}_{I}$.
 
 \begin{thm}\label{thm-1d}  Let $\Phi(z_1,\cdots, z_n;\tau)$ be a  meromorphic elliptic function on 
$\mathbb{C}^{n}\times \H$ which is holomorphic away from diagonals.  Let 
$$
R_n^{(k)}(\oint_{Y}) \Phi=\sum\limits_{I=(i_1,\cdots, i_k)}\phi^{(k)}_I(z_{i_1},\cdots,z_{i_k};\tau)\,. 
$$
\begin{itemize}
\item [(1)] The ordered $A$-cycle integral is given by the holomorphic limit
$$
\int_{A} dz_1\cdots \int_{A} dz_n \Phi(z_1,\cdots, z_n;\tau)=\lim_{\bar\tau \to \infty}\sum_{k=1}^n \sum\limits_{I=(i_1,\cdots, i_k)}\dashint_{E_\tau^k} \bracket{\prod_{j=1}^k {d^2 z_{i_j}\over \im \tau}} \phi^{(k)}_I(z_{i_1},\cdots,z_{i_k};\tau).
$$
\item [(2)] If $\Phi$ is modular of weight $m$ on $\mathbb{C}^{n}\times \H$, then each
$$
\dashint_{E_\tau^k} \bracket{\prod_{j=1}^k {d^2 z_{i_j}\over \im \tau}} \phi^{(k)}_I(z_{i_1},\cdots,z_{i_k};\tau)
$$
is modular of weight $m+k-n$. In particular, the ordered $A$-cycle integral
$$
\int_{A} dz_{1}\cdots \int_{A} dz_{n} \Phi(z_1,\cdots, z_n;\tau)
$$
is quasi-modular of mixed weight with each weight $\leq m$, and the leading weight-$m$ component is 
$\lim\limits_{\bar\tau \to \infty}
 \dashint_{E_\tau^n} \bracket{\prod\limits_{i=1}^n {d^2z_i\over \im \tau}}
  \Phi. 
$
\end{itemize}
\end{thm}

\begin{rem}\label{remotherorderings}  If $\Phi$ is modular of weight $m$ on $\mathbb{C}^{n}\times \H$,  then for any $\sigma\in S_n$, 
$$
\int_{A} dz_{\sigma(1)}\cdots \int_{A} dz_{\sigma(n)} \Phi(z_1,\cdots, z_n;\tau)
$$
is quasi-modular of mixed weight  with each weight $\leq m$ \cite{Goujard:2016counting, Oberdieck:2018}. This follows by applying Theorem \ref{thm-1d} to $\Phi_{\sigma}(z_1, \cdots, z_n;\tau):=\Phi(z_{\sigma^{-1}(1)},\cdots, z_{\sigma^{-1}(n)};\tau)$. Theorem \ref{thm-2d} says that
averaging all such ordered $A$-cycle integrals leads to cancellation of all lower-weight components.  Such cancellation phenomenon was also proved in \cite{Oberdieck:2018} using a different method. 
\end{rem}

\begin{proof}[Proof of Theorem \ref{thm-1d}] The algebraic identity $x_1x_2\cdots x_n= \sum\limits_{k=1}^n R^{(k)}_n(Y)$ together with Lemma \ref{lem-A-commutator} and Theorem \ref{thm-2d}  imply
$$
\int_{A} dz_1\cdots \int_{A} dz_n \Phi(z_1,\cdots, z_n;\tau)=\lim_{\bar\tau \to \infty}\sum_{k=1}^n \sum\limits_{I=(i_1,\cdots, i_k)}\dashint_{E_\tau^k} \bracket{\prod_{j=1}^k {d^2 z_{i_j}\over \im \tau}} \phi^{(k)}_I(z_{i_1},\cdots,z_{i_k};\tau)\,. 
$$

Assume $\Phi$ is modular of weight $m$ on $\mathbb{C}^{n}\times \H$. Each $\phi^{(k)}_I$ is obtained from $\Phi$ by applying residue  $(n-k)$ times. Since each residue map along a diagonal decreases the weight\footnote{This can be seen from the modular transformation on the variables $z_1,\cdots,z_n$ in Definition \ref{dfnellipticityandmodularity}.} by $1$, $\phi^{(k)}_I$ is modular of weight $m+k-n$. By Theorem \ref{thm-2d}, 
$$
\dashint_{E_\tau^k} \bracket{\prod_{j=1}^k {d^2 z_{i_j}\over \im \tau}} \phi^{(k)}_I(z_{i_1},\cdots,z_{i_k};\tau)
$$
is modular of weight $m+k-n$. 


\end{proof}
\begin{ex}
\label{exn=3mixedweight}
 In the case $n=3$ as in Example \ref{ex-n=3}, Theorem \ref{thm-1d} implies
\begin{align*}
&\int_A dz_1 \int_A dz_2 \int_A dz_3 \Phi=
\lim_{\bar\tau\to \infty}\left\{\dashint_{E_\tau}{d^2z_1\over \im \tau} \dashint_{E_\tau}{d^2z_2\over \im \tau} \dashint_{E_\tau}{d^2z_3\over \im \tau} \Phi\right.\\
+&{1\over 2}\dashint_{E_\tau}{d^2z_1\over \im \tau} \dashint_{E_\tau}{d^2z_3\over \im \tau} \oint_{z_3} dz_2 \Phi+{1\over 2}\dashint_{E_\tau}{d^2z_2\over \im \tau} \dashint_{E_\tau}{d^2z_3\over \im \tau} \oint_{z_3} dz_1 \Phi+{1\over 2}\dashint_{E_\tau}{d^2z_2\over \im \tau} \dashint_{E_\tau}{d^2z_3\over \im \tau} \oint_{z_2} dz_1 \Phi\\
&\left.+ {1\over 3}\dashint_{E_\tau}{d^2z_3\over \im \tau} \oint_{z_3}dz_1 \oint_{z_3}dz_2 \Phi-{1\over 6}\dashint_{E_\tau}{d^2z_3\over \im \tau} \oint_{z_3}dz_2 \oint_{z_3}dz_1 \Phi \right\}. 
\end{align*}

\end{ex}

\begin{ex}\label{exmixedweight}

Consider the  $n=4$ case, with 
$$\Phi(z_1,z_2,z_3,z_4;\tau)=\wp(z_{1}-z_{2};\tau)\wp(z_{2}-z_{3};\tau)\wp(z_{3}-z_{4};\tau)\wp(z_{4}-z_{1};\tau)\,.$$
The function $\Phi$ is invariant under the dihedral group action on the $4$ arguments $z_1,z_2,z_3,z_4$. Hence among the $4!=24$
ordered $A$-cycles integrals it suffices to consider
the following $3$ integrals:
$$\int_{A}dz_{4}\int_{A}dz_{3}\int_{A}dz_{2}\int_{A}dz_{1}\,\Phi\,,
\quad
\int_{A}dz_{3}\int_{A}dz_{4}\int_{A}dz_{2}\int_{A}dz_{1}\,\Phi\,,
\quad
\int_{A}dz_{4}\int_{A}dz_{2}\int_{A}dz_{3}\int_{A}dz_{1}\,\Phi\,.$$
Following the method outlined in Remark
\ref{remcomputationontatecurve}, we obtain
\begin{eqnarray*}
\int_{A}dz_{4}\int_{A}dz_{3}\int_{A}dz_{2}\int_{A}dz_{1}\,\Phi&=&(2\pi i)^{8}\sum_{k\geq 1}k^{4}{q^{k}+q^{3k}\over (1-q^{k})^4}+({\pi^{2}\over 3}E_{2})^{4}\,,
\\
\int_{A}dz_{3}\int_{A}dz_{4}\int_{A}dz_{2}\int_{A}dz_{1}\,\Phi
&=&(2\pi i)^{8}\sum_{k\geq 1}k^{4}{2q^{2k}\over (1-q^{k})^4}+({\pi^{2}\over 3}E_{2})^{4}\,,\\
\int_{A}dz_{4}\int_{A}dz_{2}\int_{A}dz_{3}\int_{A}dz_{1}\,\Phi&=&(2\pi i)^{8}\sum_{k\geq 1}k^{4}{2q^{2k}\over (1-q^{k})^4}+({\pi^{2}\over 3}E_{2})^{4}\,.
\end{eqnarray*}
On the other hand, 
using \eqref{eqn-FourierexpansionsE2k}, \eqref{eqn-Ramanujanidentities}, we obtain
\begin{eqnarray*}
\sum_{k\geq 1} k^{4} {q^{k}+4q^{2k}+q^{3k}\over (1-q^{k})^{4}}&=&-{1\over 24}E_{2}'''={1\over 2^{7}3^{3}}(3E_{2}^{2}E_{4}-4E_{2}E_{6}+E_{4}^2)
\,,\\
\sum_{k\geq 1} k^{4} {q^{k}\over (1-q^{k})^{2}}
&=&{1\over 240}E_{4}'={1\over 2^{4} 3^{2}5}(E_{4}E_{2}-E_{6})\,,
\end{eqnarray*}
where $'={1\over 2\pi i}{\partial\over \partial \tau}$.
Combining the above two sets of relations, we see that the above $3$ ordered $A$-cycle integrals
are holomorphic quasi-modular forms of mixed weight
$$\int_{A}dz_{4}\int_{A}dz_{3}\int_{A}dz_{2}\int_{A}dz_{1}\,\Phi=
{\pi^{8}\over 3^{4}}
E_{2}^{4} 
+2^{8}\pi^{8}\left( {1\over 3}\cdot -{1\over 24}E_{2}'''+{2\over 3}\cdot {1\over 240}E_{4}'\right)
\,,
$$
$$\int_{A}dz_{3}\int_{A}dz_{4}\int_{A}dz_{2}\int_{A}dz_{1}\,\Phi=
{\pi^{8}\over 3^{4}}
E_{2}^{4}+2^{8}\pi^{8}\left({1\over 3}\cdot -{1\over 24}E_{2}'''-{1\over 3}\cdot {1\over 240}E_{4}'\right)\,,
$$
$$\int_{A}dz_{4}\int_{A}dz_{2}\int_{A}dz_{3}\int_{A}dz_{1}\,\Phi=
{\pi^{8}\over 3^{4}}
E_{2}^{4}+2^{8}\pi^{8}\left({1\over 3}\cdot -{1\over 24}E_{2}'''-{1\over 3}\cdot {1\over 240}E_{4}'\right)\,.
$$
It follows that
\begin{eqnarray*}
{1\over 4!}\sum_{\sigma\in S_4} \int_{A}dz_{\sigma(1)}
\cdots \int_{A} dz_{\sigma(4)} \Phi
&=&
{\pi^{8}\over 3^{4}}E_{2}^{4}
+2^{8}\pi^{8}{1\over 3}\cdot -{1\over 24}E_{2}'''\\
&=&
{\pi^{8}\over 3^{4}}E_{2}^{4}+{2\pi^{8}\over 3^{4}} (3 E_{2}^{2}E_{4} -4 E_{2}E_{6} +  E_{4}^{2})\,.
\end{eqnarray*}

It is straightforward to compute the iterated residues.
For example, we have (here $'=\partial_{z}$)
\begin{eqnarray*}
\Res_{z_{2}=z_{3}}\Res_{z_{1}=z_{2}}\Phi&=&
\wp(z_{3}-z_{4}) \wp''(z_{3}-z_{4})\,,\\
\Res_{z_{2}=z_{4}}\Res_{z_{1}=z_{2}}\Phi&=&-\wp(z_{3}-z_{4}) \wp''(z_{3}-z_{4})
\,,\\
\Res_{z_{3}=z_{4}}\Res_{z_{1}=z_{2}}
\Phi
&=&\wp'(z_{2}-z_{4})\wp'(z_{4}-z_{2})\,,\\
\Res_{z_{4}=z_{3}}\Res_{z_{1}=z_{2}}\Phi
&=&-\wp''(z_{3}-z_{2})\wp(z_{2}-z_{3})\,.
\end{eqnarray*}
The result
$$\int_{A}\wp^{3}dz=-{1\over 15} \pi^{6} E_{2}E_{4}+{2^{2}\over 3^{3} 5} \pi^{6}E_{6}
$$
from Example \ref{expmoments}  implies that
$$
\int_{A}\wp \wp'' dz
=-\int_{A}\wp'\wp' dz
={2^{3}\over 3^{2}5}\pi^6 (-E_{4}E_{2}+E_{6})\,.
$$
By Proposition \ref{prop-push-poles}, $\Res\Res\Res\Res \Phi=0$.
By the pure-weight reason in Theorem \ref{thm-modularity}, the 
iterated regularized integrals of $\Res \Phi$ and $\Res\Res\Res \Phi$
are almost-holomorphic modular forms of odd weight and hence
vanish. Both of these two claims can be confirmed
in the current example from direct computations.
Applying Theorem \ref{thm-1d} to the ordered $A$-cycle integral
$\int_{A}dz_{4}\int_{A}dz_{3}\int_{A}dz_{2}\int_{A}dz_{1}\,\Phi$,
a tedious calculation 
shows that the holomorphic limits of the $\Res\Res\Phi$ terms in Theorem \ref{thm-1d} combine to 
$${2^{5}\over 3^{2} 5}\pi^{8}E_{4}'\,.$$
These match the above results for 
$\int_{A}dz_{4}\int_{A}dz_{3}\int_{A}dz_{2}\int_{A}dz_{1}\,\Phi$
and 
${1\over 4!}\sum\limits_{\sigma\in S_4} \int_{A}dz_{\sigma(1)}
\cdots \int_{A} dz_{\sigma(4)} \Phi$.

\end{ex}

\subsection{Modularity of regularized integrals}

In this subsection we establish statement (1) and (3) of Theorem \ref{thm-2d}. 

\begin{dfn}\label{defn-almost-meromorphic}  We say a function $\Psi$ on $\mathbb{C}^{n}\times \H$ is \textbf{almost-meromorphic} if $\Psi$ 
can be written as a finite sum
$$
  \Psi(z_1, \cdots, z_n;\tau)=\sum_{k_1,\cdots,k_n,m\geq 0} \Psi_{k_1,\cdots,k_n;m}(z_1, \cdots, z_n;\tau) \bracket{\im z_1\over \im \tau}^{k_1}\cdots \bracket{\im z_n\over \im \tau}^{k_n} \bracket{1\over \im \tau}^{m}\,,
$$
where each $\Psi_{k_1,\cdots,k_n;m}(z_1, \cdots, z;\tau)$ is a meromorphic function on $\mathbb{C}^{n}\times \H$. 
\end{dfn}

\begin{dfn}\label{dfnalmostmeromorphicelliptic}
 Let $\RE_n$ denote the space of functions $\Psi$ on $\mathbb{C}^{n}\times \H$ such that
\begin{itemize}
\item $\Psi$ is elliptic and almost-meromorphic. 
\item Each component $\Psi_{k_1,\cdots,k_n;m}$ as in Definition \ref{defn-almost-meromorphic} is holomorphic away from diagonals. 
\end{itemize}
\end{dfn}


Let $\Psi \in \RE_n$. Then $\Psi(-;\tau)$ defines a function on $E_\tau^n$ with  possible poles only along all the diagonals of $E_\tau^n$. 
We consider the following regularized integral 
$$
\dashint_{E_\tau} {d^2 z_n\over \im \tau} \Psi
$$
which is a well-defined function on $\mathbb{C}^{n-1}\times \H$ by Proposition \ref{prop-push-poles}. 

\begin{prop}\label{prop-iteration} The regularized integral defines a map 
$$
\dashint_{E_\tau} {d^2 z_n\over \im \tau}: \RE_n\to \RE_{n-1}
\,. 
$$
If  $\Psi \in \RE_n$ is modular of weight $k$, then $\dashint_{E_\tau} {d^2 z_n\over \im \tau}\Psi$ is also modular of weight $k$. 

\end{prop}
\begin{proof} Let $\Psi \in \RE_n$. It is clear that $\dashint_{E_\tau} {d^2 z_n\over \im \tau}\Psi$ is elliptic. By Proposition \ref{prop-push-poles}, it also has the required location of poles. We need to show that $\dashint_{E_\tau} {d^2 z_n\over \im \tau}\Psi$ is almost-meromorphic. Let
$$
f(z_1,\cdots, z_{n-1};\tau)=\dashint_{E_\tau} {d^2 z_n\over \im \tau}\Psi(z_1, \cdots, z_n;\tau)\,. 
$$

Given $z_1, \cdots, z_{n-1}$, we choose a parallelogram $\square_c$ in $\C$ with vertices $\{c,c+1,c+1+\tau, c+\tau\}$ such that the poles $D$ of $\Psi(z_1,\cdots,z_{n-1},-;\tau)$ as a function of $z_n$ do not lie on the boundary of $\square_c$. Let $A_c$ denote the interval from $c+\tau$ to $c+1+\tau$, and $B_c$ denote the interval from $c+1$ to $c+1+\tau$, as illustrated in
Fig. \ref{fig:parallelogramsquarec} below. 

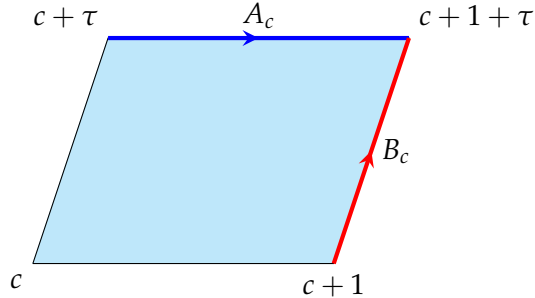
\begin{figure}[H]\centering
	\begin{tikzpicture}[scale=1]

\draw[cyan!23,fill=cyan!23](0,0)to(4,0)to(5,3)to(1,3)to(0,0);
\draw (0,0) node [below left] {$c$} to (4,0) node [below] {$c+1$} to (5,3) node [above right] {$c+1+\tau$} to (1,3) node [above left] {$c+\tau$} to (0,0);

\draw[ultra thick,blue,->-=.5,>=stealth](1,3)to(5,3);

\node [above] at (3,3) {$A_c$};

\draw[ultra thick,red,->-=.5,>=stealth](4,0)to(5,3);

\node [right] at (4.5,1.5) {$B_c$};

	\end{tikzpicture}
	\caption{Parallelogram $\square_c$.}\label{fig:parallelogramsquarec}
\end{figure}

Then we have 
$$
f(z_1,\cdots, z_{n-1};\tau)=\dashint_{\square_c} {d^2 z_n\over \im \tau}\Psi(z_1, \cdots, z_n;\tau)\,,
$$
and its value does not depend on the choice of $c$ by the translation invariance of the regularized integral
as shown in Proposition \ref{prop-pullback}.   Let
$$
  \Psi=\sum_{k} \Psi_{k}\bracket{\im z_n\over \im \tau}^{k}\,,
$$
where $\Psi_{k}$ is meromorphic in $z_n$ and almost-meromorphic in $\{z_1,\cdots, z_{n-1},\tau\}$. Using Theorem \ref{thm-de-Rham}, we find
\begin{align*}
&\dashint_{\square_c} {d^2 z_n\over \im \tau}\Psi(z_1, \cdots, z_n;\tau)
=\dashint_{\square_c} {dz_n\wedge d(\im z_n)\over \im \tau}\Psi(z_1, \cdots, z_n;\tau)\\
=&- \dashint_{\square_c} d\bracket{  dz_n \sum_{k}  {\Psi_{k}\over k+1}\bracket{\im z_n\over \im \tau}^{k+1}}\\
=&\int_{A_c}  dz_n \sum_{k}  {\Psi_{k}(-,z_n;\tau)\over k+1}\bracket{\im z_n\over \im \tau}^{k+1}-\int_{A_c}  dz_n \sum_{k}  {\Psi_{k}(-,z_n-\tau;\tau)\over k+1}\bracket{\im (z_n-\tau)\over \im \tau}^{k+1}\\
& + \sum_{p\in D} \oint_p dz_n \sum_{k}  {\Psi_{k}\over k+1}\bracket{\im z_n\over \im \tau}^{k+1}\\
=&\int_{A_c}  dz_n \sum_{k}  {\Psi_{k}(-,z_n;\tau)\over k+1}\bracket{\bracket{\im (c+\tau)\over \im \tau}^{k+1}-\bracket{\im c \over \im \tau}^{k+1}}
 + \sum_{p\in D} \oint_p dz_n \sum_{k}  {\Psi_{k}\over k+1}\bracket{\im z_n\over \im \tau}^{k+1}. 
\end{align*}
Here the $B$-cycle integration is cancelled out by the periodicity of $\Psi$ under $z_n\mapsto z_n+1$. 

Observe that at a pole $p\in D$, 
\begin{align*}
\oint_p dz_n \Psi_k\bracket{\im (z_n-p)+ \im p \over \im \tau}^{k+1}&= \sum_{a+b=k+1}\binom{k+1}{a}\oint_p dz_n \Psi_k\bracket{\im (z_n-p) \over \im \tau}^{a}\bracket{ \im p \over \im \tau}^{b}\\
&= \sum_{a+b=k+1}\binom{k+1}{a}\bracket{ \im p \over \im \tau}^{b}\oint_p dz_n \Psi_k\bracket{ z_n-p \over 2i \im \tau}^{a}.
\end{align*}
Here in the last step we have used Proposition \ref{prop-residue}. Since all the poles $p$ inside $\square_c$ are of the form $p= z_i+\lambda$  for some $\lambda\in \Lambda_\tau$, ${\im p\over \im \tau}$ has the form ${\im z_i\over \im \tau}+{\im \lambda \over \im \tau}$. It follows from the above expression that $f(z_1,\cdots, z_{n-1};\tau)$ is almost-meromorphic. 

Now let us assume $\Psi$ is modular of weight $k$ and $\gamma\in \mathrm{SL}_{2}(\mathbb{Z})$. Let $\gamma \square_c$ be the image of $\square_c$ under the $\gamma$-action. Then $\gamma \square_c$ is a fundamental domain for $\Psi(\gamma z_1,\cdots,\gamma z_{n-1},z_n;\gamma \tau)$ regarded as a function of $z_n$. Therefore
$$
f(\gamma z_1,\cdots, \gamma z_{n-1};\gamma \tau)=\dashint_{\gamma \square_c} {d^2 z_n\over \im (\gamma \tau)}\Psi(\gamma z_1, \cdots, \gamma z_{n-1},z_n;\gamma \tau).
$$
Using Proposition \ref{prop-pullback} and the modularity of $\Psi$, this is equal to
\begin{align*}
&\dashint_{\square_c} {d^2 (\gamma z_n)\over \im (\gamma \tau)}\Psi(\gamma z_1, \cdots, \gamma z_{n-1},\gamma z_n;\gamma \tau)=\dashint_{\square_c} {d^2 z_n\over \im \tau}\Psi(\gamma z_1, \cdots, \gamma z_{n-1},\gamma z_n;\gamma \tau)\\
=&(c\tau+d)^{k} \dashint_{\square_c} {d^2 z_n\over \im \tau}\Psi(z_1, \cdots, z_n;\tau)=(c\tau+d)^{k} f(z_1,\cdots,z_{n-1};\tau). 
\end{align*}
\end{proof}

 \begin{thm}\label{thm-modularity}
 Let $\Phi(z_1,\cdots, z_n;\tau)$ be a meromorphic elliptic  function on 
$\mathbb{C}^{n}\times \H$ which is holomorphic away from diagonals and modular of weight $k$.
Then
$$
\dashint_{E_\tau^n} \bracket{\prod_{i=1}^n {d^2z_i\over \im \tau}} \Phi
$$
is modular of weight $k$ and almost-holomorphic on $\H$.
 Its holomorphic limit 
 $$
\lim_{\bar\tau\to \infty} \dashint_{E_\tau^n} \bracket{\prod_{i=1}^n {d^2z_i\over \im \tau}} \Phi
 $$
 is quasi-modular of weight $k$ and holomorphic on $\H$.
 \end{thm}

\begin{proof} By Theorem \ref{thm-integral-product}, we have 
$$
\dashint_{E_\tau^n} \bracket{\prod_{i=1}^n {d^2z_i\over \im \tau}} \Phi = \dashint_{E_\tau}{d^2 z_1\over \im \tau}\cdots \dashint_{E_\tau}{d^2 z_n\over \im \tau} \Phi. 
$$
The modularity follows by applying Proposition \ref{prop-iteration} $n$ times. 
The last statement on the quasi-modularity follows from a general fact
about holomorphic limit in the theory of modular forms \cite{Kaneko:1995}.
\end{proof}
 \begin{rem}
 In the language of the theory of modular forms,
 the holomorphic limits of such iterated regularized integrals are called 
 weakly holomorphic quasi-modular forms, see Definition \ref{dfnquasimodularform}.
 \end{rem}

 \subsection{Regularized Feynman graph integrals}
 \label{subsecgraphintegrals}

One of the main motivation of this paper is to develop analytic methods for 2d chiral quantum field theories. Perturbative correlation functions of such theories are given by  sums of Feynman graph integrals, which are integrals on product of Riemann surfaces of differential forms with holomorphic poles on the big diagonal. When the theory is put on elliptic curves, Theorem \ref{thm-2d}  provides a powerful tool both for theoretical constructions and for practical computations. We illustrate this by the example of chiral boson. The same method applies to other theories such as chiral $bc$-systems and chiral $\beta\gamma$-systems. 

Let $\wp(z;\tau)$ be the Weierstrass $\wp$-function
$$
\wp(z;\tau)={1\over z^2}+\sum_{\lambda\in \Lambda_{\tau}-(0,0)} \bracket{{1\over (z+\lambda)^2} -{1\over \lambda^2}}. 
$$
Let 
$$
\widehat{P}(z_1,z_2;\tau,\bar\tau):=\wp(z_1-z_2;\tau)+{\pi^2\over 3}\widehat{E}_{2}(\tau,\bar{\tau}), \quad 
\widehat{E}_{2}(\tau,\bar{\tau})=E_{2}(\tau)-{3\over \pi }{1\over \im \tau}\,,
$$
where $E_2$ is the 2nd Eisenstein series. 
See Appendix
\ref{secmodularformsellipticfunctions} for more details.
Here we have specified the $\bar\tau$-dependence in $\widehat{P}(z_1,z_2;\tau,\bar\tau)$ which defines an elliptic function on $\C^2\times \H$ and is modular of weight $2$. As a meromorphic function on $E_\tau^2$, $\widehat{P}(-;\tau)$ has order 2 pole along the diagonal.

\begin{rem} Let $\phi$ be the field of free chiral boson. Then $\widehat{P}$ is the two-point function on $E_\tau$ (see e.g., \cite{Douglas:1995conformal,Dijkgraaf:1997chiral})
$$
\widehat{P}(z_1,z_2;\tau,\bar\tau)=\abracket{\pa \phi(z_1) \pa \phi(z_2)}_{E_\tau}. 
$$
Mathematically it is known as the Schiffer kernel \cite{Tyurin:1978, Takhtajan:2001} and is essentially given by the 2nd derivative of the Green's function associated to the flat metric on $E_{\tau}$. Its holomorphic limit gives the Bergman kernel associated to our canonical marking $\{A,B\}$ on $E_{\tau}$.
\end{rem}

Let $\Gamma$ be an oriented\footnote{The data of orientation on the graph is not strictly necessary for the Feynman graph integrals here, since
 $\widehat{P}$ is an even function. We however reserve this notion for potential generalizations such
as chiral $bc$-systems and chiral $\beta\gamma$-systems.} graph with no self-loops. Let $E(\Gamma)$ be its set of edges, and $V(\Gamma)$ be
its set of vertices with cardinality $n=|V(\Gamma)|$. We label the vertices by fixing an identification 
$$
V(\Gamma)\to \{1,2,\cdots, n\}\,. 
$$

The Feynman rule assigns to the graph $\Gamma$ a quantity
$$\Phi_{\Gamma}(z_1,\cdots,z_n;\tau, \bar\tau):=\prod_{e\in E(\Gamma)}\widehat{P}(z_{t(e)},z_{h(e)};\tau,\bar\tau) \,,
$$
where $h(e)$ is the head of the edge $e$ and $t(e)$ the tail. It is clear that $\Phi_{\Gamma}$ can be written as
$$
\Phi_{\Gamma}(z_1,\cdots,z_n;\tau, \bar\tau)=\sum_{m=0}^{|E(V)|} {\Phi_{\Gamma, m}(z_1,\cdots,z_n;\tau)\over (\im \tau)^m}\,,
$$
where $\Phi_{\Gamma, m}(z_1,\cdots,z_n;\tau)$'s are meromorphic functions on $\C^n\times \H$. Let us denote
$$
\lim_{\bar\tau\to \infty} \Phi_{\Gamma}:=\Phi_{\Gamma, 0}\,. 
$$
\begin{dfn}\label{dfn-graph}
We define the regularized Feynman graph integral $\widehat{I}_{\Gamma}$ for $\Gamma$ to be
$$\widehat{I}_{\Gamma}:=\dashint_{E_{\tau}^{n}}  \bracket{\prod_{i=1}^n {d^2z_i\over \im \tau}}
  \Phi_{\Gamma}(z_1,\cdots, z_n;\tau,\bar{\tau}) \,. 
  $$
By Corollary \ref{cor-Fubini}, $\widehat{I}_{\Gamma}$ does not depend the choice of the labeling on $V(\Gamma)$. 
\end{dfn}

The next lemma shows that these graph integrals satisfy a regularity condition at  $\tau=i\infty$.

 \begin{lem}[Regularity]\label{lem-regularity}Let  $f(x_{ij},y_{ij})\in \C[x_{ij},y_{ij}]_{1\leq i<j\leq n}$ be a polynomial. Let
 $$
 \Phi(z_1,\cdots, z_n;\tau)=f(\wp(z_{i}-z_{j};\tau),{1\over 2\pi i}\wp'(z_{i}-z_{j};\tau))\,.
 $$
 Then for any $\sigma\in S_{n}$ and disjoint representatives $A_1, \cdots, A_n$ of $A$-cycles, the $A$-cycle integral
$$
\int_{A_1}dz_{\sigma(1)}\cdots \int_{A_n}dz_{\sigma(n)} \,\Phi(z_1,\cdots, z_n;\tau)
$$
is holomorphic on $\H$ and extends to $\tau=i\infty$ by
$$
\lim_{\tau\to i\infty}\int_{A_1}dz_{\sigma(1)}\cdots \int_{A_n}dz_{\sigma(n)} \,\Phi=f(x_{ij}=-{\pi^2\over 3},y_{ij}=0)\,.
$$

 \end{lem}

\begin{proof}
By assumption,  $\Phi$ is a meromorphic elliptic function on $\mathbb{C}^{n}\times \H$ which is holomorphic away from diagonals. The $A$-cycle integral $\int_{A_1}dz_{\sigma(1)}\cdots \int_{A_n}dz_{\sigma(n)} \Phi$  is obviously holomorphic in $\tau$ on $\H$. We only need to show that it is holomorphic at $\tau=i\infty$. 

If $f$ is a constant, then 
$$
\int_{A_1}dz_{\sigma(1)}\cdots \int_{A_n}dz_{\sigma(n)}\, 1=1
$$
as desired. By linearity, we can assume $f$ is a non-constant monomial 
$$
f=\prod_{i>j}x_{ij}^{m_{ij}}y_{ij}^{\epsilon_{ij}}\,.
$$
Let 
 $$
\widetilde{\Phi}=f(P(z_i-z_j),{1\over 2\pi i}P^\prime(z_i-z_j)), \quad \text{where}\quad P(z)=\wp(z)+{\pi^{2}\over 3}E_{2}\,. 
 $$
 Since $\lim\limits_{\tau\to i\infty} E_2=1$, 
 proving the desired claims about the asymptotic at $\tau=i\infty$
 is equivalent to showing that 
 $$
 \lim_{\tau\to i\infty} \int_{A_1}dz_{\sigma(1)}\cdots \int_{A_n}dz_{\sigma(n)} \widetilde{\Phi}=0\,,\quad\forall\,\sigma\in S_{n}\,.
 $$
 
 We follow the approach
 outlined in 
 Remark \ref{remcomputationontatecurve} to evaluate the above integral.
 We first lift the functions $P(z), P'(z)$ along the Picard uniformization 
$$
u=\exp(2\pi i z), \qquad q=e^{2\pi i \tau}\,.
$$
By \eqref{eqnwponC*} in Appendix \ref{secmodularformsellipticfunctions}, we have the absolutely convergent series expansion in $u$
\begin{equation*}
P(u)=(2\pi i)^2\sum_{k\geq 1} {k u^{k}\over 1-q^{k}}
+(2\pi i)^2\sum_{k\geq 1} {k q^{k} u^{-k}\over 1-q^{k}}\,,\quad \text{valid in the region}\quad 
|q|<|u|<1\,. 
\end{equation*}
This can be written as
$$
P(u)=\sum_{k\neq 0}c_{k} u^{k}\,,\quad c_{k}=(2\pi i)^2 {k\over 1-q^{k}}, \quad |q|<|u|<1\,. 
$$

\begin{figure}[H]\centering
	\begin{tikzpicture}[scale=1]

\fill[cyan!23, even odd rule](10,1.5)circle[radius=3] circle[radius=1];
\draw[very thick, cyan!50, ->-=.5,>=stealth](10,1.5)circle(3);

\draw[cyan!23](10,1.5)circle(1); 
\node [right] at (13,1.5) {$1$};
\filldraw[black](13,1.5) circle (2pt);
\node [left] at (10.75,2.2) {$q$};
\filldraw[](10.75,2.2) circle (0.5pt);

\draw[very thick, blue,->-=.5,>=stealth](10,1.5)circle(2.7);
\node at (12.70,1.5) {$A_1$};
\draw[very thick, blue,->-=.5,>=stealth](10,1.5)circle(2.2);
\node at (12.20,1.5) {$A_2$};
\node at (11.75,1.5) {$\cdots$};
\draw[ very thick, blue,->-=.5,>=stealth](10,1.5)circle(1.25);
\node at (11.25,1.5) {$A_n$};
	\end{tikzpicture}
	\caption{$A$-cycle representatives on the $u$-plane.}
	  \label{fig:standardordering}
\end{figure}
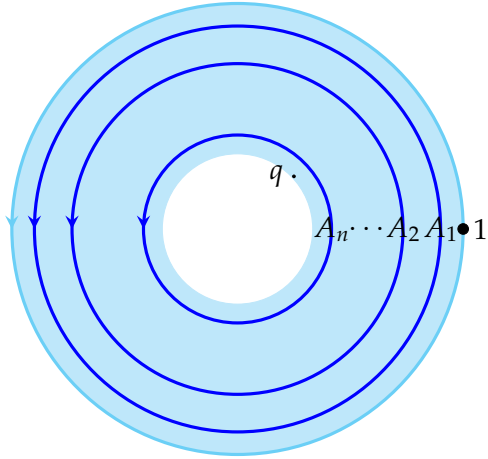

We only consider the case when $\sigma=1\in S_{n}$, and the $A$-cycles are ordered and represented on the $u$-plane within the region $|q|<|u|<1$ as in Fig. \ref{fig:standardordering}.
The same argument
applies to other cases in the way explained in Remark \ref{remotherorderings}.
Then
$$
\int_{A_1}dz_1\cdots \int_{A_n}dz_n \widetilde{\Phi}={1\over (2\pi i)^n}\int_{A_1}{du_1\over u_1} \cdots \int_{A_n}{du_n\over u_n} \prod_{i>j}(P(u_{ij}))^{m_{ij}}(u_{ij}\pa_{u}P(u_{ij}))^{\epsilon_{ij}}\,. 
$$
Here $u_{ij}=u_i/u_j$. Notice that
$$
 |q|<|u_{ij}|<1\quad \text{if}\quad i>j, \quad u_i\in A_i\,. 
$$
So we can use the above power series expression for $P(u_{ij})$. We are interested in the region when $|q|$ is very small ($\tau\to i\infty$). We fix the radius of each $A_i$, say between ${1\over 2}$ and $1$, and assume $|q|<{1\over 4}$.  Then the power series expansion of $ \prod_{i>j}(P(u_{ij}))^{m_{ij}}(u_{ij}\pa_{u}P(u_{ij}))^{\epsilon_{ij}}$ is uniformly absolutely convergent within the integration region and $|q|$ small, so we can integrate term by term and compute the limit $q\to 0$. 

Since
$$
{1\over (2\pi i)^n}\int_{A_1}{du_1\over u_1} \cdots \int_{A_n}{du_n\over u_n}  u_1^{k_1}u_2^{k_2}\cdots u_n^{k_n}=\delta_{k_1,0}\delta_{k_2,0}\cdots \delta_{k_n,0}\,, 
$$
the value of the $A$-cycle integral  is given by 
 the coefficient  of the constant term 
 $$
 u_1^0 u_2^0 \cdots u_n^0
 $$
 of the series expansion of $ \prod_{i>j}(P(u_{ij}))^{m_{ij}}(u_{ij}\pa_{u}P(u_{ij}))^{\epsilon_{ij}}$. 
 
Let us assume $\sum\limits_{i<n}(m_{ni}+ \epsilon_{ni})>0$. Otherwise the integral does not depend on $u_n$, so we can integrate out $u_n$ first and repeat this process to arrive at this situation. Consider the series expansion that involves $u_n$
$$
 \prod_{i<n}(P(u_{ni}))^{m_{ni}}(u_{ij}\pa_{u}P(u_{ij}))^{\epsilon_{ni}}, \quad P(u_{ni})=\sum_{k\neq 0}c_{k} u_{ni}^{k}=\sum_{k\neq 0}c_{k} {u_n^k\over u_i^k}\,. 
$$
Each term that has $u_n^0$-order contains at least one factor of $c_k$ with $k<0$. Since
$$
\lim_{q\to 0}c_k=\begin{cases}(2\pi i)^2 k & k>0\,,\\
 0 & k<0\,. 
\end{cases}
$$
We find the desired vanishing property
$$
\lim_{q\to 0}{1\over (2\pi i)^n}\int_{A_1}{du_1\over u_1} \cdots \int_{A_n}{du_n\over u_n} \prod_{i>j}(P(u_{ij}))^{m_{ij}}(u_{ij}\pa_{u}P(u_{ij}))^{\epsilon_{ij}}=0\,. 
$$

\end{proof}

\begin{rem} The regularity result Lemma
\ref{lem-regularity} actually holds for more general 
meromorphic elliptic functions. For example, one also has the holomorphicity at $\tau=i\infty$ for the $A$-cycle integral of a meromorphic elliptic function of the form
$$\Phi(z_1,\cdots,z_n;\tau)=\prod_{1\leq i<j\leq n}{\theta(z_{i}-z_{j}+c_{ij};\tau)\over \theta(z_{i}-z_{j};\tau)}\,.$$
Here $\theta(z;\tau)$ is the unique theta function with odd characteristic of genus one that 
vanishes at $z=0$, $c_{ij}$ are constants that could depend linearly in $\tau$.
This can be proved in a similar way as Lemma
\ref{lem-regularity} by using the Jacobi triple product formula for $\theta$
on $\mathbb{C}^{*}$.
Regularized integrals of functions of this form
include Feynman graph integrals  that appear in chiral $bc$-systems and chiral $\beta\gamma$-systems. 

\end{rem}

\begin{thm}\label{thm-Feynmangraphintegralsmodular} 
For each oriented graph $\Gamma$ with no self-loops, the regularized Feynman graph integral 
$\widehat{I}_{\Gamma}$
  is an almost-holomorphic modular form
 of weight $2|E(\Gamma)|$.
  Its holomorphic limit is given by 
$$
\lim_{\bar\tau\to \infty}\widehat{I}_{\Gamma}={1\over n!}\sum_{\sigma\in S_n} \int_{A_{1}} dz_{\sigma(1)}\cdots \int_{A_{n}} dz_{\sigma(n)} \,\lim_{\bar{\tau}\rightarrow \infty}
\Phi_{\Gamma}(z_1,\cdots, z_n;\tau,\bar{\tau}) $$
which is a holomorphic quasi-modular form of the same weight. Here $A_1, \cdots, A_n$ are $n$ disjoint representatives of the $A$-cycle class on $E_\tau$. 
\end{thm}
\begin{proof} $\widehat{I}_{\Gamma}$ is modular of weight $2|E(\Gamma)|$ since $\widehat{P}$ is modular of weight $2$. By Theorem \ref{thm-2d}
$$
\dashint_{E_{\tau}^{n}}  \bracket{\prod_{i=1}^n {d^2z_i\over \im \tau}}
  \Phi_{\Gamma,m}(z_1,\cdots, z_n;\tau)\in  \OO_{\H}[{1\over \im \tau}]
$$
for each $m$. Therefore $\widehat{I}_{\Gamma}\in  \OO_{\H}[{1\over \im \tau}]$ and 
\begin{align*}
\lim_{\bar \tau\to \infty}\widehat{I}_{\Gamma}=&\lim_{\bar \tau\to \infty}\dashint_{E_{\tau}^{n}}  \bracket{\prod_{i=1}^n {d^2z_i\over \im \tau}}
  \Phi_{\Gamma,0}(z_1,\cdots, z_n;\tau)\\
  =& {1\over n!}\sum_{\sigma\in S_n} \int_{A_{1}} dz_{\sigma(1)}\cdots \int_{A_{n}} dz_{\sigma(n)} \,\lim_{\bar{\tau}\rightarrow \infty}
\Phi_{\Gamma}(z_1,\cdots, z_n;\tau,\bar{\tau}). 
\end{align*}
By Lemma \ref{lem-regularity} , $\lim\limits_{\bar \tau\to \infty}\widehat{I}_{\Gamma}$ is holomorphic on $\H$ and holomorphically extended to $\tau=i\infty$. Therefore
$\widehat{I}_{\Gamma}$ is an almost-holomorphic modular form
and $\lim\limits_{\bar \tau\to \infty}\widehat{I}_{\Gamma}$
is a holomorphic quasi-modular form, both of weight $2|E(V)|$.
\end{proof}

We discuss some examples to illustrate how to compute $\widehat{I}_{\Gamma}$ using Theorem \ref{thm-Feynmangraphintegralsmodular}.

\begin{ex}\label{exFeymangraphintegral1}
Consider the following graph $\Gamma_\ell$ in Fig. \ref{fig:Feynmangraph1} with $2$ vertices and $\ell$ edges.

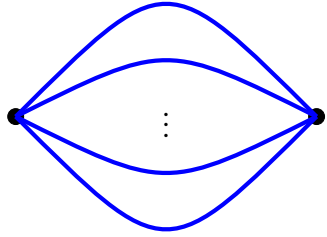
\begin{figure}[h]\centering
	\begin{tikzpicture}[scale=1]

\filldraw[black](0,0) circle (3pt);
\filldraw[black](4,0) circle (3pt);
\draw[ultra thick, blue] (0,0)..controls (2,2)  and (2,2) .. (4,0);
\draw[ultra thick, blue] (0,0)..controls (2,1)  and (2,1) .. (4,0);
\node[] at (2,0){$\vdots$};
\draw[ultra thick, blue] (0,0)..controls (2,-1)  and (2,-1) .. (4,0);
\draw[ultra thick, blue] (0,0)..controls (2,-2)  and (2,-2) .. (4,0);


	\end{tikzpicture}
	\caption{The banana graph $\Gamma_\ell$: $2$ vertices and $\ell$ edges.}
	  \label{fig:Feynmangraph1}
\end{figure}

The regularized Feynman graph integral is
$$\widehat{I}_{\Gamma_\ell}=\dashint_{E_{\tau}^2} \bracket{\prod_{i=1}^2 {d^2z_i\over \im \tau}} \widehat{P}^{\ell}(z_1,z_2;\tau,\bar\tau)=\dashint_{E_{\tau}^2} \bracket{\prod_{i=1}^2 {d^2z_i\over \im \tau}} \bracket{\wp(z_1-z_2;\tau)+{\pi^2\over 3}\widehat{E}_{2}(\tau,\bar{\tau})}^\ell.
$$
Translation symmetry implies that 
\begin{align*}
\lim_{\bar\tau \to \infty}\widehat{I}_{\Gamma_\ell}&=\int_{A_1} dz _1\int_{A_2} dz_2 \bracket{\wp(z_1-z_2;\tau)+{\pi^2\over 3}{E}_{2}(\tau)}^\ell\\
&=\int_{A} dz \bracket{\wp(z;\tau)+{\pi^2\over 3}{E}_{2}(\tau)}^\ell.
\end{align*}
Here $A$ is a representative of the $A$-cycle class away from the origin $O\in E_\tau$ which is the pole of the 2nd kind Abelian differential 
$$
\varphi=\bracket{\wp(z;\tau)+{\pi^2\over 3}{E}_{2}(\tau)}^\ell dz\,. 
$$
By Proposition \ref{prop-modular-completion}, 
$$
\widehat{I}_{\Gamma_\ell}= \int_{A}\varphi -{1\over 2 i \im \tau}\cdot  2\pi i\,\langle \varphi, dz\rangle_{\mathrm{P}}. 
$$
Theorem \ref{thmstructuretheorem} implies that $\widehat{I}_{\Gamma_\ell}$
is recovered from $\lim\limits_{\bar\tau \to \infty}\widehat{I}_{\Gamma_\ell}$
by writing the latter in terms of a polynomial in $E_{2},E_{4},E_{6}$,
then replacing $E_{2}$ by $\widehat{E}_{2}$.
\end{ex}

We list the results for the first few graphs up to $\ell=3$ as follows. 
From Example \ref{expmoments}, we have 
 the following formulae in terms of holomorphic quasi-modular forms
\begin{align*}
  &\int_A dz \bracket{\wp(z;\tau)+{\pi^2\over 3}{E}_{2}(\tau)}=0\,,\\
  &\int_A dz \bracket{\wp(z;\tau)+{\pi^2\over 3}{E}_{2}(\tau)}^2=\pi^{4}{
-E_{2}^2+E_{4} \over 9}\,,\\
  &\int_A dz \bracket{\wp(z;\tau)+{\pi^2\over 3}{E}_{2}(\tau)}^3=\pi^{6}
  {
-10 E_{2}^3 +6 E_{2}E_{4}+4  E_{6}\over  
   5\cdot 27}\,.
\end{align*}
By Theorem \ref{thm-Feynmangraphintegralsmodular}, the regularized graph integrals are given by their modular completions 
\begin{align*}
  \widehat{I}_{\Gamma_1}=0,\quad \widehat{I}_{\Gamma_2}=\pi^{4}{
-\widehat{E}_{2}^2+E_{4} \over 9}, \quad \widehat{I}_{\Gamma_3}=\pi^{6}
  {
-10 \widehat{E}_{2}^3 +6 \widehat{E}_{2}E_{4}+4   E_{6}\over  
   5\cdot 27}\,.
\end{align*}

\begin{ex}\label{exFeymangraphintegral2}
Consider the following graph $\Gamma_{\Delta}$ with 3 vertices and 3 edges (Fig.   \ref{fig:Feynmangraph2}).

\begin{figure}[H]\centering
	\begin{tikzpicture}[scale=1]
%

\filldraw[black](0,0) circle (3pt);
\filldraw[black](4,0) circle (3pt);
\filldraw[black](2,3) circle (3pt);
\draw[ultra thick, blue] (0,0)to(4,0)to(2,3)to(0,0);

	\end{tikzpicture}
	\caption{The triangle graph $\Gamma_\Delta$:  $3$ vertices and  $3$ edges.}
	  \label{fig:Feynmangraph2}
\end{figure}
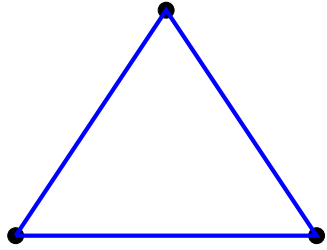

The regularized Feynman graph integral is
\begin{eqnarray*}
\widehat{I}_{\Gamma_{\Delta}}&=&\dashint_{E_{\tau}^3} \bracket{\prod_{i=1}^3 {d^2z_i\over \im \tau}} \Phi_{\Gamma_{\Delta}}
\,,\quad
\Phi_{\Gamma_{\Delta}}=\widehat{P}(z_1,z_2) \widehat{P}(z_2,z_3)
\widehat{P}(z_3,z_1)
\,.
\end{eqnarray*}

By Theorem \ref{thm-Feynmangraphintegralsmodular}  and Theorem \ref{thmstructuretheorem}, it suffices to compute 
the iterated $A$-cycle integrals.
In the current case there is a permutation symmetry on the graph, hence iterated $A$-cycle integrals with different orderings give rise to the same result.
Therefore we have
\begin{equation*}
\lim_{\bar \tau\to \infty}\widehat{I}_{\Gamma_{\Delta}}=
\int_{A} dz_3 \int_{A} dz_2   \int_{A} dz_1\,
P(z_{1}-z_{2})P(z_{2}-z_{3})P(z_{3}-z_{1})\,,\quad
P(z)=\wp(z)+{\pi^2\over 3}E_{2}
\,.
\end{equation*}
Following the approach outlined in Remark
\ref{remcomputationontatecurve}, we obtain
\begin{eqnarray*}
 \int_{A} dz_1\,
P(z_{1}-z_{2})P(z_{2}-z_{3})P(z_{3}-z_{1})
&=&(2\pi i)^{4} P(z_{2}-z_{3})
\sum_{k\geq 1} {k^{2}q^{k}\over (1-q^{k})^2}({u_{3}\over u_{2}})^{k}\,,\\
  \int_{A} dz_2\int_{A} dz_1\,
P(z_{1}-z_{2})P(z_{2}-z_{3})P(z_{3}-z_{1})
&=&(2\pi i)^{6} 
\sum_{k\geq 1} {k^{3}(q^{k}+q^{2k})\over (1-q^{k})^3}\,.
\end{eqnarray*}
Using  \eqref{eqn-FourierexpansionsE2k}, \eqref{eqn-Ramanujanidentities}, we obtain 
$$(2\pi i)^6 \sum_{k\geq 1} {k^3 (q^{k}+ q^{2k})\over (1-q^{k})^3 }
=-{1\over 24 } \cdot (2\pi i)^6  (q{ d\over dq})^2 E_{2}={1\over 12^3} (2\pi i)^6 (-E_{2}^3+3E_{2}E_{4}-2E_{6})
\,.
$$
Using Theorem \ref{thm-Feynmangraphintegralsmodular} and Theorem \ref{thmstructuretheorem}  we then have
$$\widehat{I}_{\Gamma_{\Delta}}= {1\over 12^3} (2\pi i)^6 (-\widehat{E}_{2}^3+3\widehat{E}_{2}E_{4}-2E_{6})\,.$$
As a comparison, 
a straightforward way of evaluating the iterated regularized 2d integrals 
is presented in Appendix \ref{secstraightforwardevaluation}.
\end {ex}

\begin{ex}\label{exFeymangraphintegral3}
Consider the graphs
 $\Gamma_{1}, \Gamma_{2}$
in Fig.  \ref{fig:Feynmangraph3} below. These are the only two trivalent graphs with 4 vertices and no self-loops.

\begin{figure}[H]\centering

	\begin{tikzpicture}[scale=1]

\filldraw[black](0,0) circle (3pt);
\filldraw[black](4,0) circle (3pt);
\filldraw[black](2,3) circle (3pt);
\filldraw[black](2,1) circle (3pt);
\draw[ultra thick, blue] (0,0)to(4,0)to(2,3)to(0,0);
\draw[ultra thick, blue] (0,0)to(2,1);
\draw[ultra thick, blue] (4,0)to(2,1);
\draw[ultra thick, blue] (2,3)to(2,1);

\node [below] at (0,-0.25) {$1$};
\node [below] at (4, -0.25) {$2$};
\node [below] at (2, 3.65) {$3$};
\node [below] at (2, 0.85) {$4$};
\node [below] at (2, -1) {$\Gamma_1$};

\filldraw[black](7,0) circle (3pt);
\filldraw[black](11,0) circle (3pt);
\filldraw[black](11,3) circle (3pt);
\filldraw[black](7,3) circle (3pt);
\draw[ultra thick, blue] (7,0)..controls (6.25,1.5)  and (6.25,1.5) .. (7,3);
\draw[ultra thick, blue] (7,0)..controls (7.75,1.5)  and (7.75,1.5) .. (7,3);
\draw[ultra thick, blue] (11,0)..controls (10.25,1.5)  and (10.25,1.5) .. (11,3);
\draw[ultra thick, blue] (11,0)..controls (11.75,1.5)  and (11.75,1.5) .. (11,3);
\draw[ultra thick, blue] (7,0)to(11,0);
\draw[ultra thick, blue] (7,3)to(11,3);

\node [below] at (7,-0.25) {$1$};
\node [below] at (11, -0.25) {$2$};
\node [below] at (7, 3.65) {$4$};
\node [below] at (11, 3.65) {$3$};

\node [below] at (9, -1) {$\Gamma_2$};

	\end{tikzpicture}
	\caption{Two trivalent graphs $\Gamma_{1}, \Gamma_{2}$ with $4$ vertices each.}
	  \label{fig:Feynmangraph3}
\end{figure}
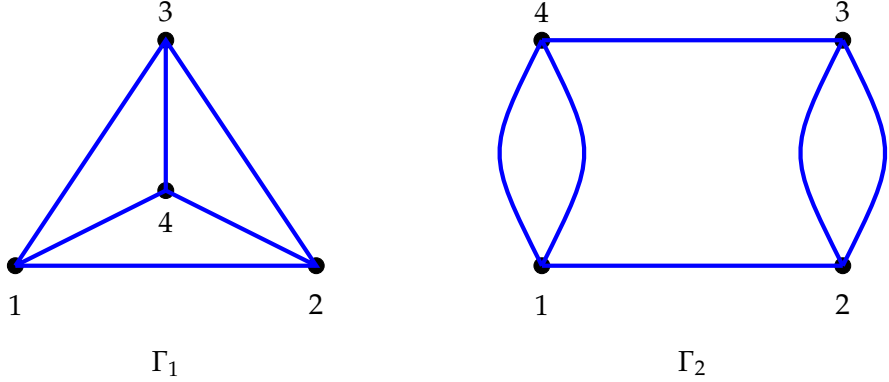

The ordered $A$-cycle integrals for these graphs are studied in e.g.,
 \cite[Section 6]{Roth:2009mirror}, \cite[Example 3.5]{Boehm:2014tropical}, \cite[Section 5.4, Section 9.3]{Goujard:2016counting}.
 In particular, the mixed-weight phenomenon for the ordered $A$-cycle integrals 
discovered 
 in \cite{Goujard:2016counting,Oberdieck:2018} are demonstrated on these examples. 
 
 As pointed out in the above-cited works, changing the labeling of the vertices
 is equivalent to changing the ordering for the iterated $A$-cycle integrals. This is also evident from our definition of ordered $A$-cycle integrals.
 Hence we can stick to a particularly chosen labeling for the graph as
 we are going to consider all possible orderings. 
 
Consider first the Feynman graph integral associated to $\Gamma_{1}$.
We fix the labeling of vertices to be the one indicated in Fig.  \ref{fig:Feynmangraph3}.
The function $\Phi_{\Gamma_1}(z_1,z_2,z_3,z_4;\tau)$ associated to this 
labeled graph is 
$$\Phi_{\Gamma_1}(z_1,z_2,z_3,z_4;\tau)=\widehat{P}(z_1-z_2)\widehat{P}(z_2-z_3)\widehat{P}(z_3-z_1)\widehat{P}(z_1-z_4)
\widehat{P}(z_2-z_4)\widehat{P}(z_3-z_4)\,.$$
It is invariant under the action of $S_4$ which permutes the vertices labeled by $1,2,3,4$. 
Since all these ordered $A$-cycles integrals are equal, according to Theorem \ref{thm-2d} they have pure modular weight $12$.
Indeed,  their holomorphic limits are given by  (see e.g., \cite[Section 9.3]{Goujard:2016counting}) the following quasi-modular form of pure weight $12$
\begin{eqnarray*}
\bold{C}:={(2\pi i)^{12}\over 2^{11}\cdot 3^{5}}(-E_{2}^6+3E_{2}^4 E_{4}-3E_{2}^2 E_{4}^2+E_{4}^{3})\,.
\end{eqnarray*}
Note that the results in  \cite{Goujard:2016counting} are expressed in terms of the basis $G_{2k}=-{B_{2k}\over 4k} E_{2k},k\geq 1$, here we use the $E_{2k}$'s whose normalizations are more convenient in consideration of Lemma \ref{lem-regularity}. The average of such quantities is thus $\bold{C}$ itself
\begin{eqnarray*}
{1\over 4!}\sum_{\sigma\in S_4} \int_{A}dz_{\sigma(1)}
\cdots \int_{A} dz_{\sigma(4)}\lim_{\bar{\tau}\rightarrow \infty} \Phi_{\Gamma_1}
&=& {(2\pi i)^{12}\over 2^{11}\cdot 3^{5}}(-E_{2}^6+3E_{2}^4 E_{4}-3E_{2}^2 E_{4}^2+E_{4}^{3})
\,.
\end{eqnarray*}
Applying Theorem \ref{thm-2d}, we obtain the following result for the regularized integral 
$$\widehat{I}_{\Gamma_1}=
\dashint_{E_{\tau}^4} \bracket{\prod_{i=1}^4 {d^2z_i\over \im \tau}} \Phi_{\Gamma_{1}}
={(2\pi i)^{12}\over 2^{11}\cdot 3^{5}}(-\widehat{E}_{2}^6+3\widehat{E}_{2}^4 E_{4}-3\widehat{E}_{2}^2 E_{4}^2+E_{4}^{3})\,.
$$

For the regularized Feynman graph integral associated to $\Gamma_{2}$,
we fix the labeling of vertices to be the one indicated in Fig.  \ref{fig:Feynmangraph3}.
The associated function $\Phi_{\Gamma_2}(z_1,z_2,z_3,z_4;\tau)$ is
$$\Phi_{\Gamma_2}(z_1,z_2,z_3,z_4;\tau)=\widehat{P}(z_1-z_2)\widehat{P}^2(z_2-z_3)\widehat{P}(z_3-z_4)\widehat{P}^2(z_4-z_1)\,.$$
It is invariant under the automorphism group $G$ of the labeled graph which is generated by
horizontal and vertical flips.
Among the $4!=24$
ordered $A$-cycles integrals it suffices to consider
$4!/|G|=6$ of them.
The same reasoning as in the $\Gamma_1$ case tells that the holomorphic limits of these ordered $A$-cycle
integrals are linear combinations of quasi-modular forms of weight $12$ and $10$.
The results for these integrals, which we quote from
\cite[Section 5.4, Section 9.3]{Goujard:2016counting}, are as follows.
Let
\begin{eqnarray*}
\bold{a}:&=&{(2\pi i)^{12}\over 2^{10}\cdot 3^{7}}(-3E_{2}^6+6E_{2}^4 E_{4}+4 E_{2}^3 E_{6}-3 E_{2}^2 E_{4}^2-12 E_{2}E_{4}E_{6}+4E_{4}^{3}+4E_{6}^2)\,,\\
\bold{b}:&=&{(2\pi i)^{12}\over 2^{6} \cdot 3^{5}\cdot 5\cdot 7}(-7E_{2}^3 E_{4}-3E_{2}^4 E_{6}+3E_{2} E_{4}^2+7 E_{4}E_{6})\,.
\end{eqnarray*}
Then $2$ out of the $6$ inequivalent ordered $A$-cycle integrals are equal to
$$\bold{A}:=\int_{A}dz_{4}\int_{A}dz_{3}\int_{A}dz_{2}\int_{A}dz_{1}\,\lim_{\bar{\tau}\rightarrow \infty}\Phi_{\Gamma_1}=\bold{a}-2\bold{b}\,.$$
The other $4$ are equal to
$$\bold{B}:=\int_{A}dz_{2}\int_{A}dz_{3}\int_{A}dz_{4}\int_{A}dz_{1}\,\lim_{\bar{\tau}\rightarrow \infty}\Phi_{\Gamma_1}=\bold{a}+\bold{b}\,.$$
The average of the holomorphic limits of ordered $A$-cycle integrals is then
\begin{eqnarray*}
{1\over 4!}\sum_{\sigma\in S_4} \int_{A}dz_{\sigma(1)}
\cdots \int_{A} dz_{\sigma(4)}\lim_{\bar{\tau}\rightarrow \infty} \Phi_{\Gamma_2}
&=& {1\over 4!} \cdot |G| \cdot (\bold{A}\cdot 2+ \bold{B}\cdot 4)=\bold{a}
\,.
\end{eqnarray*}
Applying Theorem \ref{thm-2d}, we obtain the following result for the regularized integral 
$$\widehat{I}_{\Gamma_2}=
\dashint_{E_{\tau}^4} \bracket{\prod_{i=1}^4 {d^2z_i\over \im \tau}} \Phi_{\Gamma_{2}}
={(2\pi i)^{12}\over 2^{10}\cdot 3^{7}}(-3\widehat{E}_{2}^6+6\widehat{E}_{2}^4 E_{4}+4 \widehat{E}_{2}^3 E_{6}-3 \widehat{E}_{2}^2 E_{4}^2-12 \widehat{E}_{2}E_{4}E_{6}+4E_{4}^{3}+4E_{6}^2)\,.
$$

\end{ex}

\subsection{Proof of Theorem \ref{thm-2d}} \label{sec-proof}

In this subsection, we complete the proof of Theorem \ref{thm-2d}. The statements (1) and (3) of Theorem \ref{thm-2d} follow from Theorem \ref{thm-modularity}. We next compute the limit 
$$
\lim_{\bar\tau \to \infty} \dashint_{E_\tau^n} \bracket{\prod_{i=1}^n {d^2z_i\over \im \tau}} \Phi(z_1,\cdots, z_n;\tau). 
$$

Let $\square_c$ ($c\in \C$) denote the parallelogram in $\C$ with vertices $\{c,c+1,c+1+\tau, c+\tau\}$. 
Let $A_c^+, A_c^-, B_c^+, B_c^-$ denote the intervals as illustrated in 
Fig. \ref{fig:parallelogramsquarecwithcycles}.

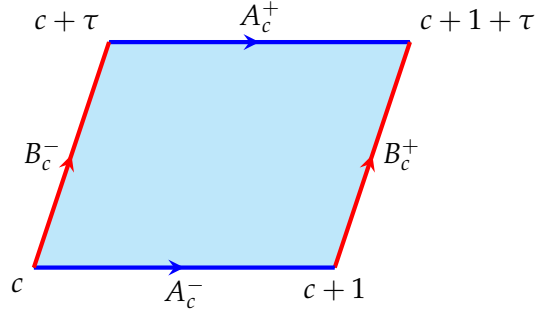
\begin{figure}[H]\centering
	\begin{tikzpicture}[scale=1]

\draw[cyan!23,fill=cyan!23](0,0)to(4,0)to(5,3)to(1,3)to(0,0);
\draw (0,0) node [below left] {$c$} to (4,0) node [below] {$c+1$} to (5,3) node [above right] {$c+1+\tau$} to (1,3) node [above left] {$c+\tau$} to (0,0);

\draw[ultra thick,blue,->-=.5,>=stealth](1,3)to(5,3);

\node [above] at (3,3) {$A_c^+$};

\draw[ultra thick,red,->-=.5,>=stealth](4,0)to(5,3);

\node [right] at (4.5,1.5) {$B_c^+$};

\node [below] at (2,0) {$A_c^-$};

\draw[ultra thick,blue,->-=.5,>=stealth](0,0)to(4,0);

\draw[ultra thick,red,->-=.5,>=stealth](0,0)to(1,3);

\node [left] at (0.5,1.5) {$B_c^-$};

	\end{tikzpicture}
	\caption{Intervals on the parallelogram $\square_c$.}\label{fig:parallelogramsquarecwithcycles}
\end{figure}

\begin{lem}\label{lem-integral}Let $\Psi$ be an almost-meromorphic elliptic function on $\C\times \H$. Let us write
$$
\Psi=\sum_{k} \Psi_k \bracket{\im z\over \im \tau}^k, \quad \Psi_k=\sum_{m=0}^{n_k}{\Psi_{k,m} \over (\im \tau)^m}, \quad \text{where $\Psi_{k,m}$ is meromorphic on $\C\times \H$}\,.  
$$
Let $\square_c$ be a parallelogram whose boundary does not meet poles of $\Psi$. Then 
\begin{align*}
  \dashint_{E_\tau} {d^2z \over \im \tau} \Psi &= \sum_k {1\over k+1} \int_{A_c^+} dz\,{\Psi_k(z)}+\sum_{k} {1\over k+1} \sum_{w\in D}\bracket{{\im w\over \im \tau}}^{k+1}\oint_w dz\,  {\Psi_k(z) } \\
   &\quad + \sum_{k} {1\over k+1} \sum_{j=1}^{k+1} {1\over (\im \tau)^j}  \binom{k+1}{j}  \sum_{w\in D} \bracket{ {\im w\over \im \tau}}^{k+1-j} \oint_w dz \,\Psi_k(z)\bracket{{z-w\over 2i}}^j  .
\end{align*}
Here $D$ consists of the poles of $\Psi$ inside $\square_c$. 
\end{lem}
\begin{proof} Let 
$$
\Upsilon=\sum_k {\Psi_k\over k+1}\bracket{\im z\over \im \tau}^{k+1},\quad \pa_{\bar z} \Upsilon={i\over 2\im \tau} \Psi\,. 
$$
The relation
$\Psi(z+1)=\Psi(z)$ implies $\Upsilon(z+1)=\Upsilon(z)$. The relation $\Psi(z+\tau)=\Psi(z)$ implies
$$
\pa_{\bar z} \bracket{\Upsilon(z+\tau) -\Upsilon(z)}={i\over 2 \im\tau}\bracket{\Psi(z+\tau)-\Psi(z)}=0\,. 
$$
This further implies 
$$
   \Upsilon(z+\tau) -\Upsilon(z)=\sum_k {\Psi_k(z+\tau)\over k+1}\,.
$$
By Theorem \ref{thm-de-Rham} and Proposition \ref{prop-residue}
\begin{align*}
  \dashint_{E_\tau} {d^2z\over \im \tau} \Psi &=  \dashint_{\square_c} {d^2z\over \im \tau} \Psi =   \dashint_{\square_c} \Psi  dz \wedge {d \im z\over \im \tau}=-\dashint_{\square_c} d\bracket{\Upsilon dz}\\
   &= \int_{A_c^-} dz \bracket{\Upsilon(z+\tau) -\Upsilon(z)}-\int_{B_c^-}dz \bracket{\Upsilon(z+1) -\Upsilon(z)}+\sum_{w\in D}\oint_w \Upsilon dz\\
   &=\int_{A_c^-} dz \sum_k {\Psi_k(z+\tau)\over k+1}+\sum_{k} {1\over k+1} \sum_{w\in D}\oint_w dz  \bracket{\Psi_k(z) \bracket{{z-\bar z\over 2i \im \tau}}^{k+1}} \\
      &= \sum_k {1\over k+1} \int_{A_c^+} dz\, {\Psi_k(z)}+\sum_{k} {1\over k+1} \sum_{w\in D}\oint_w dz  \bracket{\Psi_k(z) \bracket{{z-\bar w\over 2i \im \tau}}^{k+1}} \\
   &= \sum_k {1\over k+1} \int_{A_c^+} dz \,{\Psi_k(z)}+\sum_{k} {1\over k+1} \sum_{w\in D}\oint_w dz  \bracket{\Psi_k(z) \bracket{{z-w\over 2i \im \tau}+{\im w\over \im \tau}}^{k+1}}\\
   &= \sum_k {1\over k+1} \int_{A_c^+} dz\, {\Psi_k(z)}+\sum_{k} {1\over k+1} \sum_{w\in D}\bracket{{\im w\over \im \tau}}^{k+1}\oint_w dz \, {\Psi_k(z) } \\
   &\quad + \sum_{k} {1\over k+1} \sum_{j=1}^{k+1} {1\over (\im \tau)^j}  \binom{k+1}{j}  \sum_{w\in D} \bracket{ {\im w\over \im \tau}}^{k+1-j} \oint_w dz\, \Psi_k(z)\bracket{{z-w\over 2i}}^j  \,.
\end{align*}

\end{proof}

Our strategy to prove Theorem \ref{thm-2d} is to apply Lemma \ref{lem-integral} $n$ times to

 $$
 \dashint_{E_\tau^n} \bracket{\prod_{i=1}^n {d^2z_i}} \Phi(z_1,\cdots, z_n;\tau)=\dashint_{E_\tau} {d^2z_{1}} \cdots \dashint_{E_\tau}  {d^2z_{n}} \Phi(z_1,\cdots, z_n;\tau)
 $$
 as an iterated integral over parallelogram $\square_c$'s 
 $$
 \dashint_{\square_{c_1}} {d^2z_{1}} \cdots \dashint_{\square_{c_n}}  {d^2z_{n}} \Phi(z_1,\cdots, z_n;\tau)\,. 
 $$
One immediate difficulty is that we need to ensure all poles lie in the interior of parallelograms in each step of  integration.  Therefore we have to choose the shift $c_i$'s suitably.  
 
Let $c_0\in \square_0$ be a chosen point such that
 $$
z\in \square_{-c_0}\Rightarrow -{1\over 2}z\in \text{interior of}\  \square_{-c_0}\,. 
$$
Such $c_0$ always exists and can be chosen to be independent of $\tau$ under a small perturbation of $\tau$. Let $\epsilon_1, \cdots, \epsilon_{n-1}\in (0,1)$ be small enough positive real numbers satisfying 
$$
 \epsilon_{i+1}+\cdots +\epsilon_{n-1}<{1\over 2}\epsilon_i\,, \quad \forall 1\leq i\leq n-2\,. 
$$

Consider the following linear change of variables $z_i\mapsto w_i$
$$
\begin{pmatrix}
z_1\\
z_2\\
\vdots\\
\vdots\\
z_n
\end{pmatrix}
=\begin{pmatrix}
1 & 0 &  0 & \cdots & 0\\
\epsilon_1 & 1 & 0&  \cdots & 0\\
\epsilon_1 & \epsilon_2 & 1 & \cdots & 0\\
\vdots & \vdots & \vdots & \vdots & \vdots\\
\epsilon_1 & \epsilon_2 & \cdots & \epsilon_{n-1} & 1
\end{pmatrix}
\begin{pmatrix}
w_1\\
w_2\\
\vdots\\
\vdots\\
w_n
\end{pmatrix}\,.
$$

 Let 
$$
\begin{cases}
c_1=-c_0 \\
c_2=c_2(z_1)=-c_0+ \epsilon_1 w_1 \\
 c_3=c_3(z_1,z_2)= -c_0+\epsilon_1 w_1+\epsilon_2 w_2\\
 \vdots \\
 c_{n}=c_n(z_1,\cdots,z_{n-1})= -c_0+ \epsilon_1 w_1 + \epsilon_2w_2+\cdots+ \epsilon_{n-1} w_{n-1}\,. 
\end{cases}
$$
Here $c_i$ is viewed as a function of $z_1,\cdots, z_{i-1}$ under the inverse of the above transformation. 
  
\begin{lem}\label{Claim 1}  For any $2\leq i\leq n$, if $z_1\in \square_{c_1},\cdots, z_{i-1}\in \square_{c_{i-1}}$, then 
$$
\text{$z_1, \cdots, z_{i-1}$ lie in the interior of $\square_{c_i}$\,.} 
$$
\end{lem}
\begin{proof} Observe that
$$
z_j=w_j+c_0+c_j\,.
$$
The condition $z_j\in \square_{c_j}$ is the same as 
$$
w_j\in \square_{-c_0}\,. 
$$
For any $j<i$, 
\begin{align*}
z_j - (\epsilon_1 w_1+\cdots + \epsilon_{i-1}w_{i-1})&=(1-\epsilon_j)w_j- \epsilon_{j+1}w_{j+1}-\cdots- \epsilon_{i-1}w_{i-1}\\
&=(1-\epsilon_j)w_j+ (2\epsilon_{j+1})(-{1\over 2}w_{j+1})+\cdots+ (2\epsilon_{i-1})(-{1\over 2}w_{i-1})\,. 
\end{align*}
By our choice of $c_0$, all points $w_j, (-{1\over 2}w_{j+1}), \cdots, (-{1\over 2}w_{i-1})$ lie in $\square_{-c_0}$. Since 
$$
(1-\epsilon_{j} ) +(2\epsilon_{j+1})+\cdots+  (2\epsilon_{i-1}) <1
$$
and $\square_{-c_0}$ is a convex set containing the origin, the value $z_j - (\epsilon_1 w_1+\cdots + \epsilon_{i-1}w_{i-1})$ lies in the interior of $\square_{-c_0}$. So $z_j$ lies in the interior of $\square_{c_i}$ for any $j<i$. This proves the lemma. 
\end{proof}
 
Now we can write
$$
 \dashint_{E_\tau^n} \bracket{\prod_{i=1}^n {d^2z_i}} \Phi= \dashint_{\square_{c_1}} {d^2z_{1}} \cdots \dashint_{\square_{c_n}}  {d^2z_{n}} \Phi
$$
as an ordered regularized integral where $c_i$ depends on $z_1, \cdots, z_{i-1}$ as chosen above. It ensures that $z_1,\cdots, z_{i-1}$ lies in the interior of the integration parallelogram of $z_j$ for any $j\geq i$. The value of this iterated integral does not depend on the choice of $\epsilon_i$'s. 

We now apply Lemma \ref{lem-integral} $n$ times to this integral and keep the leading term in the ${1\over \im \tau}$-expansion in order to compute the limit $\bar\tau\to \infty$.  It is not hard to see that for an almost-meromorphic elliptic function $\Psi$ 
$$
  \dashint_{E_\tau} {d^2z \over \im \tau} \Psi = \sum_k {1\over k+1} \int_{A_c^+} dz \,{\Psi_k(z)}+\sum_{k} {1\over k+1} \sum_{w\in D}\bracket{{\im w\over \im \tau}}^{k+1}\oint_w dz\,  {\Psi_k(z) } + \mathcal{O}({1\over \im \tau})\,. 
$$ 
and hence we only need to keep the first two terms on the right hand side at each step of integration in the limit $\bar\tau\to \infty$.

The answer will become a combinatorial expression in terms of $A$-cycle integrals and residues. Let us first introduce some notations in order to describe this combinatorial result. 
\begin{dfn}
A tree is a connected undirected graph with no simple circuits. A rooted tree is a tree in which one vertex has been designated as the root. 
\end{dfn}
\begin{ex}
Here is an example of rooted tree. 
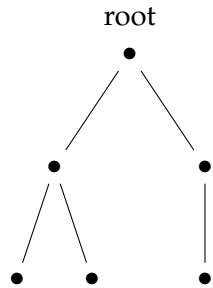
\begin{figure}[H]\centering
\begin{tikzpicture}[level distance=1.5cm,
  level 1/.style={sibling distance=2cm},
  level 2/.style={sibling distance=1 cm}]

  \node at (0,0) {$\bullet$}
    child {node {$\bullet$}
      child {node {$\bullet$}}
      child {node {$\bullet$}}
    }
    child {node {$\bullet$}
      child {node {$\bullet$}}
    };
    \node at (0,0.5) {root};
\end{tikzpicture}

\end{figure}
\end{ex}

Let $T$ be a rooted tree. Let
$$
V(T)=\text{vertices of $T$}, \quad  rt(T)\in V(T)\ \text{the root vertex}\,. 
$$
The level $l(v)$ of a vertex $v\in V(T)$ is the length of the unique path from the root to $v$.  The level of the root vertex is $0$. A vertex $v^\prime$ is called a child of $v$ if there is an edge from $v$ to $v^\prime$ and $l(v^\prime)=l(v)+1$. In this case, $v$ is called the parent of $v^\prime$.   A vertex $v^\prime$ is called a descendant of $v$ if there is a path from $v$ to $v^\prime$, with strictly increasing levels along the consecutive vertices lying on the path.
In the above example, the root has two children and five descendants.

\begin{dfn}
A rooted forest $F=\{T_1, \cdots, T_m\}$ is a disjoint union of rooted trees $T_1, \cdots, T_m$. Let 
$
V(F)$ denote the disjoint union of vertices of rooted trees of $F$. 
\end{dfn}

\begin{ex}
Here is an example of a rooted forest consisting of two rooted trees
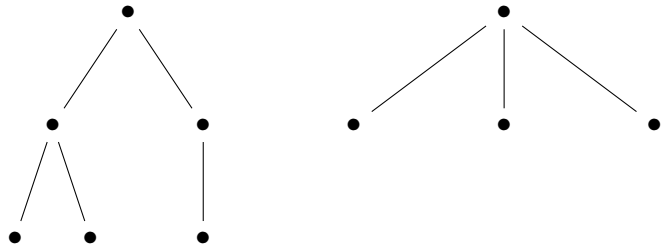
\begin{figure}[H]\centering
\begin{tikzpicture}[level distance=1.5cm,
  level 1/.style={sibling distance=2cm},
  level 2/.style={sibling distance=1 cm}]

  \node at (0,0) {$\bullet$}
    child {node {$\bullet$}
      child {node {$\bullet$}}
      child {node {$\bullet$}}
    }
    child {node {$\bullet$}
      child {node {$\bullet$}}
    };
    
    \node at (5,0) {$\bullet$}
     child {node {$\bullet$}
    }
    child{node{$\bullet$}}
    child {node {$\bullet$}
    };
\end{tikzpicture}
\caption{A rooted forest with two trees.}
\end{figure}
\end{ex}

\begin{dfn}
Let $F$ be a rooted forest with $n$ vertices. We define a normal marking of $F$ to be an one-to-one map
$$
\chi: V(F)\to \{1,2,\cdots, n\}
$$
such that 
$$
    \chi(v)< \chi(v^\prime)\quad \text{if $v^\prime$ is a child of $v$}\,. 
$$
An isomorphism between $(F_1, \chi_1)$ and $(F_2, \chi_2)$ is a graph isomorphism 
$$
g: F_1\to F_2
$$
such that the following diagram commutes 
$$
   \xymatrix{
      F_1 \ar[rr]^{g}\ar[dr]_{\chi_1}&& F_2 \ar[dl]^{\chi_2} \\
      &\{1,\cdots,n\}&
   }
$$
\end{dfn}

\begin{dfn}Let $\Gamma_n$ denote the isomorphism classes of pairs $(F, \chi)$ where
\begin{itemize}
\item $F$ is a rooted forest with $n$ vertices.
\item $\chi$ is a normal marking of $F$. 
\end{itemize}
\end{dfn}

Given $(F, \chi)\in \Gamma_n$, we order the rooted trees $T_1, \cdots, T_m$ of $F$ such that
$$
 \chi(rt(T_1))<\chi(rt(T_2))<\cdots< \chi(rt(T_m))\,. 
$$

\begin{ex}\label{exruningexample}
Fig. \ref{figure-forest} below gives an example of an element in $\Gamma_{10}$. 
The left tree is $T_1$ and the right tree is $T_2$. 
\begin{figure}[H]\centering 
\begin{tikzpicture}[level distance=1.5cm,
  level 1/.style={sibling distance=2cm},
  level 2/.style={sibling distance=1 cm}]
   
   \node at (0,-4) {$T_1$}; 
  
  \node at (0,0) {1}
    child {node {2}
      child {node {5}}
      child {node {10}}
    }
    child {node {6}
      child {node {8}}
    };
    
       \node at (5,-4) {$T_2$}; 
    
    \node at (5,0) {3}
     child {node {4}
    }
    child{node{7}}
    child {node {9}
    };
\end{tikzpicture}
\caption{A  forest with normal marking.}\label{figure-forest}
\end{figure}
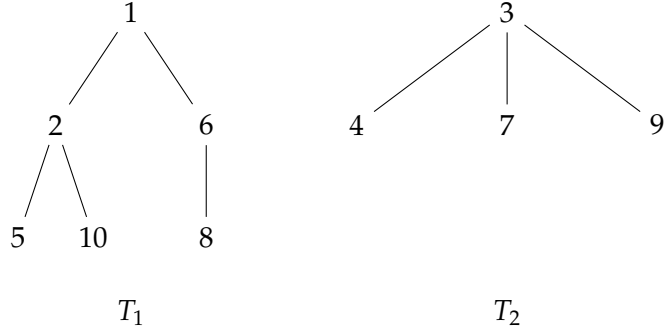
\end{ex}

For each vertex $v_0$ in $F$, it defines a rooted tree $T_{v_0}$ consisting of $v_0$ and all its descendants. Then $T_{v_0}$ is rooted at $v_0$. We define a residue operation 
$$
\oint_{T_{v_0}}
$$ 
as follows. Let $v_1,\cdots, v_k$ be all the children of $v_0$ ordered by
$$
\chi(v_1)<\chi(v_2)<\cdots< \chi(v_k)\,. 
$$
Then $\oint_{T_{v_0}}$ is recursively defined by
$$
 \oint_{T_{v_0}}:= 
     \bracket{\oint_{z_{\chi(v_0)}}dz_{\chi(v_1)}\oint_{T_{v_1}}} \cdots \bracket{\oint_{z_{\chi(v_0)}} dz_{\chi(v_k)}\oint_{T_{v_k}}}.
$$

For a normally marked forest $(F, \chi)$, we denote the following operation
$$
\int_{(F,\chi)}:= \bracket{\int_{A^+_{c_{i_1}}} dz_{{i_1}} \oint_{T_1}} \cdots \bracket{\int_{A^+_{c_{i_m}}} dz_{{i_m}} \oint_{T_m}}, \quad \text{where}\quad i_k=\chi(rt(T_k))\,.  
$$

Given $(F, \chi)$, we let $v\in V(F)$ with marking $\chi(v)=i$. We denote
$$
p_\chi(i)=\begin{cases}
\text{the marking of the parent of $v$} & \text{if $v$ is not a root}\,,\\
0 & \text{if $v$ is a root}\,. 
\end{cases}
$$ 
We assign the following rational number to $(F, \chi)$
$$
p(F, \chi):=\int_0^1 dx_1 \int_0^{x_{p_\chi(2)}}dx_2 \cdots\int_0^{x_{p_\chi(i)}}dx_i \cdots \int_0^{x_{p_\chi(n)}}dx_n\,, \quad x_0\equiv 1\,. 
$$

\begin{ex}
For Example \ref{exruningexample} in Fig. \ref{figure-forest}, one has
$$
\oint_{T_1}=\bracket{\oint_{z_1}dz_2 \oint_{z_2}dz_5 \oint_{z_2}dz_{10}} \bracket{\oint_{z_1}dz_6\oint_{z_6}dz_8}, \quad \oint_{T_2}=\oint_{z_3} dz_4 \oint_{z_3}dz_7 \oint_{z_3}dz_9\,. 
$$
$$
\int_{(F,\chi)}=\bracket{\int_{A^+_{c_1}}dz_1 {\oint_{z_1}dz_2 \oint_{z_2}dz_5 \oint_{z_2}dz_{10}} 
{\oint_{z_1}dz_6\oint_{z_6}dz_8} } \bracket{ \int_{A^+_{c_3}}dz_3 \oint_{z_3} dz_4 \oint_{z_3}dz_7 \oint_{z_3}dz_9}\,.
$$
$$
p(F,\chi)=\int_0^1 dx_1 \int_0^{x_1}dx_2 \int_0^1 dx_3 \int_0^{x_3}dx_4 \int_0^{x_2}dx_5 \int_0^{x_1}dx_6 \int_0^{x_3}dx_7\int_0^{x_6}dx_8\int_0^{x_3}dx_9 \int_0^{x_2} dx_{10}={1\over 144}\,. 
$$
\end{ex}

It is easy to see that $p(F,\chi)$ does not depend on the choice of $\chi$. Therefore we write
$$
p(F)=p(F, \chi)\,. 
$$
In particular, we can define $p(T)$ for any rooted tree.\footnote{The quantity $1/p(T)$ also appears as the tree factorial in \cite{Kreimer:2000}.
We thank the anonymous referee for pointing this out.} 
If $F=\{T_1,\cdots, T_m\}$, then
$$
p(F)=p(T_1)\cdots p(T_m)\,. 
$$
The following lemma gives a useful recursive formula for $p(T)$. 

\begin{lem}\label{lem-p}
Let $T$ be a rooted tree. Let $v_1,\cdots,v_m$ be all the children of the root vertex. Then
$$
p(T)={1\over |V(T)|} p(T_{v_1})\cdots p(T_{v_m})\,.
$$
Recall  $T_{v_i}$ is the rooted tree consisting of $v_i$ and all its descendants in $T$.
\end{lem}
\begin{proof} Let $n_i=|V(T_{v_i})|$. From the integration formula, it is easy to see that
$$
p(T)=\int_0^1 dx \prod_{i=1}^m\bracket{x^{n_i} p(T_{v_i})}={1\over 1+\sum\limits_{i=1}^m n_i}p(T_{v_1})\cdots p(T_{v_m})={1\over |V(T)|} p(T_{v_1})\cdots p(T_{v_m})\,.
$$

\end{proof}

The next lemma is the key combinatorial formula for computing the holomorphic limit. 
\begin{lem} \label{Claim 2}
Let $\Phi(z_1,\cdots,z_n;\tau)$ be a meromorphic elliptic function on $\mathbb{C}^{n}\times \H$. Then
$$
\lim_{\bar\tau\to \infty}\dashint_{\square_{c_1}} {d^2z_{1}} \cdots \dashint_{\square_{c_n}}  {d^2z_{n}} \Phi=\sum_{(F,\chi)\in \Gamma_n} p(F) \int_{(F,\chi)}\Phi\,. 
$$
\end{lem}
\begin{proof} This follows by applying Lemma \ref{lem-integral} $n$ times. In each step, we only need to keep the first two terms on the right hand side in
$$
  \dashint_{E_\tau} {d^2z \over \im \tau} \Psi = \sum_k {1\over k+1} \int_{A_c^+} dz \,{\Psi_k(z)}+\sum_{k} {1\over k+1} \sum_{w\in D}\bracket{{\im w\over \im \tau}}^{k+1}\oint_w dz\,  {\Psi_k(z) } + \mathcal{O}({1\over \im \tau})\,. 
$$ 
The operation $\int_{(F,\chi)}\Phi$ keeps track of the residues and $A$-cycle integrals. The combinatorial factor $p(F)$ keeps track of those ${1\over k+1}$-factors that appear at each step of integration. 
\end{proof}

We can further express this combinatorial formula in terms of the ordered $A$-cycle integrals as defined in Definition \ref{dfnorderedAcycleintegral}.
Recall from Lemma
\ref{lem-A-commutator} that

$$
 \int_A dz_1 \oint_{z_1}dz_2 \oint_{z_1} dz_3 \cdots \oint_{z_1}dz_k
 =\bbracket{\int_A dz_2, \bbracket{\int_A dz_{3},\cdots, \bbracket{\int_A dz_k, \int_A dz_1}}}\,.
$$

Let $R_n=\C\abracket{x_1,\cdots,x_n}$ be the tensor algebra in $n$ variables $x_1,\cdots, x_n$.  The generators $x_i$'s do not commute with each other, and each element in $R_n$ is expressed in terms of a linear combination of words in $x_i$'s. For any $a, b\in R_n$, we denote
$$
[a,b]:= a b- ba\,. 
$$

Let $(F, \chi)\in \Gamma_n$, $v_0\in V(F)$ and $T_{v_0}$ be the tree rooted at $v_0$ as defined above. We define 
$$
x_{T_{v_0}}\in R_n
$$ 
as follows. Let $v_1,\cdots, v_k$ be all the children of $v_0$ such that
$$
\chi(v_1)<\chi(v_2)<\cdots< \chi(v_k)\,. 
$$
Then $x_{T_{v_0}}$ is recursively defined by
$$
 x_{T_{v_0}}:= {\bbracket{x_{T_{v_1}},\cdots,\bbracket{x_{T_{v_k}},x_{\chi(v_0)}}}}\,. 
$$

Let $T_1, \cdots, T_m$ be rooted trees of $F$ ordered by the marking as  above. We define
$$
x_{(F, \chi)}:= x_{T_1}\cdots x_{T_m}\,. 
$$
\begin{ex}
For Example \ref{exruningexample} in Fig. \ref{figure-forest}, one has
$$
x_{T_1}=[ [x_5,[x_{10},x_2]], [[x_8,x_6],x_1]], \quad x_{T_2}=[x_4,[x_7,[x_9,x_3]]]\,,
$$
and
$$
x_{(F, \chi)}=x_{T_1}x_{T_2}\,. 
$$
\end{ex}

Given such $x_{(F, \chi)}$, we associate an ordered $A$-cycle integral 
$$
x_{(F, \chi)}(\oint_A)
$$
by the substitution
$$
x_i\mapsto \oint_A dz_i\,. 
$$
 Lemma \ref{Claim 3} below is basically a reformulation of Lemma \ref{Claim 2} by the above discussion. The ordering of $A$-cycles comes from
 the ordering in the iterated regularized integrals adapted to our choice of $c_i$'s. 

\begin{lem}\label{Claim 3}
Let $\Phi(z_1,\cdots,z_n;\tau)$ be a meromorphic elliptic function on $\mathbb{C}^{n}\times \H$. Then
one has
$$
\lim_{\bar\tau\to \infty}\dashint_{\square_{c_1}} {d^2z_{1}} \cdots \dashint_{\square_{c_n}}  {d^2z_{n}} \Phi=\sum_{(F,\chi)\in \Gamma_n} p(F)  x_{(F, \chi)}(\oint_A) \Phi\,. 
$$
\end{lem}

Lemma \ref{Claim 4}  below together with Lemma \ref{Claim 3} will prove statement (2) of Theorem \ref{thm-2d}. 

\begin{lem}\label{Claim 4}  
One has
$$
\sum_{(F,\chi)\in \Gamma_n}p(F)   x_{(F, \chi)}={1\over n!}\sum_{\sigma\in S_n} x_{\sigma(1)}x_{\sigma(2)}\cdots x_{\sigma(n)}\,. 
$$
\end{lem}
\begin{ex}\label{expermutationsumn=3}
Here is an illustration of Lemma \ref{Claim 4} when $n=3$. There are 6 elements of $\Gamma_3$
$$
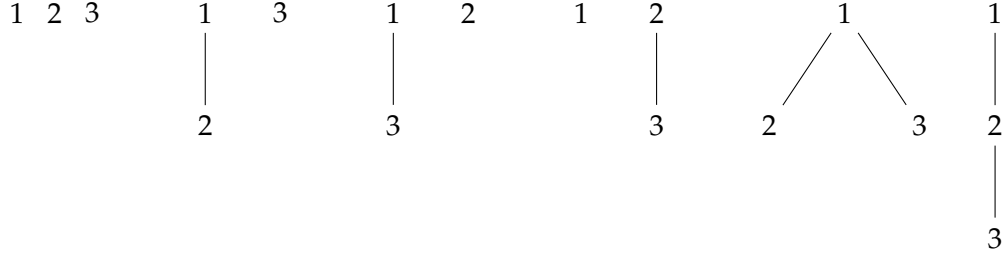
\begin{tikzpicture}[level distance=1.5cm,
  level 1/.style={sibling distance=2cm},
  level 2/.style={sibling distance=1 cm}]

\node at (0,0){1};
\node at (0.5,0){2};
\node at (1,0){3};

  \node at (2.5,0) {1}
  child{node{2}};
    \node at (3.5,0) {3};
    
      \node at (5,0) {1}
  child{node{3}};
    \node at (6,0) {2};
    
      \node at (7.5,0) {1};
      \node at (8.5,0){2}
      child{node{3}};
     
     \node at (11,0){1}
     child{node{2}}
     child{node{3}};
     
     \node at (13,0){1}
     child{node{2}
     child{node{3}}
     };
    
\end{tikzpicture}
$$
The sum $\sum_{(F,\chi)\in \Gamma_n} p(F)  x_{(F, \chi)}$ gives 
\begin{align*}
&x_1x_2x_3+{1\over 2} [x_2,x_1]x_3+ {1\over 2}[x_3,x_1]x_2+{1\over 2} x_1[x_3,x_2]+{1\over 3} [x_2,[x_3,x_1]]+{1\over 6}[[x_3,x_2],x_1]\\
=&{1\over 6}\bracket{x_1x_2x_3+ x_1x_3x_2+x_2x_1x_3+x_2x_3x_1+x_3x_1x_2+x_3x_2x_1}.
\end{align*}
\end{ex}

We are only left to prove Lemma \ref{Claim 4}.  

\subsubsection*{Proof of Lemma \ref{Claim 4}}
We prove Lemma \ref{Claim 4} by induction on $n$. We first need a recursion formula. 
\begin{lem}\label{lem-ad} 
Assume Lemma \ref{Claim 4} holds for $n-1$. Then
$$
\sum_{\substack{(F,\chi)\in \Gamma_n\\ F=T\ \text{is a rooted tree}}}p(F) x_{(F,\chi)} ={1\over n !}\sum_{\substack{\sigma: \{2,\cdots,n\}\to \{2,\cdots,n\}\\ \text{$\sigma$ is a permutation}}} [x_{\sigma(2)},[x_{\sigma(3)},\cdots,[x_{\sigma(n)},x_1]]]\,. 
$$
\end{lem}
\begin{proof} Given $F=T$, let $v_1, \cdots, v_k$ be all the children of the root. By Lemma \ref{lem-p}, 
$$
p(T)={p(T_{v_1})\cdots p(T_{v_k})\over n}\,. 
$$
Introduce new variables $y_1, \cdots, y_{n-1}$. Consider a word in $y$ defined the same as that for $x$
$$
\sum_{(F^\prime,\chi^\prime)\in \Gamma_{n-1}} p(F^\prime)  y_{(F^\prime, \chi^\prime)}\,.
$$
Let $ad_{x_i}$ denote the operator $ad_{x_i}=[x_i,-]$.  Let $\Xi$ denote the substitution
$$
y_1\mapsto ad_{x_2},\quad y_2\mapsto ad_{x_3}, \quad \cdots, y_{n-1}\mapsto ad_{x_n}\,. 
$$
Using $[ad_{x_i}, ad_{x_j}]=ad_{[x_i,x_j]}$, it is not hard to see that 
$$
\sum_{\substack{(F,\chi)\in \Gamma_n\\ F=T\ \text{is a rooted tree}}}p(F) x_{(F,\chi)} ={1\over n} \Xi\bracket{\sum_{(F^\prime,\chi^\prime)\in \Gamma_{n-1}} p(F^\prime)  y_{(F^\prime, \chi^\prime)}}x_1\,.  
$$
Assume Lemma \ref{Claim 4} holds for $n-1$. Then
$$
{\sum_{(F^\prime,\chi^\prime)\in \Gamma_{n-1}} p(F^\prime)  y_{(F^\prime, \chi^\prime)}}={1\over (n-1)!}\sum_{\sigma\in S_{n-1}} y_{\sigma(1)}y_{\sigma(2)}\cdots y_{\sigma(n)}\,. 
$$
This proves the lemma. 
\end{proof}

We can now proceed to prove Lemma \ref{Claim 4}. 

Let $\sigma\in S_n$ be a permutation of $\{1,2,\cdots, n\}$. $\sigma$ can be expressed in cyclic notation by 
$$
\sigma=(a_1 \cdots a_{i_1})  (a_{i_1+1}\cdots a_{i_2})\cdots (a_{i_{k-1}+1}\cdots a_{i_k}), \quad i_k=n\,. 
$$
We say the above cyclic expression is \textbf{ordered} if
$$
a_1=\min\{a_1,\cdots, a_{i_1}\}, \quad a_{i_1+1}=\min\{a_{i_1+1},\cdots, a_{i_2}\}, \quad \cdots \quad, a_{i_{k-1}+1}=\min\{a_{i_{k-1}+1},\cdots, a_n\}
$$
and
$$
a_1< a_{i_1+1}<a_{i_2+1}<\cdots< a_{i_{k-1}+1}\,. 
$$
Each $\sigma$ has a unique ordered cyclic expression.  We denote the following number
$$
 |\sigma|:= { i_1!(i_2-i_1)!\cdots (i_k-i_{k-1})!}\,.
$$
Given  $(j_1 j_2\cdots j_k)$, we denote
$$
x_{(j_1)}:=x_{j_{1}}\,,\quad
x_{(j_1,\cdots,j_k)}:= [x_{j_k},[x_{j_{k-1}},\cdots, [x_{j_2},x_{j_1}]]]\,. 
$$
We also denote
$$
x_{\sigma}:=x_{(a_1 \cdots a_{i_1})}  x_{(a_{i_1+1}\cdots a_{i_2})}\cdots x_{(a_{i_{k-1}+1}\cdots a_{i_k})}\in R_n\,. 
$$

Let $\Omega_k$ denote the set of partitions of $\{1,\cdots,n\}$ into $k$ subsets. We write each $\omega\in\Omega_k$ as
$$
\omega= I_1 \cup I_2 \cup \cdots I_k
$$
where the index is ordered in such a way that 
$$
   \min_{i\in I_1} i < \min_{i\in I_2} i <\cdots < \min_{i \in I_k} i\,.
$$
For such partition $\omega= I_1 \cup I_2 \cup \cdots \cup I_k$, let
$$
\Gamma_n^\omega \subset \Gamma_n
$$
be those normally marked forest $(F,\chi)$ that consists of $k$ rooted trees $T_1, \cdots, T_k$ and
$$
\chi: V(T_1)\mapsto I_1\,, \quad V(T_2)\mapsto I_2\,, \quad \cdots \quad\,, V(T_k)\mapsto I_k\,. 
$$
Then we have 
$$
\sum_{(F,\chi)\in \Gamma_n}p(F)   x_{(F, \chi)}=\sum_{k} \sum_{\omega\in \Omega_k}\sum_{\substack{(F,\chi)\in \Gamma^\omega_n\\ F=\{T_1,\cdots,T_k\}}}  \bracket{p(T_1)x_{T_1}} \bracket{p(T_2)x_{T_2} }\cdots \bracket{p(T_k)x_{T_k}}\,. 
$$

Apply Lemma \ref{lem-p} and Lemma \ref{lem-ad} to each partition in $\omega$, we find
$$
\sum_{k} \sum_{\omega\in \Omega_k}\sum_{\substack{(F,\chi)\in \Gamma^\omega_n\\ F=\{T_1,\cdots,T_k\}}}  \bracket{p(T_1)x_{T_1}} \bracket{p(T_2)x_{T_2} }\cdots \bracket{p(T_k)x_{T_k}}=\sum_{\sigma\in S_n} {1 \over |\sigma|}x_\sigma\,. 
$$
By Proposition \ref{prop-identity} in Appendix \ref{appendixalgebraicidentity},  we have the combinatorial formula
$$
\sum_{\sigma\in S_n} {1 \over |\sigma|}x_\sigma={1\over n!}\sum_{\sigma\in S_n} x_{\sigma(1)}x_{\sigma(2)}\cdots x_{\sigma(n)}\,. 
$$
It follows that 
$$
\sum_{(F,\chi)\in \Gamma_n}p(F)   x_{(F, \chi)}= {1\over n!}\sum_{\sigma\in S_n} x_{\sigma(1)}x_{\sigma(2)}\cdots x_{\sigma(n)}\,. 
$$
This finishes the induction step for Lemma \ref{Claim 4}, hence completes the proof of Theorem \ref{thm-2d}.

\appendix

\section{Modular forms and elliptic functions}
\label{secmodularformsellipticfunctions}

\subsection*{Modular forms}
 
 Modular forms are functions on the upper-half plane $\H$ that enjoy nice transform properties under the action of 
 $\mathrm{SL}_{2}(\mathbb{Z})<\mathrm{SL}_{2}(\mathbb{R})=\mathrm{Aut} \,\H$.
 Quasi-modular forms and almost-holomorphic modular forms are generalizations of modular forms.
 Readers who are not familiar with these notions are referred to \cite{Kaneko:1995, Zagier:2008} for details.
 Here we only collect some basic definitions and facts that are frequently used in this paper.
 
 \begin{dfn}\label{dfnmodularquasimodular}
  Let $f$ be a meromorphic function on $\H$. 
 \begin{itemize}
\item [(1)] $f$  is said to be \textbf{modular} of weight $k$ if
  $$f({a\tau+b\over c\tau+d})=(c\tau+d)^{k}f(\tau)\,,\quad \forall \,
 \begin{pmatrix}
 a & b\\
 c &d
 \end{pmatrix}\in\mathrm{SL}_{2}(\mathbb{Z})\,. $$
 \item [(2)] $f$ is said to be \textbf{quasi-modular} of weight $k$ (and depth $\ell$) if there exist holomorphic functions $f_{1},\cdots ,f_{\ell}$ on $\H$ such that
 $$f({a\tau+b\over c\tau+d})=(c\tau+d)^{k}f(\tau)
 +\sum_{i=1}^{\ell} c^{k-i}(c\tau+d)^{k-i}f_{i}(\tau)\,,\quad \forall \,
 \begin{pmatrix}
 a & b\\
 c &d
 \end{pmatrix}\in\mathrm{SL}_{2}(\mathbb{Z})\,. $$
 \end{itemize}
 \end{dfn}
 
 \begin{dfn}
 A function $f$ on $\H$ is called a \emph{holomorphic modular form} of weight $k$ if 
 \begin{enumerate}[label=(\roman*)]
  \item \label{dfnmodularcondition1}
 $f$ is holomorphic on $\H$.
 \item 
  \label{dfnmodularcondition2}
  $f$ is modular of weight $k$.
 \item \label{dfnmodularcondition3}
 $f$ has sub-exponential growth at $\tau=i\infty$ in the sense that 
 $f(\tau)=\mathcal{O}(e^{C \im \tau})$ as $\im \tau\rightarrow \infty$, for any $C>0$.
 \end{enumerate}
  \end{dfn}
  
 \begin{dfn}\label{dfnquasimodularform}
 If $f$ satisfies  \ref{dfnmodularcondition1}, \ref{dfnmodularcondition3} above, 
and
 \begin{enumerate}
\item[\mylabel{dfnquasimodulartransformationlaw}{(ii')}]
$f$ is quasi-modular of weight $k$ (and depth $\ell$). 
  \end{enumerate}
Then it is called a   \emph{holomorphic quasi-modular form} of weight $k$ (and depth $\ell$).
If $f$ only satisfies \ref{dfnmodularcondition1} and \ref{dfnquasimodulartransformationlaw}, then
it is called a \emph{weakly holomorphic quasi-modular form} of weight $k$
(and depth $\ell$).
  \end{dfn}

 \begin{dfn}
 If $f$ satisfies \ref{dfnmodularcondition3} above, and
  \begin{enumerate}[label=(\roman*)]
\item[\mylabel{dfnalmostholomorphic}{(i'')}]
 $f$ is almost-holomorphic on $\H$: $f\in\OO_{\H}[{1\over \im\tau}]$ (Definition 
\ref{dfn-holomorphiclimit}). 
\item[\mylabel{dfnalmostholomorphic}{(ii'')}]
  $f$ is modular in the sense that
  $$f\left({a\tau+b\over c\tau+d} ,\overline{({a\tau+b\over c\tau+d})}\right)=(c\tau+d)^{k}f(\tau,\bar{\tau})
 \,,\quad \forall \,
 \begin{pmatrix}
 a & b\\
 c &d
 \end{pmatrix}\in\mathrm{SL}_{2}(\mathbb{Z})\,. $$
 \end{enumerate}
Then it is called an   \emph{almost-holomorphic modular form} of weight $k$.
 \end{dfn}
 
Of central importance 
are
the Eisenstein series that are defined by
\begin{eqnarray*}
E_{2k}(\tau)&=&{1\over 2\zeta(2k) }\sum_{(m,n)\in \mathbb{Z}^{2}-\{(0,0)\}}
{1\over (m\tau+n)^{2k}}\,,\quad k\geq 2\,,\\
E_{2}(\tau)&=&{1\over 2\zeta(2) }
\left(
\sum_{n\neq 0}{1\over n^{2}}+
\sum_{m\neq 0}\sum_{n\in\mathbb{Z}}{1\over (m\tau+n)^{2}}\right)\,,
\end{eqnarray*}
where $\zeta(2k),k\geq	 1$ are the zeta-values.
Define also
\begin{equation*}\label{eqn-E2hat}
\widehat{E}_{2}(\tau,\bar{\tau})=E_{2}(\tau)-{3\over \pi }{1\over \im \tau}\,.
\end{equation*}
The Eisenstein series admit Fourier expansions in $q=\exp(2\pi i \tau)$ given by 
\begin{equation}\label{eqn-FourierexpansionsE2k}
E_{2k}(\tau)=
1-{4k\over B_{2k}}
\sum_{d\geq 1}\sigma_{2k-1}(d)q^{d}
=1-{4k\over B_{2k}}
\sum_{m\geq 1}{m^{2k-1}q^{m}\over 1-q^{m}}\,,\quad 
q=e^{2\pi i\tau}\,,\quad
k\geq 1\,.
\end{equation}
They have the following transformations under the action of $\gamma= \begin{pmatrix}
 a & b\\
 c &d
 \end{pmatrix}\in \mathrm{SL}_{2}(\mathbb{Z})$
\begin{eqnarray}\label{eqn-transformationsofmodularforms}
E_{2k}(\gamma\tau)&=&(c\tau+d)^{2k} E_{2k}(\tau)\,, \quad k\geq 2\nonumber\\
E_{2}(\gamma\tau)&=&(c\tau+d)^{2}E_{2}(\tau)+{12\over 2\pi i} c(c\tau+d)\,,\nonumber\\
\widehat{E}_{2}(\gamma\tau, \overline{\gamma\tau})&=&(c\tau+d)^{2}\widehat{E}_{2}(\tau,\bar{\tau})\,,
\quad\quad\quad\quad\quad\quad\,.
\end{eqnarray}

The spaces of holomorphic modular, holomorphic quasi-modular,  almost-holomorphic modular forms for $\mathrm{SL}_{2}(\mathbb{Z})$ form graded rings.
Denote these rings
by $M,\widetilde{M},\widehat{M}$ respectively, then
\begin{equation}\label{eqn-ringpresentations}
M\cong \mathbb{C}[E_4, E_6]\,,\quad
\widetilde{M}\cong \mathbb{C}[E_{2}, E_4, E_6]
\,,\quad
\widehat{M}\cong \mathbb{C}[\widehat{E}_{2}, E_4, E_6]\,.
\end{equation}
The ring $\widetilde{M}=\mathbb{C}[E_{2}, E_{4},E_{6}]$ is furthermore a differential ring
under ${1\over 2\pi i}\partial_{\tau}=q{d\over dq}$.
The following relations are known as the Ramanujan identities
\begin{equation}\label{eqn-Ramanujanidentities}
{1\over 2\pi i}\partial_{\tau} E_{2}={1\over 12}(E_{2}^2-E_{4})\,,\quad
{1\over 2\pi i}\partial_{\tau} E_{4}={1\over 3}(E_{2}E_4-E_{6})\,,\quad
{1\over 2\pi i}\partial_{\tau} E_{6}={1\over 2}(E_{2}E_6-E_{4}^2)\,.
\end{equation}
See \cite{Kaneko:1995} for details on these.

Since the generator $1/\im \tau$ is algebraically independent over
the ring $\OO_\H$, the notion of holomorphic limit  in Definition 
\ref{dfn-holomorphiclimit}
is well-defined.
This notion
plays an important role in discussing the relation between quasi-modular and almost-holomorphic modular forms.
\begin{thm}[Kaneko-Zagier  \cite{Kaneko:1995}]\label{thmstructuretheorem}
The holomorphic limit 
\begin{equation*}\label{thmstructuretheorem}
\lim_{\bar\tau \to \infty}: \widehat{M} \longrightarrow \widetilde{M}
\end{equation*}
induces
a graded ring isomorphism between $\widetilde{M} $ and $ \widehat{M} $.
The inverse is called modular completion.
\end{thm}

It is straightforward to check that Theorem \ref{thmstructuretheorem} can be generalized to give an isomorphism  (again by the holomorphic limit $\lim\limits_{\bar\tau \to \infty}$) between 
the space of modular functions in $\mathfrak{M}_{\H}[{1\over \im \tau}]$
and the space of quasi-modular, meromorphic functions on $\H$.

\subsection*{Elliptic functions}

Elliptic functions with respect to the lattice
$\Lambda_{\tau}=\mathbb{Z}\oplus \mathbb{Z}\tau$ 
 are meromorphic functions on $\mathbb{C}$ that are invariant under the translation by the lattice.
 They are 
  pull-backs of meromorphic functions on 
 $E_{\tau}=\mathbb{C}/\Lambda_\tau$.
 It is a classical fact that the functional field of $E_{\tau}$ is given by
 \begin{equation}\label{eqn-functionalfieldofellipticcurve}
 k(E_{\tau})=\mathbb{C}(\wp(z), \wp'(z))/ 
 \langle (\wp'(z))^2=4\wp(z)^3-g_{2}(\tau)\wp(z)-g_{3}(\tau) \rangle\,,
 \end{equation}
 where $\wp(z)$ is the Weierstrass $\wp$-function, $\wp'=\partial_{z}\wp$, and
 \begin{equation}\label{eqn-g2g3}
g_{2}(\tau)={4\over 3}\pi^4E_{4}(\tau)\,,\quad g_{3}={8\over 27}\pi^6 E_{6}(\tau)
\end{equation}
are holomorphic modular forms of weight $4,6$ respectively.

The elliptic function $\wp(z;\tau)$ is even in $z$, with order 2 poles along $\Lambda_{\tau} $ on the universal cover $\mathbb{C}$.
Explicitly one has
\begin{equation}\label{eqn-wponC}
\wp(z;\tau)={1\over z^2}+\sum_{\lambda\in \Lambda_{\tau}-(0,0)} \bracket{{1\over (z+\lambda)^2} -{1\over \lambda^2}}
=
{1\over z^2}+{g_{2}(\tau)\over 20} z^2+{g_{3}(\tau)\over 28}z^{4}+\cdots
\end{equation}
The  elliptic functions $\wp,\wp'$ are also modular under the action of $\gamma= \begin{pmatrix}
 a & b\\
 c &d
 \end{pmatrix}\in \mathrm{SL}_{2}(\mathbb{Z})$
\begin{eqnarray*}\label{eqn-transformationsofwp}
\wp( {z\over c\tau+d};{a\tau+b\over c\tau+d})
&=&(c\tau+d)^{2}\wp(z;\tau)\,,\nonumber\\
\wp'( {z\over c\tau+d}; {a\tau+b\over c\tau+d})
&=&(c\tau+d)^{3}\wp'(z;\tau)\,.
\end{eqnarray*}
The Fourier expansion of the meromorphic function $\wp(z;\tau)$ is given by 
\begin{equation}\label{eqnwponC*}
\wp(u;q)=(2\pi i)^2\sum_{k\geq 1} {k u^{k}\over 1-q^{k}}
+(2\pi i)^2\sum_{k\geq 1} {k q^{k} u^{-k}\over 1-q^{k}}-{\pi^2\over 3}E_{2}\,,\quad
|q|<|u|<1\,,
\end{equation}
where $u=e^{2\pi i z},q=e^{2\pi i\tau}$.
See the textbooks \cite{Lang:1985, Silverman:2009} for more details.

Throughout this work we often suppress the arguments $z,\tau,u,q$ etc. in the functions when no confusion should arise.

When integrating elliptic functions, one naturally encounters
the so-called
quasi-elliptic functions \cite{Zagier:1991, Libgober:2009, Goujard:2016counting}.
This notion is derived from the notion of quasi-Jacobi forms of index zero \cite{Eichler:1985, Libgober:2009, Dabholkar:2012, Goujard:2016counting}. In this work we only need the special case, namely the Weierstrass zeta function given by
$$\zeta(z)={1\over z}+\sum_{\lambda\in \Lambda_{\tau}-\{(0,0)\}}   \bracket{{1\over z+\lambda}-{1\over \lambda}+{z\over \lambda^2}}\,.
$$
It satisfies $\partial_{z}\zeta=-\wp$.
Under the elliptic transformations $z\mapsto z+1,z\mapsto z+\tau$ one has
\begin{eqnarray*}
\zeta(z+1)-\zeta(z)=\eta_{1}\,,\quad \zeta(z+\tau)-\zeta(z)=\eta_{2}\,,\quad
\forall z\neq 0\,.
\end{eqnarray*}
Here $\eta_{1},\eta_{2}$ are given by 
\begin{equation*}\label{eqn-quasiperiodsasquasimodularforms}
\eta_{1}(\tau)={\pi^2\over 3}E_{2}(\tau)\,,\quad
\eta_{2}(\tau)={\pi^2\over 3}{1\over \tau}E_{2}(-{1\over \tau})\,.
\end{equation*}
From the transformation of $E_{2}$ given in \eqref{eqn-transformationsofmodularforms}, one sees that
 $Z:=\zeta-z\eta_{1}$ satisfies
 \begin{eqnarray*}\label{eqn-quasiellipticityofZ}
Z(z+1)-Z(z)=0\,,\quad Z(z+\tau)-Z(z)=-2\pi i\,.
\end{eqnarray*}

The de Rham cohomology $H^{1}_{dR}(E_{\tau},\mathbb{C})$ is generated by the cohomology classes of the Abelian differentials
$dz,\wp(z)dz$.
With respect to the canonical representatives $\{A,B\}$ mentioned earlier, the integrals of the former are given by the periods
\begin{equation*}\label{eqn-periods}
\int_{A}dz=1\,,\quad \int_{B}dz=\tau\,.
\end{equation*}
While the integrals of the latter are called quasi-periods and are given by
\begin{equation}\label{eqn-quasiperiods}
\int_{A}\wp(z;\tau)dz=-\eta_{1}(\tau)\,,\quad \int_{B}\wp(z;\tau)dz=-\eta_{2}(\tau)\,.
\end{equation}
They satisfy the Legendre period relation (i.e., the 1st Riemann-Hodge bilinear relation)
\begin{equation*}\label{eqn-Legendreperiodrelation}
-\eta_{2}(\tau)+\eta_{1}(\tau)\tau=2\pi i\,. 
\end{equation*}
which is equivalent to the quasi-modularity of $E_{2}(\tau)$ given in \eqref{eqn-transformationsofmodularforms}.
See \cite{Katz:1976p} for a nice account on this.

\section{An algebraic identity}
\label{appendixalgebraicidentity}

In this appendix, we present an algebraic identity that  is used in the proof of our main Theorem \ref{thm-2d}.  We first recall the following notations in the proof of Theorem \ref{thm-2d}. 

Let $\sigma\in S_n$ be a permutation of $\{1,2,\cdots, n\}$. Then $\sigma$ can be expressed in cyclic notation by 
$$
\sigma=(a_1 \cdots a_{i_1})  (a_{i_1+1}\cdots a_{i_2})\cdots (a_{i_{k-1}+1}\cdots a_{i_k}), \quad i_k=n\,. 
$$
We say the above cyclic expression is ordered if
$$
a_1=\min\{a_1,\cdots, a_{i_1}\}, \quad a_{i_1+1}=\min\{a_{i_1+1},\cdots, a_{i_2}\}, \quad \cdots \quad, a_{i_{k-1}+1}=\min\{a_{i_{k-1}+1},\cdots, a_n\}
$$
and
$$
a_1< a_{i_1+1}<a_{i_2+1}<\cdots< a_{i_{k-1}+1}\,. 
$$
Each $\sigma$ has a unique ordered cyclic expression.  We denote the following number
$$
 |\sigma|:= { i_1!(i_2-i_1)!\cdots (i_k-i_{k-1})!}\,.
$$

Let $R_n=\C\abracket{x_1,\cdots,x_n}$ be the tensor algebra in $n$ variables $x_1,\cdots, x_n$.  The generators $x_i$'s do not commute with each other, and each element in $R_n$ is expressed in terms of a linear combination of words in $x_i$'s. For any $a, b\in R_n$, we denote
$$
[a,b]:= a b- ba\,. 
$$

Given  $(j_1 j_2\cdots j_k)$, we denote
$$
x_{(j_1)}:= x_{j_1}\,,\quad 
x_{(j_1,\cdots,j_k)}:= [x_{j_k},[x_{j_{k-1}},\cdots, [x_{j_2},x_{j_1}]]]\,. 
$$
Given $\sigma$ with its ordered cyclic expression as above, we write
$$
x_{\sigma}:=x_{(a_1 \cdots a_{i_1})}  x_{(a_{i_1+1}\cdots a_{i_2})}\cdots x_{(a_{i_{k-1}+1}\cdots a_{i_k})}\in R_n\,. 
$$

\begin{ex} 
Let $n=3$. We are interested in the sum $\sum_{\sigma\in S_n} {x_\sigma\over |\sigma|}$ which is given by
\begin{align*}
\sum_{\sigma\in S_3} {x_\sigma\over |\sigma|}=&x_1x_2x_3+{1\over 2}[x_2,x_1]x_3+{1\over 2}[x_3,x_1]x_2+{1\over 2}x_1[x_3,x_2]+{1\over 6}[x_3,[x_2,x_1]]+{1\over 6}[x_2,[x_3,x_1]]\\
=& {1\over 6}\bracket{x_1x_2x_3+ x_1x_3x_2+x_2x_1x_3+x_2x_3x_1+x_3x_1x_2+x_3x_2x_1}\\
=&{1\over 3!}\sum_{\sigma\in S_3} x_{\sigma(1)}x_{\sigma(2)}x_{\sigma(3)}. 
\end{align*}
\end{ex}

The combinatorial identity that we find in this example actually holds in general. The next proposition might be known to experts. Since we couldn't locate a precise reference, in what follows we
supply a proof for completeness.

\begin{prop}\label{prop-identity}The following identity holds in $R_n$
$$
\sum_{\sigma\in S_n} {x_\sigma\over |\sigma|}={1\over n!} \sum_{\sigma\in S_n} x_{\sigma(1)}x_{\sigma(2)}\cdots x_{\sigma(n)}\,. 
$$

\end{prop}
\begin{proof} We prove by induction on $n$. 

Let $I$ be the ideal in $R_n$ generated by elements of the form 
$$
\cdots x_i\cdots x_i\cdots
$$
i.e., those words where some variable $x_i$ has appeared at least twice. Let 
$$
G_n=R_n/I
$$
be the quotient ring. Let $\bar x_i$ be the corresponding generator in $G_n$. We only need to prove
$$
\sum_{\sigma\in S_n} {\bar x_\sigma\over |\sigma|}={1\over n!} \sum_{\sigma\in S_n} \bar x_{\sigma(1)}\bar x_{\sigma(2)}\cdots \bar x_{\sigma(n)}\quad \text{holds in $G_n$\,.}
$$

 Each element $f\in G_n$ can be written as
$$
f=f_{(0)}+f_{(1)}+\cdots f_{(n)}
$$
where $f_{(k)}$ is homogeneous of degree $k$ in $\bar x_i$'s. We first observe that 
$$
{1\over n!} \sum_{\sigma\in S_n} \bar x_{\sigma(1)}\bar x_{\sigma(2)}\cdots \bar x_{\sigma(n)}= \bracket{e^{\bar x_1+\cdots+ \bar x_n}}_{(n)}. 
$$
Introduce a variable $t$. Then in $G_n$, we have
$$
 \bracket{e^{\bar x_1+\cdots+ \bar x_n}}_{(n)}= {\pa\over \pa t}  \bracket{e^{t\bar x_1+\bar x_2+\cdots+ \bar x_n}}_{(n)}\,.
$$
By Duhamel’s formula, we have
$$
{\pa\over \pa t}  \bracket{e^{t\bar x_1+\bar x_2+\cdots+ \bar x_n}}=\int_0^1 ds e^{s(t\bar x_1+\bar x_2+\cdots+ \bar x_n)} \bar x_1 e^{-s(t\bar x_1+\bar x_2+\cdots+ \bar x_n)}e^{t\bar x_1+\bar x_2+\cdots+ \bar x_n}\,.
$$
Let us write $Y=\bar x_2+\cdots \bar x_n$. Using the quotient relation in $G_n$, we find 
\begin{align*}
{\pa\over \pa t}  \bracket{e^{t\bar x_1+\bar x_2+\cdots+ \bar x_n}}&=\int_0^1 ds e^{s Y}\bar x_1 e^{-s Y} e^{Y}=\sum_{k\geq 0}\int_0^1 ds {s^k\over k!} [\underbrace{Y, \cdots, [Y}_{k}, \bar x_1]] e^Y\\
&=\sum_{k\geq 0}{1\over (k+1)!} [\underbrace{Y, \cdots, [Y}_{k}, \bar x_1]] e^Y\\
&=\sum_{k\geq 0}{1\over (k+1)!} \sum_{\substack{i_1,\cdots,i_k\in \{2,\cdots, n\}\\i_1,\cdots,i_k \ \text{distinct}}}[\bar x_{i_k},[\bar x_{i_{k-1}},\cdots,[\bar x_{i_1}, \bar x_1]]] \exp(\sum\limits_{\substack{j\in \{2,\cdots, n\}\\ j \notin\{i_1,\cdots, i_k\}}}\bar x_j)\,. 
\end{align*}
We can view $[\bar x_{i_k},[\bar x_{i_{k-1}},\cdots,[\bar x_{i_1}, \bar x_1]]] $ as coming from an ordered cyclic expression with 
$$
(1 i_1\cdots i_k) (\cdots)\cdots(\cdots)\,. 
$$
Then the equality 
$$
{1\over n!} \sum_{\sigma\in S_n} \bar x_{\sigma(1)}\bar x_{\sigma(2)}\cdots \bar x_{\sigma(n)}=\sum_{\sigma\in S_n} {\bar x_\sigma\over |\sigma|}
$$
follows from the above expression and the induction applied to $\exp(\sum\limits_{\substack{j\in \{2,\cdots, n\}\\ j \notin\{i_1,\cdots, i_k\}}}\bar x_j)$. 
\end{proof}

\section{Examples on evaluation of  integrals}
 \label{secstraightforwardevaluation}

In this part, as a double check we offer an alternative direct computation of
the  Feynman graph integral in Example
\ref{exFeymangraphintegral2}
\begin{eqnarray*}
\widehat{I}_{\Gamma}&=&\dashint_{E_{\tau}^3} \bracket{\prod_{i=1}^3 {d^2z_i\over \im \tau}} \Phi_{\Gamma}
\,,\quad
\Phi_{\Gamma}=\widehat{P}(z_1,z_2) \widehat{P}(z_2,z_3)
\widehat{P}(z_3,z_1)
\,.
\end{eqnarray*}
We apply  Proposition \ref{prop-bilinear} 
successively for the evaluation of the iterated integral.
The details are given as follows.

The first integration on $z_3$ gives
\begin{eqnarray*}
&&\dashint_{E_{\tau}} {d^2z_3\over \im \tau} \Phi_{\Gamma}\\
&=&\int_{A_{3}}\Phi_{\Gamma}dz_{3}- {-\pi\over \im \tau} 
\cdot \Res_{z_{3}=z_{2}} (z_{3}\Phi_{\Gamma}dz_{3})
- {-\pi\over \im \tau} 
\cdot \Res_{z_{3}=z_{1}} (z_{3}\Phi_{\Gamma}dz_{3})\\
&&+
{-\pi\over \im \tau}
\cdot \Res_{z_{3}=z_{2}} ( \Phi_{\Gamma}dz_{3})\cdot \bar{z}_{3}|^{z_{2}}_{p_0}
+
{-\pi\over \im \tau}
\cdot \Res_{z_{3}=z_{1}} ( \Phi_{\Gamma}dz_{3}) \cdot \bar{z}_{3}|^{z_{1}}_{p_0}
\,\\
&=&\int_{A_{3}}\Phi_{\Gamma}dz_{3}+ {-\pi\over \im \tau} 
\cdot (-2)
\widehat{P}^{2}(z_{2}-z_{1})+
{-\pi\over \im \tau}
\cdot \widehat{P}(z_{2}-z_{1}) \widehat{P}'(z_{2}-z_{1})
(\bar{z}_{2}-z_2-\bar{z}_{1}+z_1)\,.
\end{eqnarray*}

The previous computations 
in Example
\ref{exFeymangraphintegral2}
show that
$$
 \int_{A_{3}} \Phi_{\Gamma}dz_3\,
=\widehat{P}(z_{2}-z_{1})
\left( (2\pi i)^{4} 
\sum_{k\neq 0} {k^{2}q^{k}\over (1-q^{k})^2}({u_{1}\over u_{2}})^{k}
+ ({-\pi\over \im \tau})^2
\right)\,.$$
This is not elliptic anymore but only quasi-elliptic.
However, in the integration domain for
the iterated regularized integral we can compute directly
$$\Res_{z_{2}=z_{1}} ( dz_{2}\int_{A_{3}}\Phi_{\Gamma} dz_3 )  =0\,.$$
Similarly, by using the following identity in Remark \ref{remcomputationontatecurve}
\begin{equation*}
(2\pi i)^4 \cdot 2 \sum_{k\geq 1}{k^{2}q^{k}\over (1-q^{k})^2}={1\over 9}\pi^{4} (E_{4}-E_{2}^2)\,,
\end{equation*}
we have
\begin{eqnarray*}
\Res_{z_{2}=z_{1}} (z_{2} dz_{2}\int_{A_{3}} \Phi_{\Gamma} dz_3 )  
&=&    (2\pi i)^{4} \sum_{k\neq 0} {k^{2}q^{k}\over (1-q^{k})^2}+ ({-\pi\over \im \tau})^2 \\
&=&
 {1\over 9}\pi^{4} (E_{4}-E_{2}^2)
+ ({-\pi\over \im \tau})^2 
\,.
\end{eqnarray*}
It follows that
\begin{eqnarray*}
&&\dashint_{E_{\tau}} {d^2z_2\over \im \tau}
  \int_{A_{3}} \Phi_{\Gamma}dz_3\\
&=&\int_{A_{2}}dz_{2}
\int_{A_{3}}\Phi_{\Gamma}dz_{3}-
{-\pi\over \im \tau}
\cdot 
\Res_{z_{2}=z_{1}}
\left( z_{2}dz_{2}\left(\int_{A_{3}}\Phi_{\Gamma}dz_{3}\right) \right)
\\
&&+
{-\pi\over \im \tau}
\Res_{z_{2}=z_{1}}\left(dz_{2}\int_{A_{3}}\Phi_{\Gamma}dz_{3}\right)\cdot  \bar{z}_{2}|^{z_{1}}_{p_0}\\
&=&
\int_{A_{2}}dz_{2}
\int_{A_{3}}\Phi_{\Gamma}dz_{3}
- ({-\pi\over \im \tau}) {1\over 9}\pi^{4} (E_{4}-E_{2}^2)
- ({-\pi\over \im \tau})^3\,\\
&=&
{1\over 12^3} (2\pi i)^6 (-E_{2}^3+3E_{2}E_{4}-2E_{6})- ({-\pi\over \im \tau}) {1\over 9}\pi^{4} (E_{4}-E_{2}^2)\,.
\end{eqnarray*}
We also have
\begin{equation*}
\dashint_{E_{\tau}} {d^2z_2\over \im \tau}
 {-\pi\over \im \tau} 
\cdot (-2)
\widehat{P}^{2}(z_{2}-z_{1})
= (-2) {-\pi\over \im \tau} \cdot
 \dashint_{E_{\tau}} {d^2 z_2\over \im \tau}
\widehat{P}^{2}(z_{2}-z_{1})
= (-2) {-\pi\over \im \tau} \cdot
{1\over 9}\pi^{4} (E_{4}-\widehat{E}_{2}^2)\,.
\end{equation*}
For the last term, using the translation invariance of regularized integrals
we see that
\begin{eqnarray*}
&&\dashint_{E_{\tau}} {d^2z_2\over \im \tau}
 \left(
 {-\pi\over \im \tau}
\cdot \widehat{P}(z_{2}-z_{1}) \widehat{P}'(z_{2}-z_{1})
(\bar{z}_{2}-z_2-\bar{z}_{1}+z_1)
\right)\\
&=&
{1\over 2\pi i}
\cdot 
( {-\pi\over \im \tau} )^{2}\dashint_{E_{\tau}}
\widehat{P}(z)\widehat{P}'(z)( \bar{z}-z)dz \wedge d\bar{z}
\,.
\end{eqnarray*}
Applying Proposition \ref{prop-bilinear} and \eqref{eqn-wponC}, one has
\begin{eqnarray*}
&&{1\over 2\pi i}
\cdot 
( {-\pi\over \im \tau} )^{2}\dashint_{E_{\tau}}
\widehat{P}(z)\widehat{P}'(z)( \bar{z}-z)dz \wedge d\bar{z}\\
&=&{1\over 2\pi i}
\cdot 
( {-\pi\over \im \tau} )^{2}
\left(
{1\over 2}(\bar{\tau}-\tau)^{2}\int_{A} d({1\over 2}\widehat{P}^2(z))-2\pi i \,\mathrm{Res}_{z=0} (\widehat{P}(z)\widehat{P}'(z)\cdot {1\over 2}(\bar{z}-z)^2)
\right)\\
&=&- ( {-\pi\over \im \tau} )^{2}
 {\pi^{2}\over 3}\widehat{E}_{2}\,.
\end{eqnarray*}

Putting all these together, we obtain the desired result in Example
\ref{exFeymangraphintegral2}
$$\widehat{I}_{\Gamma}= 
\dashint_{E_{\tau}^3} \bracket{\prod_{i=1}^3 {d^2z_i\over \im \tau}} \Phi_{\Gamma}
=
\dashint_{E_{\tau}^2} \bracket{\prod_{i=1}^2 {d^2z_i\over \im \tau}} \Phi_{\Gamma}
={1\over 12^3} (2\pi i)^6 (-\widehat{E}_{2}^3+3\widehat{E}_{2}E_{4}-2E_{6})\,.$$

\bibliographystyle{amsalpha}

\begin{thebibliography}{BBBM17}

\bibitem[Arn69]{Arnol1969The}
V.~I. Arnold, \emph{The cohomology ring of the colored braid group},
  Mathematical Notes of the Academy of Ences of the Ussr \textbf{5} (1969),
  no.~2, 138--140.

\bibitem[AS93]{axelrod1993chern}
S.~Axelrod and I.~M. Singer, \emph{Chern--{S}imons {P}erturbation {T}heory
  {II}}, Journal of Differential Geometry \textbf{39} (1993),
  no.~hep-th/9304087, 173--213.

\bibitem[BBBM17]{Boehm:2014tropical}
J.~ B{\"o}hm, K.~Bringmann, A.~Buchholz, and H.~Markwig,
  \emph{Tropical mirror symmetry for elliptic curves}, Journal f{\"u}r die
  reine und angewandte Mathematik (Crelles Journal), \textbf{732} (2017), 211--246.
    \emph{\ \ \ \ Erratum to Tropical mirror symmetry for elliptic curves (J. reine angew. Math. 732 (2017), 211--246)}, Journal f{\"u}r die
  reine und angewandte Mathematik (Crelles Journal), \textbf{760} (2020), 163--164.

\bibitem[BCOV94]{Bershadsky:1993cx}
M.~Bershadsky, S.~Cecotti, H.~Ooguri, and C.~Vafa, \emph{{Kodaira-Spencer
  theory of gravity and exact results for quantum string amplitudes}},
  Communications in Mathematical Physics \textbf{165} (1994), 311--428.
  
\bibitem[BD21]{Brown:2021}
F.~Brown, and C.~Dupont,  \emph{Single-valued integration and double copy}, Journal f{\"u}r die reine und angewandte Mathematik (Crelles Journal),
\textbf{775} (2021), 145--196.  



\bibitem[BEK06]{Bloch:2006motives}
S.~Bloch, H.~Esnault, and D.~Kreimer, \emph{On motives
  associated to graph polynomials}, Communications in Mathematical Physics
  \textbf{267} (2006), no.~1, 181--225.

\bibitem[BK08]{Bloch:2008mixed}
S.~Bloch and D.~Kreimer, \emph{Mixed hodge structures and
  renormalization in physics}, Communications in Number Theory and Physics
  \textbf{2} (2008), no.~4, 637--718.

\bibitem[Blo07]{Bloch:2007motives}
S.~Bloch, \emph{Motives associated to graphs}, Japanese Journal of
  Mathematics \textbf{2} (2007), no.~1, 165--196.

\bibitem[Blo08]{Bloch:2008motives}
\bysame, \emph{Motives associated to sums of graphs}, 
The Geometry of Algebraic Cycles--Proceedings of the Conference on Algebraic Cycles, Columbus, Ohio, 2008. 
Clay Mathematics Proceedings Volume \textbf{9}, 2010, 137--143.

\bibitem[{Blo}15]{Bloch:2015}
\bysame, \emph{{Feynman Amplitudes in Mathematics and Physics}},  	arXiv:1509.00361 [math.AG].


\bibitem[BO00]{Bloch:2000}
S.~{Bloch} and A.~{Okounkov}, \emph{{The Character of the Infinite
  Wedge Representation}}, Advances in Mathematics \textbf{149} (2000), 1--60.

\bibitem[BSV20]{Benini:2020}
M.~Benini, A. ~Schenkel and B.~Vicedo, \emph{Homotopical analysis of 4d Chern-Simons theory and integrable field theories}, arXiv:2008.01829 [hep-th].




\bibitem[CK00]{Connes:2000renormalization}
A.~Connes and D.~Kreimer, \emph{Renormalization in quantum field theory
  and the Riemann--Hilbert problem i: The Hopf algebra structure of graphs and
  the main theorem}, Communications in Mathematical Physics \textbf{210}
  (2000), no.~1, 249--273.

\bibitem[C11]{Costello:2011book}
K.~J. Costello, \emph{{Renormalization and effective field theory}}, Mathematical Surveys and Monographs, Volume 170,
American Mathematical Society (2011).

\bibitem[CL12]{Costello:2012cy}
K.~J. Costello and S.~Li, \emph{{Quantum BCOV theory on Calabi-Yau manifolds
  and the higher genus B-model}}, arXiv:1201.4501[math.QA].

\bibitem[CMSZ20]{Chen:2020masur}
D.~Chen, M.~M{\"o}ller, A.~Sauvaget, and D.~Zagier,
  \emph{Masur--Veech volumes and intersection theory on moduli spaces of
  abelian differentials}, Inventiones mathematicae (2020).

\bibitem[CMZ18]{Chen:2018quasimodularity}
D.~Chen, M.~M{\"o}ller, and D.~Zagier, \emph{Quasimodularity and large
  genus limits of siegel-veech constants}, Journal of the American Mathematical
  Society \textbf{31} (2018), no.~4, 1059--1163.

\bibitem[Dem12]{Demailly:2012complex}
J.~P.~ Demailly, \emph{Complex analytic and differential geometry}. 2007. 

\bibitem[Dij95]{Dijkgraaf:1995}
R.~Dijkgraaf, \emph{Mirror symmetry and elliptic curves}, The moduli space
  of curves ({T}exel {I}sland, 1994), Progr. Math., vol. 129, Birkh\"auser
  Boston, Boston, MA, 1995, pp.~149--163.

\bibitem[Dij97]{Dijkgraaf:1997chiral}
\bysame, \emph{Chiral deformations of conformal field theories}, Nuclear
  physics B \textbf{493} (1997), no.~3, 588--612.


\bibitem[Dou95]{Douglas:1995conformal}
M.~Douglas, \emph{Conformal Field Theory Techniques in Large N Yang-Mills Theory}, In: Baulieu L., Dotsenko V., Kazakov V., Windey P. (eds) Quantum Field Theory and String Theory. NATO ASI Series (Series B: Physics), vol \textbf{328}. Springer, Boston, MA, 1995.


\bibitem[DMZ12]{Dabholkar:2012}
A.~{Dabholkar}, S.~{Murthy}, and D.~{Zagier}, \emph{{Quantum Black Holes, Wall
  Crossing, and Mock Modular Forms}}, arXiv:1208.4074 [hep-th].

\bibitem[EO01]{Eskin:2001asymptotics}
A.~Eskin and A.~Okounkov, \emph{Asymptotics of numbers of branched
  coverings of a torus and volumes of moduli spaces of holomorphic
  differentials}, Inventiones mathematicae \textbf{145} (2001), no.~1, 59--103.

\bibitem[EZ85]{Eichler:1985}
M.~Eicher and D.~Zagier, \emph{The theory of Jacobi forms}, Progress in
  Math. \textbf{55}, Birkh\"auser Boston, Boston, MA, 1985.

\bibitem[GH14]{Griffiths:2014principles}
P.~Griffiths and J.~Harris, \emph{Principles of algebraic geometry},
  John Wiley \& Sons, 2014.

\bibitem[GJ94]{getzler1994operads}
E.~Getzler and J.~D.~S.~ Jones, \emph{Operads, homotopy algebra and iterated
  integrals for double loop spaces}, arXiv:hep-th/9403055.

\bibitem[GM20]{Goujard:2016counting}
E.~Goujard and M.~M{\"o}ller, \emph{Counting Feynman-like graphs:
  Quasimodularity and Siegel-Veech weight}, Journal of the European Mathematical Society, 22 (2020), No 2,  pp. 365–412.

\bibitem[Kat76]{Katz:1976p}
N.~M.~Katz, \emph{p-adic interpolation of real analytic Eisenstein
  series}, Annals of Mathematics (1976), 459--571.

\bibitem[Kon94]{kontsevich1994feynman}
M.~Kontsevich, \emph{Feynman diagrams and low-dimensional topology}, First
  European Congress of Mathematics Paris, July 6--10, 1992, Springer, 1994,
  pp.~97--121.

\bibitem[Kon03]{kontsevich2003deformation}
\bysame, \emph{Deformation quantization of Poisson manifolds}, Letters in
  Mathematical Physics \textbf{66} (2003), no.~3, 157--216.



\bibitem[Kre00]{Kreimer:2000}
D.~Kreimer, \emph{Chen's Iterated Integral represents the Operator Product Expansion}, 
Adv. Theor. Math. Phys. \textbf{3} (2000) 627--670.





\bibitem[KZ95]{Kaneko:1995}
M.~Kaneko and D.~Zagier, \emph{A generalized {J}acobi theta function and
  quasimodular forms}, The moduli space of curves ({T}exel {I}sland, 1994),
  Progr. Math., vol. 129, Birkh\"auser Boston, Boston, MA, 1995, pp.~165--172.



\bibitem[Lan85]{Lang:1985}
S.~Lang, \emph{Complex Analysis}, Graduate Texts in Mathematics, vol \textbf{103}. Springer, New York, NY, 1985.

\bibitem[Ler59]{Leray:1959}
J.~Leray, \emph{Le calcul differentiel et integral sur une variete analytique complexe (Probleme de Cauchy, III)},
Bull. Soc. Math. France \textbf{87} (1959), 81--180.



\bibitem[Li12]{Li:2011mi}
S.~Li, \emph{{Feynman Graph Integrals and Almost Modular Forms}},
  Communications in Number Theory and Physics, \textbf{6} (2012), 129--157.

\bibitem[Li16]{li2016vertex}
\bysame, \emph{Vertex algebras and quantum master equation}, 
  arXiv:1612.01292[math.QA].

\bibitem[{Lib}09]{Libgober:2009}
A.~{Libgober}, \emph{{Elliptic genera, real algebraic varieties and
  quasi-Jacobi forms}}, arXiv:0904.1026 [math.AG].

\bibitem[Mar09]{Marcolli:2009feynman}
M.~Marcolli, \emph{Feynman integrals and motives}, European Congress of Mathematics, 2009, vol \textbf{313}, No. 1, pp.~293-332.


\bibitem[OP18]{Oberdieck:2018}
G.~Oberdieck and A.~Pixton, \emph{Holomorphic anomaly equations and the Igusa cusp form conjecture}, Invent. math. \textbf{213}, 507--587 (2018). 


  
\bibitem[Rud94]{rudd1994string}
R.~Rudd, \emph{The string partition function for QCD on the torus}, arXiv: hep-th/9407176.

\bibitem[RY09]{Roth:2009mirror}
M.~Roth and N.~Yui, \emph{Mirror symmetry for elliptic curves: the
  B-model (bosonic) counting}, preprint.

\bibitem[RY10]{Roth:2010mirror}
\bysame, \emph{Mirror symmetry for elliptic curves: the {A}-model (fermionic)
  counting}, Clay Math. Proc \textbf{12} (2010), 245--283.
  
  
\bibitem[Siv09]{Silverman:2009}
J.~H.~Silverman, \emph{The Arithmetic of Elliptic Curves}, Graduate Texts in Mathematics, vol \textbf{106}. Springer, New York, NY, 2009.  
  
\bibitem[Sol68]{Solomon:1968}
L.~Solomon, \emph{On the Poincar\'e-Birkhoff-Witt theorem}, Journal of Combinatorial Theory, vol \textbf{4}, issue 6, 1968, 363--375.  
  
  
  
  
  
  
  
  
  \bibitem[Tak01]{Takhtajan:2001}
 L.~A.~Takhtajan, \emph{Free Bosons and Tau-Functions for Compact Riemann Surfaces and Closed Smooth Jordan Curves. Current Correlation Functions},
 Letters in Mathematical Physics \textbf{56} (2001), 181--228.


  
 \bibitem[Tyu78]{Tyurin:1978}
A.~Tyurin, \emph{On periods of quadratic differentials}, Russian mathematical 
surveys \textbf{33} (1978), 169--221.
  
  
  
  
 
\bibitem[Zag91]{Zagier:1991}
D.~Zagier, \emph{Periods of modular forms and Jacobi theta functions},
  Inventiones mathematicae \textbf{104} (1991), no.~1, 449--465.

\bibitem[Zag08]{Zagier:2008}
\bysame, \emph{Elliptic modular forms and their applications}, The 1-2-3 of
  modular forms, Universitext, Springer, Berlin, 2008, pp.~1--103. 

\bibitem[Zag16]{Zagier:2016partitions}
\bysame, \emph{Partitions, quasimodular forms, and the Bloch--Okounkov
  theorem}, The Ramanujan Journal \textbf{41} (2016), no.~1-3, 345--368.

\end{thebibliography}

\providecommand{\bysame}{\leavevmode\hbox to3em{\hrulefill}\thinspace}
\providecommand{\MR}{\relax\ifhmode\unskip\space\fi MR }
\providecommand{\MRhref}[2]{%
  \href{http://www.ams.org/mathscinet-getitem?mr=#1}{#2}
}
\providecommand{\href}[2]{#2}

\bigskip{}

\noindent{\small Yau Mathematical Sciences Center, Tsinghua University, Beijing 100084, P. R. China}

\noindent{\small Email: \tt  sili@mail.tsinghua.edu.cn}

\medskip{}
\noindent{\small Yau Mathematical Sciences Center, Tsinghua University, Beijing 100084, P. R. China}

\noindent{\small Email: \tt jzhou2018@mail.tsinghua.edu.cn}

\end{document}